\documentclass[11pt,a4paper]{article}
\usepackage[marginratio=1:1,height=600pt,width=460pt,tmargin=100pt]{geometry}
\usepackage{epsfig}
\usepackage{graphicx}
\usepackage{amsmath}
\usepackage{amssymb}
\usepackage{graphicx}
\graphicspath{{./pics/}}
\usepackage{subcaption}
\usepackage[usenames, dvipsnames]{color}
\usepackage{amsthm}
\usepackage{amsfonts}
\usepackage{amscd}
\usepackage{mathtools}
\usepackage{mathabx}
\usepackage{mathrsfs}
\usepackage{bm}
\usepackage{enumerate}
\usepackage{algorithmic, algorithm}
\usepackage{multirow}
\usepackage{authblk}
\usepackage{array}
\usepackage{url}

\usepackage{sidecap}
\sidecaptionvpos{figure}{c}

\newcolumntype{C}[1]{>{\centering\let\newline\\\arraybackslash\hspace{0pt}}m{#1}}

\usepackage{sidecap}
\sidecaptionvpos{figure}{c}

\newcommand{\R}{\mathbb{R}}
\newcommand{\RR}{\mathcal{R}}
\newcommand{\Z}{\mathbb{Z}}
\newcommand{\CS}{\mathcal{S}}
\newcommand{\M}{\mathcal{M}}
\newcommand{\N}{\mathcal{N}}
\newcommand{\A}{\mathcal{A}}
\newcommand{\B}{\mathcal{B}}

\DeclareMathOperator{\Span}{span}

\DeclareMathOperator{\supp}{supp}

\newtheorem{thm}{Theorem}
\newtheorem{lemma}{Lemma}
\newtheorem{prop}{Proposition}
\newtheorem{definition}{Definition}
\newtheorem{remark}{Remark}

\title{Constructing curvelet-like bases and low-redundancy frames}
\author[1]{Wei Zhu\footnote{Email: zhu@math.duke.edu.}}
\author[1]{Ingrid Daubechies}

\affil[1]{Department of Mathematics, Duke University}
\date{}

\begin{document}

\maketitle

\begin{abstract}
  We provide a detailed analysis of the obstruction (studied first by S. Durand and later by R. Yin and one of us) in the construction of multidirectional wavelet orthonormal bases corresponding to any admissible frequency partition in the framework of subband filtering with non-uniform subsampling. To contextualize our analysis, we build, in particular, multidirectional alias-free hexagonal wavelet bases and low-redundancy frames with optimal spatial decay. In addition, we show that a 2D cutting lemma can be used to subdivide the obtained wavelet systems in higher frequency rings so as to generate bases or frames that satisfy the ``parabolic scaling'' law enjoyed by curvelets and shearlets. Numerical experiments on high bit-rate image compression are conducted to illustrate the potential of the proposed systems.
\end{abstract}

\section{Introduction}

Since its original development more than two decades ago, the wavelet transform has become a standard tool in signal and image processing, and it is especially renowned for being an integral part of the JPEG2000 image compression standard \cite{jpeg}. Despite the spatial-frequency localization inherited from 1D wavelets, the widely used 2D tensor wavelet systems suffer from poor orientation selectivity: only horizontal or vertical edges are well represented, which makes them non-optimal for images with discontinuities along regular curves.

Several approaches have been developed in order to achieve more efficient representation of directional structures, the most notable among which are curvelets \cite{candes2006fast,candes2004new} and shearlets \cite{easley2008sparse,guo2006sparse,kutyniok2014shearlab}. Curvelet and shearlet systems both follow a ``parabolic scaling'' law for their time-frequency decomposition, which has been shown to help them achieve optimal asymptotic rate of approximation for ``cartoon'' images, i.e., piecewise smooth functions with jumps occurring possibly along piecewise $C^2$-curves. However, all of the available discrete implementations of curvelets and shearlets suffer from being highly redundant (this is especially true for those implementations,  such as \cite{kutyniok2014shearlab}, that remain faithful to their continuous framework), which prevents them to deliver up to their theoretical potential for reasonable-size problems in practice. Another popular scheme with similar time-frequency decomposition is the contourlet transform \cite{do2005contourlet}, which, unlike curvelets and shearlets, is defined directly on a discrete lattice. Contourlets are still redundant frames (with a redundancy factor 4/3) because they are based on the (naturally redundant) Laplacian pyramid scheme; in addition, they are also not well-localized in the frequency domain \cite{lu2006new}. Lu and Do proposed in \cite{lu2003crisp} a contourlet transform with critical sample rate, but because their construction does not correspond to a multiresolution approximation (MRA) of $L^2(\R^2)$, it
is not clear whether it corresponds to a basis for $L^2(\R^2)$.

Most of the non-redundant transforms rely on a tree of critically sampled directional filter banks \cite{durand2005orthonormal,durand2007,GUO2006202,nguyen2005multiresolution,yin-wavelet-1, yin-wavelet-2}. The authors in \cite{nguyen2005multiresolution} introduced the framework of {\em non-uniform directional filter banks} (nuDFB) leading to an MRA of $L^2(\R^2)$; their filters are obtained by solving a non-convex optimization problem that provides non-unique near-orthogonal or biorthogonal solutions without guarantee of convergence. It was later shown by Durand in \cite{durand2007} that it is impossible in the framework of nuDFB to construct regularized  directional wavelet bases without aliasing. In \cite{yin-wavelet-1, yin-wavelet-2}, Yin and Daubechies made a detailed study  of these obstructions in the case of dyadic quincunx subsampling corresponding to the frequency partition studied in \cite{nguyen2005multiresolution} (see Figure~\ref{fig:rec_p1}), and constructed orthonormal/biorthogonal bases that, within the constraints they impose, give better smoothness.

Building on the idea of \cite{durand2007,yin-wavelet-1}, we provide, in this paper, a detailed analysis and explicit construction of optimal directional wavelet bases and low-redundancy frames for any admissible frequency partition (to be explained in Section~\ref{sec:admissibility}) within the framework of nuDFB. To contextualize the idea, we build, in particular, alias-free hexagonal wavelet bases with optimal continuity in the frequency domain and  low-redundancy (factor 2) frames with arbitrarily fast spatial decay. We also explain the usage of a 2D cutting lemma (similar to the composition of 2-band filters used in \cite{durand2007}) to subdivide the obtained bases and frames in higher frequency rings so as to satisfy the ``parabolic scaling'' law typical for curvelet or shearlet bases. 

The paper is organized as follows. In Section~\ref{sec:subband_mra}, we provide a brief review of subband filtering with non-uniform sampling and multiresolution approximation of $L^2(\R^d)$. Section~\ref{sec:adm_perm} explains the admissibility and permissibility of the frequency domain partition. We study the explicit construction of alias-free hexagonal wavelet bases and low-redundancy frames with optimal spatial decay in Section~\ref{sec:basis_and_frame}. In Section~\ref{sec:cutting}, a 2D cutting lemma is used to subdivide the obtained hexagonal wavelet systems to satisfy the ``parabolic scaling'' law. Numerical experiments on high bit rate image compression using the proposed hexagonal wavelets are shown in Section~\ref{sec:ex}. Finally, we present some
conclusions in Section~\ref{sec:conclusion}.

\section{Subband filtering and multiresolution approximation}
\label{sec:subband_mra}

In this section, we provide a brief review of lattices in $\R^d$, subband filtering with non-uniform sampling, and the multiresolution approximation (MRA) of $L^2(\R^d)$.

\subsection{Lattices and subband filtering}
A discrete subset $\Lambda\subset \R^d$ is called a (full-rank) {\em lattice} in $\R^d$ if there exists a basis $\left(e_k\right)_{1\le k \le d}$ of $\R^d$, such that
\begin{align*}
  \Lambda = [e_1, \cdots, e_d]~\Z^d = 
  \left\{\,\sum_{k=1}^d z_k~e_k \,;\,z_1, \ldots,\,z_d \in \Z \right\}\,;
\end{align*}
$\Lambda$ can be viewed as
an additive subgroup of $\R^d$,  and $\Span(\Lambda) = \R^d$. 
Define the determinant of $\Lambda$ as $\det(\Lambda) \coloneqq \sqrt{\det(E^TE)}$, where $E = [e_1, \cdots, e_d]\in\R^{d\times d}$. Given a {\em sublattice} $\Gamma\subset\Lambda$ (i.e. $\Gamma$ is again
a full-rank lattice in $\R^d$, and included in $\Lambda$ 
as a set), the quotient $\Lambda/\Gamma$ is defined as the collection of cosets $\gamma+\Gamma$ in $\Lambda$:
\begin{align*}
  \Lambda/\Gamma \coloneqq \left\{[\gamma] = \gamma+\Gamma; \gamma \in \Lambda \right\}.
\end{align*}
The quotient $\Lambda/\Gamma$ is a finite set with cardinality $|\Lambda/\Gamma| = \det(\Gamma)/\det(\Lambda)$. Depending on the context, we sometimes abuse the notation and use $\gamma$ to denote its coset $\gamma+\Gamma$.
%
%
The {\em reciprocal lattice} of $\Lambda\subset \R^d$ is defined by $\Lambda^* \coloneqq \left\{ \gamma \in \R^d: \left<\eta, \gamma\right> \in 2\pi\Z, ~\forall \eta \in \Lambda \right\}$. Equivalently, let $\left(\widetilde{e}_k\right)_{1\le k\le d}$ be the {\em biorthogonal sequence} of $\left\{e_k\right\}_{k=1}^d$, i.e., $\left<e_i,\widetilde{e}_j\right> = \delta_{i,j}$, then
\begin{align*}
  \Lambda^* = 2\pi[\widehat{e}_1, \cdots, \widetilde{e}_d]~\Z^d.
\end{align*}
A subset $S\subset\R^d$  is called a {\em reciprocal cell} of $\Lambda$ if $\{S+{\gamma}\}_{\gamma\in\Lambda^*}$ is a partition of $\R^d$ (modulo a set of zero Lebesgue measure.) It is called the {\em Voronoi reciprocal cell} if $S$ satisfies in addition that
\begin{align*}
  \xi\in S\implies \|\xi\|\le\|\xi+\gamma\|, \quad \forall \gamma \in \Lambda^*.  
\end{align*}
Note that for any reciprocal cell of $\Lambda$, its Lebesgue measure $|S|$ is given by 
\begin{align*}
    |S|=det\left(\Lambda^*\right)= (2 \pi )^d \left[det(\Lambda) \right]^{-1}\,.
\end{align*}
If $S^*$ is a reciprocal cell for $\Lambda^*$, then a function $f \in L^2(S)$ can be extended ``by periodicity under $\Lambda^*$'' to all of $\R^d$ by setting $f(\xi+\gamma)=f(\xi)$ if $\xi \in S$ and $\gamma \in \Lambda^*$. The original function $f$ and its 
periodic extension can both be represented by a Fourier series,
\begin{align*}
     f(\xi)=\sum_{\lambda \in \Lambda} c_{\lambda}\, 
     e^{i \lambda \xi} \,,
\end{align*}
where the sequence $\left(c_{\lambda} \right)_{\lambda \in \Lambda}$ is square-summable. If one multiplies such a function $f$ with a (bounded) $\Lambda^*$-periodic function $M$, then the outcome $g=Mf$ is again 
$\Lambda^*$-periodic, and its Fourier series is given by
\begin{align*}
    g(\xi)=\sum_{\lambda \in \Lambda} d_{\lambda}\, 
     e^{i \lambda \xi} \,,\, \mbox{ where }
     d_{\lambda}=\sum_{\mu \in \Lambda} m_{\mu} 
     c_{\lambda -\mu}\, \mbox{ and } \left(m_{\lambda}\right)_{\lambda \in \Lambda} 
     \mbox{ is the sequence of Fourier coefficients of } M\,.
\end{align*}
The operation of multiplying a $\Lambda^*$-periodic function with $M$, or equivalently, convolving a 
$\Lambda$-indexed sequence with $\left(m_{\lambda}\right)_{\lambda \in \Lambda} $, is 
typically called {\em filtering}; the function $M$ is then called the {\em transfer function} for this filtering 
operation.\\
Given a collection of sublattices $\left\{\Gamma_k\right\}_{k=1}^K$ of $\Lambda$, 
we define a {\em filter bank} $\left(M_k,\widetilde{M_k}, \Gamma_k\right)_{k=1}^K$ 
by the 
choice of $K$ transfer functions $(M_k,\widetilde{M_k})_{k=1}^K \in L^2(S^*)\simeq L^2(\R^d/\Lambda^*)$.
The corresponding {\em non-uniform subsampled filtering} acts on a sequence 
$x\in l^2(\Lambda)$ 
in two steps: first it is filtered on $K$ subbands with the transfer functions $\overline{M_k}$; then each of
these filter-output sequences is subsampled: for
the $k$-th subband, only the entries of the sequence
with indices in $\Gamma_k$ are retained, 
to obtain $y_k\in l^2(\Gamma_k)$. This is the {\em 
decomposition} or {\em analysis} step. At reconstruction, each $y_k$ is upsampled on $\Lambda$ by padding zeros (i.e. for each k an auxiliary sequence 
$z_k$ is defined for which
$\left(z_k \right)_{\lambda}$ equals 
$\left(y_k \right)_{\lambda}$ if $\lambda \in \Gamma_k$,
and 0 otherwise) and then filtered with the transfer functions $\widetilde{M_k}$. The reconstructed signal $\widetilde{x}$ is defined as the sum of these $K$ upsampled and filtered signals. The filter bank $(M_k,\widetilde{M_k},\Gamma_k)_{k=1}^K$ is said to perform a perfect reconstruction (PR) if $\widetilde{x} = x$. The following theorem by Durand \cite{durand2007} provides a necessary and sufficient condition for a PR filter bank.

\begin{thm}
  \label{thm:PR}
  A filter bank $(M_k, \widetilde{M}_k, \Gamma_k)_{k=1}^K$  performs a perfect reconstruction if and only if
  \begin{align}
    \label{eq:PR}
    M(\xi)^* \widetilde{M}(\xi) = |\Lambda/\Gamma|Id_{|\Lambda/\Gamma|}, \quad a.e.~\xi \in \R^d,
  \end{align}
where $\Gamma$ is a common sublattice of all $\Gamma_k$, $M(\xi)^*$ is the conjugate transpose of $M(\xi)$, and the matrices $M(\xi), \widetilde{M}(\xi)$ are defined as
\begin{align}
  \label{eq:m}
  M(\xi) &= \left( e^{i\left<\eta_k,\xi+\gamma\right>}M_k(\xi+\gamma) \right)_{k\in[K]=\{1,\cdots,K\}, ~\eta_k \in \Gamma_k/\Gamma; ~\gamma\in \Gamma^*/\Lambda^*},\\
  \label{eq:m-tilde}
 \widetilde{M}(\xi) &= \left( e^{i\left<\eta_k,\xi+\gamma\right>}\widetilde{M}_k(\xi+\gamma) \right)_{k\in[K]=\{1,\cdots, K\}, ~\eta_k \in \Gamma_k/\Gamma; ~\gamma\in \Gamma^*/\Lambda^*}.
\end{align}
\end{thm}
 
A filter bank $(M_k, \widetilde{M}_k, \Gamma_k)_{k=1}^K$ is said to be critically sampled if
\begin{align*}
  \sum_{k=1}^K|\Lambda/\Gamma_k|^{-1} = \sum_{k=1}^K|\Gamma_k/\Gamma|/|\Lambda/\Gamma| = 1.
\end{align*}
The matrices $M(\xi)$ and $\widetilde{M}(\xi)$ in Theorem~\ref{thm:PR} are square matrices if $(M_k, \widetilde{M}_k, \Lambda \to \Gamma_k)_{k=1}^K$ is critically sampled. This leads to the following proposition, the proof of which is detailed in Appendix~\ref{apdx:proof_crit_PR}.

\begin{prop}
  \label{prop:crit_PR}
  Suppose $(M_k, \widetilde{M}_k, \Lambda \to \Gamma_k)_{k=1}^K$ is a critically sampled PR filter bank, then
  \begin{align}
    \label{eq:mirror}
    \sum_{t \in \Gamma_k^*/\Lambda^*}\widetilde{M}_k(\xi+t)\overline{M_k}(\xi+t) = |\Lambda/\Gamma_k|, \quad a.e.~\xi \in \R^d, ~\forall k \in [K]
  \end{align}
\end{prop}

\subsection{Multiresolution approximation of $L^2(\R^d)$}

Let $\left\{\Lambda^{(j)}\right\}_{j\in\Z}$ be a sequence of lattices in $\R^d$ with the subsampling map $D\in\R^{d\times d}$, i.e.,
\begin{align}
  \label{eq:def_lambda}
  \cdots \Lambda^{(-2)}\supset \Lambda^{(-1)} \supset \Lambda^{(0)} \supset \Lambda^{(1)} \supset \Lambda^{(2)} \cdots, \quad \Lambda^{(j)}  = D^j\Lambda^{(0)}, ~\forall j\in \Z.
\end{align}
An MRA of $L^2(\R^d)$  is defined similarly as that of $L^2(\R)$ in \cite{mallat1989multiresolution}:
\begin{definition}
\label{def:mra}
  A sequence $\left\{V_j\right\}_{j\in\Z}$ of closed subspaces of $L^2(\R^d)$ is an MRA of $L^2(\R^d)$ associated with $\left\{\Lambda^{(j)} = D^j\Lambda^{(0)}\right\}_{j\in\Z}$ if
  \begin{itemize}
  \item $f\in V_j \iff f(\cdot -n)\in V_j, ~\forall j\in\Z, \forall n\in\Lambda^{(j)}$.
  \item $V_{j+1}\subset V_j, ~\forall j \in \Z$.
  \item $f\in V_0 \iff f(D^{-j}\cdot)\in V_j, ~\forall j\in\Z$.
  \item $\lim_{j\to +\infty}V_j = \bigcap_{j\in\Z}V_j = \left\{0\right\}$.
  \item $\lim_{j\to -\infty}V_j = \overline{\bigcup_{j\in\Z}V_j} = L^2(\R^d)$.
  \item There exists a scaling function $\phi\in L^2(\R^d)$ such that $\left\{\phi(\cdot - n)\right\}_{n\in\Lambda^{(0)}}$ constitutes an orthonormal basis of $V_0$.
  \end{itemize}
\end{definition}

Define the scaled and shifted versions of the scaling function
\begin{align}
  \label{eq:scaling_shift}
  \phi_{j}(x) = D^{-j/2}\phi(D^{-j}x),~~\text{and}~~ \phi_{j,n}(x) = \phi_j(x-n), \quad \forall j\in \Z, ~\forall n\in\Lambda^{(j)}.
\end{align}
The scaling relation in Definition~\ref{def:mra} implies that $\left\{\phi_{j,n}\right\}_{n\in\Lambda^{(j)}}$ is an orthonormal basis of $V_j$, and there exists $h_0\in l^2(\Lambda^{(0)})$ such that
\begin{align*}
  \phi_{1}(x) = |D|^{-1/2}\phi(D^{-1}x) = \sum_{n\in\Lambda^{(0)}}h_0[n]\phi(x-n)\,,
  \mbox{ where } |D|=|\det(D)|\,.
\end{align*}
In the frequency domain, this is equivalent to
\begin{align*}
  \widehat{\phi}(\xi) = |D|^{-1/2}\widehat{\phi}(D^{-T}\xi)M_0(D^{-T}\xi),
\end{align*}
where $M_0 = \widehat{h_0}\in L^2(\R^d/\Lambda^{(0)*})$ is the Fourier transform of $h_0$. Hence we have
\begin{align}
  \label{eq:def_phi}
  \widehat{\phi}(\xi) = \prod_{p=1}^\infty |D|^{-1/2}M_0\left( (D^{-T})^{p}\xi \right).
\end{align}

Let $W_j$ be the orthogonal complement of $V_j$ in $V_{j-1}$, i.e., $V_{j-1} = W_j ~\obot V_j$. In order to obtain an orthonormal basis of $L^2(\R^d)$, one needs to find orthonormal bases of $W_j$; because of the inclusion
relation among the $V_j$, and the complementarity, within $V_{j-1}$, of each 
$W_j$ to $V_j$, the union of the $W_j$-bases would indeed provide an orthonormal basis of $L^2(\R^d)$. Moreover, because the different layers of
the MRA are obtained by simple scaling, we really need to
find only an orthonormal basis 
$\left(\psi^k\right)_{k \in [K]}$
for $W_0$, where (as before), we have used the notation
$[K]:=\{1,2,\ldots,K \}$. 
If we define the $K$ sublattices
 $\left\{\Gamma_k^{(0)} \right\}_{k\in [K]}$  of $\Lambda^{(-1)}$, and set
\begin{align}
  \label{eq:def_gamma}
  \Gamma_k^{(j)} \coloneqq D^j\Gamma_k^{(0)}\subset \Lambda^{(j-1)}\,,
\end{align}
then we just need to find $K$ ``mother wavelet functions'' $\psi^k\in L^2(\R^d), k\in [K]$, such that
\begin{align}
  \label{eq:w_onb}
  W_j = \bigobot_{k=1}^K W_j^k =\bigobot_{k=1}^K  \overline{\Span}\left\{ \psi_{j,n}^k \right\}_{n\in \Gamma_k^{(j)}},
\end{align}
where
\begin{align}
  \label{eq:def_psi}
  \left\{
  \begin{aligned}
    &\widehat{\psi^k}(\xi) = |D|^{-1/2}\hat{\phi}(D^{-T}\xi)M_k(D^{-T}\xi),\\
    &\psi_{j}^k(x) = D^{-j/2}\psi^k(D^{-j}x),\quad \quad &&\forall j\in \Z,\\
    &\psi_{j,n}^k(x) = \psi_j^k(x-n), \quad \quad &&\forall n\in\Gamma_k^{(j)},
  \end{aligned}\right.
\end{align}
and $\left\{\psi_{j,n}^k\right\}_{n\in\Gamma_k^{(j)}}$ is an orthonormal basis of $W_j^k$. Figure~\ref{fig:ladder} provides a visual illustration of the inclusion relations among the lattices $\left\{ \Gamma_k^{(j)} \right\}_{j\in\Z,~ 0\le k\le K}$, where we identify $\Lambda^{(j)}$ as $\Gamma_0^{(j)}$.

\begin{figure}
  \centering
  \includegraphics[width=.83\textwidth]{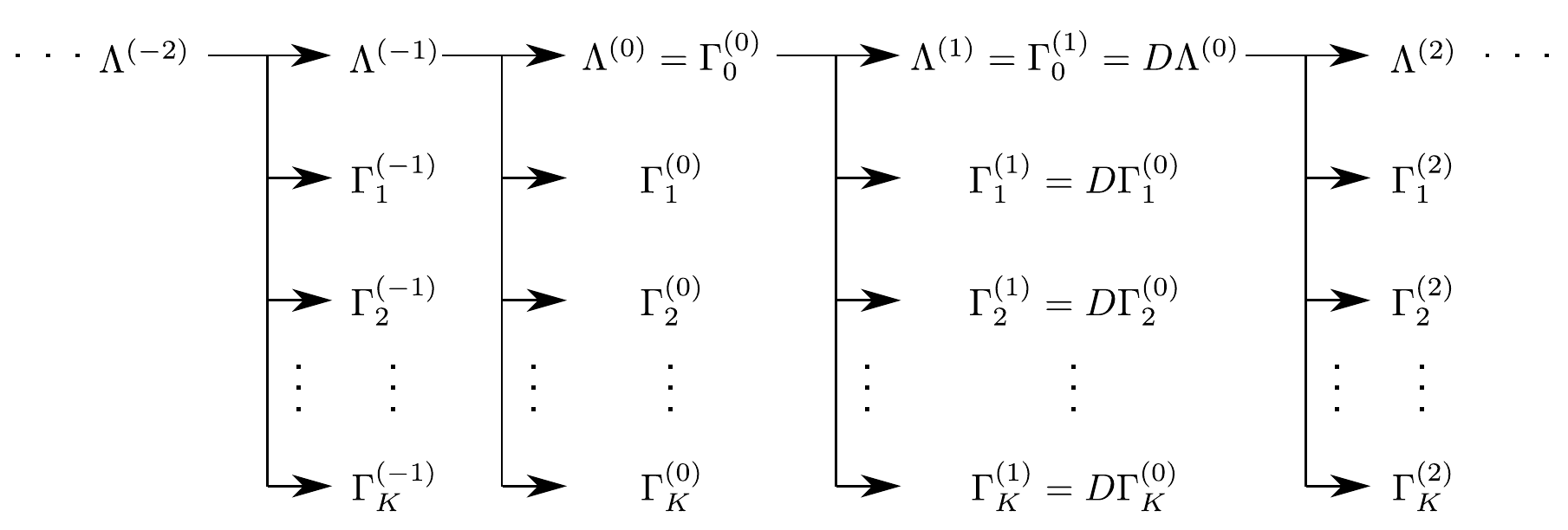}
  \caption{ \small The inclusion relations among the lattices $\left\{ \Gamma_k^{(j)} \right\}_{j\in\Z, ~ 0 \le k\le K}$, where $\Lambda^{(j)}$ is identified with $\Gamma_0^{(j)}$.}
  \label{fig:ladder}
\end{figure}

The following theorem relates the MRA of $L^2(\R^d)$ to the 1-level PR condition for the filter bank $(M_k, M_k, \Gamma_k^{(1)})_{k=0}^K$; this results is a $d$-dimensional generalization of the classical 1D results in \cite{cohen1992biorthogonal, daubechies1992ten, mallat1989multiresolution}.

\begin{thm}
  \label{thm:MRA-filter}
  Suppose $\left\{\Lambda^{(j)} = \Gamma_0^{(j)}\right\}_{j\in\Z}$ and $\left\{\Gamma_k^{(j)}\right\}_{j\in\Z, ~k\in [K]}$ are a collection of lattices as defined in \eqref{eq:def_lambda} and \eqref{eq:def_gamma}, and $(M_k, M_k, \Gamma_k^{(1)})_{k=0}^K$ is a PR filter bank with $M_0(0)= |D|^{1/2}$. Define $\phi$ and $\left\{\psi^k\right\}_{k=1}^K$ in the Fourier domain as in \eqref{eq:def_phi} and \eqref{eq:def_psi}. Then $\left\{\psi_{j,n}^k \right\}_{k\in [K],~j\in\Z, ~n\in \Gamma_k^{(j)}}$ constitutes a Parseval frame of $L^2(\R^d)$. In particular, for any $f\in L^2(\R^d)$, we have
  \begin{align}
    \label{eq:parseval}
    f = \sum_{k=1}^K\sum_{j\in\Z}\sum_{n\in\Gamma_k^{(j)}}\left<f,\psi^k_{j,n}\right>\psi^k_{j,n}.
  \end{align}
Moreover, if $(M_k, M_k, \Lambda^{(0)} \to \Gamma_k^{(1)})_{k=0}^K$ is critically sampled (and thus \eqref{eq:mirror} holds true for $k=0$  by Proposition~{\rm{\ref{prop:crit_PR}}}), and if there exists a compact reciprocal cell $S$ of $\Lambda^{(0)}$ containing a neighborhood of the origin that satisfies
\begin{align*}
  \inf_{p>0, \xi\in S}\left|M_0((D^{-T})^p\xi)\right| > 0,
\end{align*}
then $\left\{\psi_{j,n}^k \right\}_{k\in [K],~j\in\Z, ~n\in \Gamma_k^{(j)}}$ is an orthonormal basis of $L^2(\R^d)$.
\end{thm}

\section{Admissible and permissible frequency domain partition}
\label{sec:adm_perm}

In this section, we first review the multidirectional Shannon wavelets with ideal localization in the Fourier domain. These wavelets are constructed using PR filter banks whose transfer functions are indicator functions forming an admissible partition of the frequency domain. We then discuss the permissibility condition, i.e., the existence of continuous, alias-free, and critically sampled PR filter banks supported mainly on a given admissible partition.

\subsection{Admissibility}
\label{sec:admissibility}

Given a lattice $\Lambda\subset \R^d$, let $S$ be the Voronoi reciprocal cell of $\Lambda$, and let $\left\{A_k\right\}_{k=0}^K$ be a partition of $S$.
\begin{definition}
\label{def:admissible}
  {\normalfont \cite{durand2007}} A partition $\left\{A_k\right\}_{k=0}^K$ is said to be admissible (with respect to $\Lambda$) if there exists a collection of sublattices $\left\{ \Gamma_k\right\}_{k=0}^K$ of $\Lambda$ such that
  \begin{align*}
    \left( \sqrt{|\Lambda/\Gamma_k|}~\chi_{\A_k},\sqrt{|\Lambda/\Gamma_k|}~\chi_{\A_k},\Gamma_k  \right)_{k=0}^K
  \end{align*}
  is a critically sampled PR filter bank, where
  \begin{align}
    \label{eq:mathcal_A_k}
    \A_k = A_k+\Lambda^*, \quad \text{and}~~\chi_{\A_k}~\text{is the indicator function on}~\A_k.
  \end{align}
This can be proved to be equivalent to $A_k$ being a reciprocal cell of $\Gamma_k$, $\forall k$, or
  \begin{align}
    \label{eq:a_k_shift_admissible}
    \left\{\A_k+\gamma\right\}_{\gamma \in \Gamma_k^*/\Lambda^*} \text{ is a partition of } \R^d.
  \end{align}
\end{definition}

Multidirectional Shannon wavelets are obtained from such admissible frequency partitions when the frequency support $\A_0$ of the refinement filter contains a neighborhood of the origin. We hereby present two examples of admissible partitions in 2D which lead to dyadic and hexagonal Shannon wavelets studied in \cite{durand2007, nguyen2005multiresolution,yin-wavelet-1, yin-wavelet-2}.

\vspace{.5em}
\noindent\textbf{Dyadic wavelets}. Let $S = [-\pi,\pi]^2$ be the reciprocal cell of $\Lambda = \Z^2$. We set $A_0$ and $\Gamma_0$  to  $A_0=[-\pi/2,\pi/2]^2$ and $\Gamma_0 = 2\Z^2$ corresponding to the classical dyadic refinement filter  in the 2D tensor wavelets. The frequency ring $S\setminus A_0$ is equally partitioned into $K = 6p~(p\ge 1)$ fan-shaped regions symmetric with respect to the origin. More specifically, the first $3p$ subbands $(A_k^p)_{1\le k\le 3p}$ subdivide the ``horizontal'' fan with the polar angle $\theta \in [-\pi/4, \pi/4]\cup [3\pi/4, 5\pi/4]$, and the remaining $(A_k^p)_{3p+1\le k\le 6p}$ subdivide the ``vertical'' fan with the polar angle $\theta \in [\pi/4, 3\pi/4]\cup [5\pi/4, 7\pi/4]$. The corresponding sublattices $\Gamma_k^p$ are
\begin{align}
\label{eq:rec_lattice_hp}
  \Gamma_k^p = \left\{
  \begin{aligned}
    &
    \begin{bmatrix}
      2p & 0\\
      2p & 4
    \end{bmatrix}\Z^2, \quad 1\le k \le 3p.\\
    &
    \begin{bmatrix}
      4 & 2p\\
      0 & 2p
    \end{bmatrix}\Z^2, \quad 3p+1\le k\le 6p.\\
  \end{aligned}\right.
\end{align}
 Two special cases of such frequency partitions corresponding to $p=1$ (6 directions) and $p=2$ (12 directions) are shown in  Figures~\ref{fig:rec_p1} and \ref{fig:rec_p2}.

\begin{figure}
    \centering
    \begin{subfigure}[t]{0.22\textwidth}
        \includegraphics[width=\textwidth]{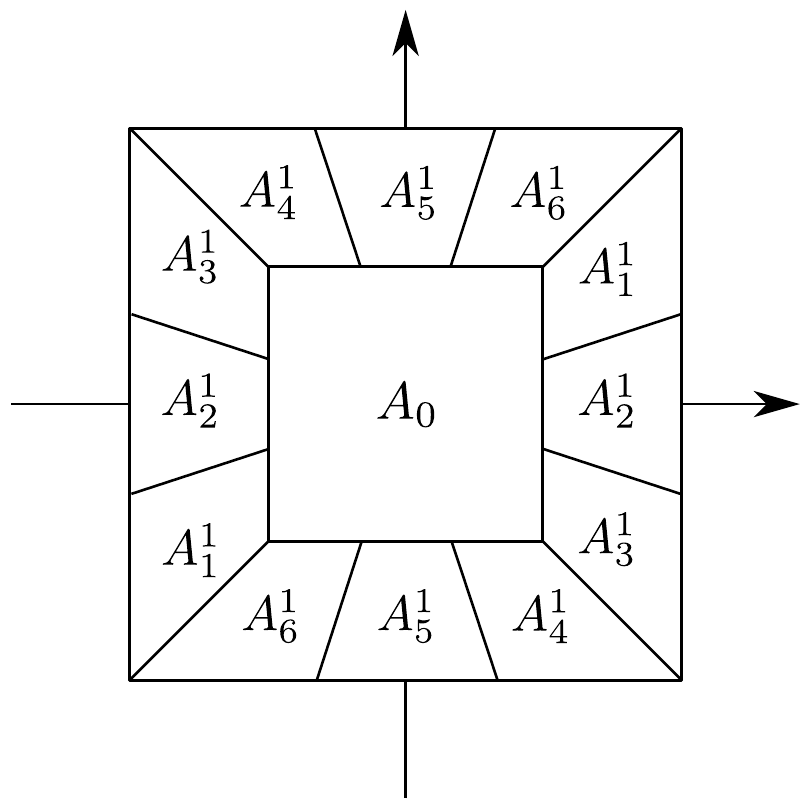}
        \caption{\small Dyadic wavelets with $6p = 6$ directions.}
        \label{fig:rec_p1}
    \end{subfigure}
     ~
    \begin{subfigure}[t]{0.22\textwidth}
        \includegraphics[width=\textwidth]{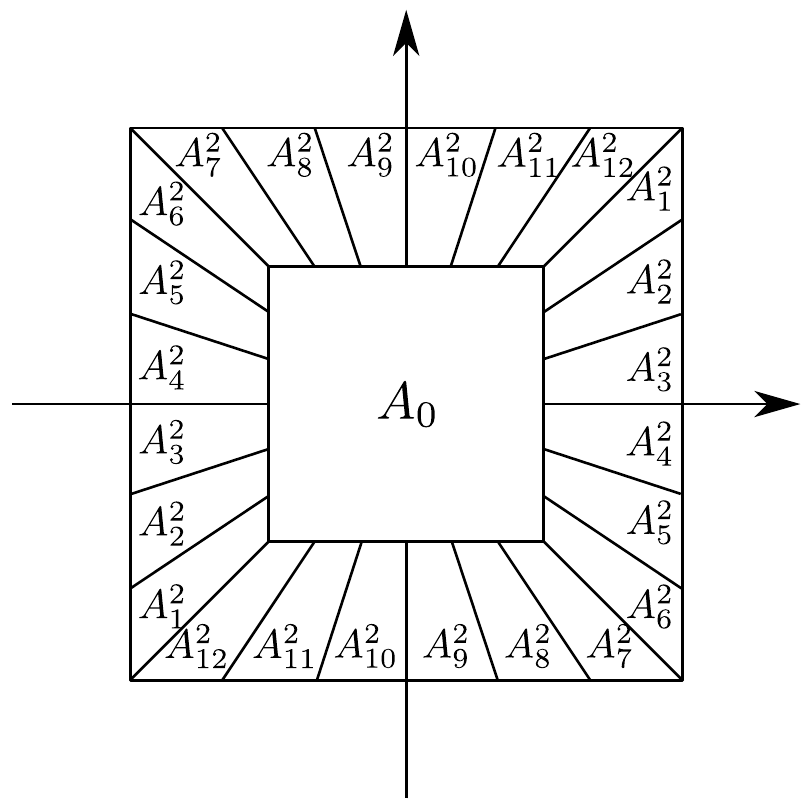}
        \caption{\small Dyadic wavelets with $6p = 12$ directions.}
        \label{fig:rec_p2}
    \end{subfigure}
     ~
    \begin{subfigure}[t]{0.22\textwidth}
        \includegraphics[width=\textwidth]{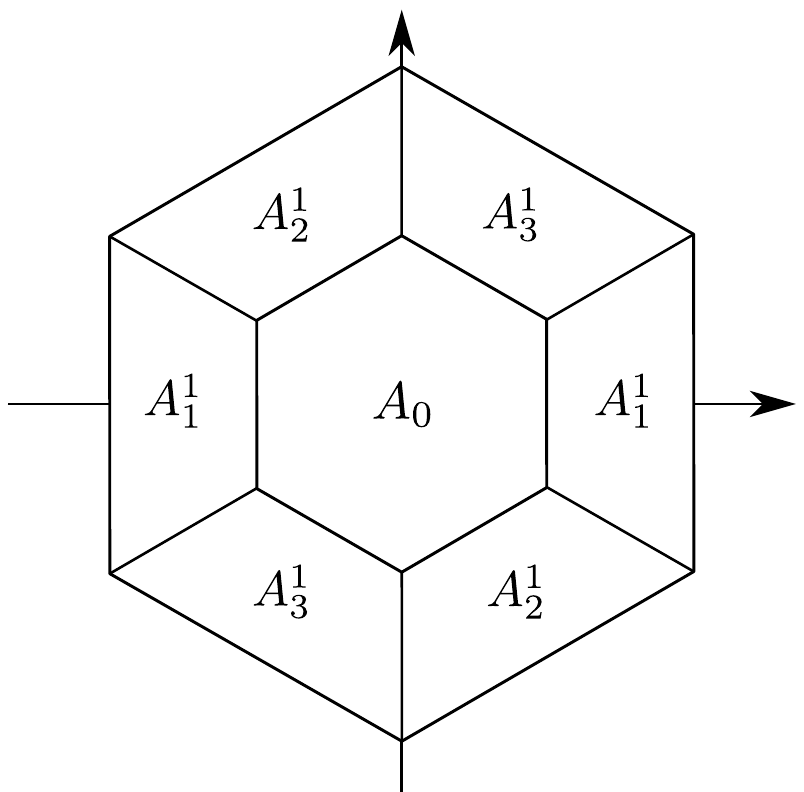}
        \caption{\small Hexagonal wavelets with $3p = 3$ directions.}
        \label{fig:hex_p1}
    \end{subfigure}
     ~
    \begin{subfigure}[t]{0.22\textwidth}
        \includegraphics[width=\textwidth]{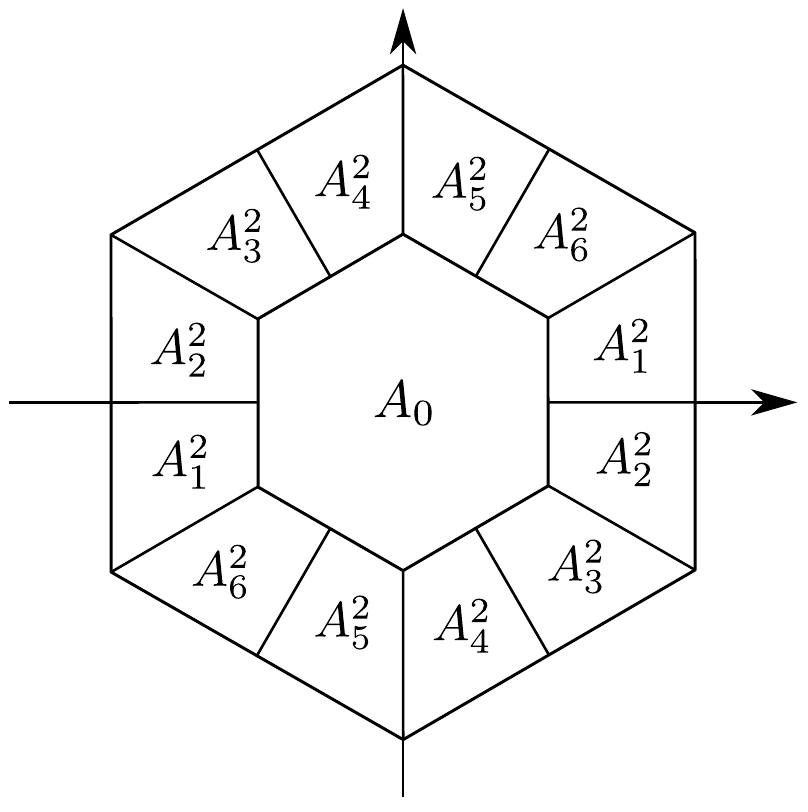}
        \caption{\small Hexagonal wavelets with $3p = 6$ directions.}
        \label{fig:hex_p2}
    \end{subfigure}
    \caption{ \small Admissible frequency partitions.}
\end{figure}

\vspace{.5em}
\noindent\textbf{Hexagonal wavelets}. We consider next the hexagonal wavelets, where the original lattice is
\begin{align}
  \label{eq:hex_lattice}
  \Lambda =
  \begin{bmatrix}
    1 & 0\\
    \frac{1}{\sqrt{3}} & \frac{2}{\sqrt{3}}
  \end{bmatrix} \Z^2.
\end{align}
The corresponding reciprocal cell $S$ is a hexagon defined by
\begin{align*}
  S = \left\{\xi\in\R^2: |\xi_1|\le \pi, ~|\xi_1+\sqrt{3}\xi_2|\le 2\pi, ~|\xi_1-\sqrt{3}\xi_2|\le 2\pi  \right\}.
\end{align*}
The refinement filter $A_0$ is chosen again to have the same shape but $1/4$ the size of $S$, i.e., 
\begin{align*}
  A_0 = \frac{1}{2}S, \quad \Gamma_0 = 2\Lambda.
\end{align*}
As before, the frequency ring $S\setminus A_0$ can be further partitioned into $3p$ fan-shaped regions (See Figure~\ref{fig:hex_p1} and Figure~\ref{fig:hex_p2} for the special cases when $p=1$ and $2$.) The first $p$ subbands $\left(A_k^p\right)_{1\le k\le p}$ subdividing the ``horizontal'' fan with the polar angle $\theta\in [-\pi/6, \pi/6]\cup [5\pi/6, 7\pi/5]$ are downsampled on the same sublattice
\begin{align}
  \label{eq:hex_lattice_hp_1}
\Gamma_k^p = 
  \begin{bmatrix}
    2p & 4 \\
    -\frac{2}{\sqrt{3}}p & 0
  \end{bmatrix}\Z^2, \quad 1\le k \le p.
\end{align}
The other two groups of subbands $\left(A_k^p\right)_{p+1\le k\le 2p}$ and $\left(A_k^p\right)_{2p+1\le k\le 3p}$ are downsampled on the sublattices obtained from rotating \eqref{eq:hex_lattice_hp_1} by $\pm \pi/3$. More specifically,
\begin{align}
\label{eq:hex_lattice_hp_2}
  \Gamma_k^p = \left\{
  \begin{aligned}
    &
    \begin{bmatrix}
      0 & -2\\
      \frac{4}{\sqrt{3}}p & 2\sqrt{3}
    \end{bmatrix}\Z^2, \quad p+1\le k \le 2p.\\
    &
    \begin{bmatrix}
      2p & 2\\
      \frac{2}{\sqrt{3}}p & 2\sqrt{3}
    \end{bmatrix}\Z^2, \quad 2p+1 \le k\le 3p.\\
  \end{aligned}\right.
\end{align}

\subsection{Permissibility}

An admissible partition is said to be permissible if there exists a critically sampled continuous PR filter bank with each subband filter $M_k$ supported in an 
$\epsilon$-neighborhood of $\A_k$. It has been shown in \cite{durand2007} that the dyadic wavelets with $6p~(p\ge 1)$ directions and the hexagonal wavelets with $3p~(p\ge 2)$ directions are not permissible. More specifically,

\begin{prop}
  \label{prop:permissibility}
  Let $(A^p_k)_{0\le k\le K}$ be the admissible frequency partition defined in Section~\ref{sec:admissibility}, where $K = 6p ~(p\ge 1)$ for dyadic wavelets and $K = 3p ~(p\ge 2)$ for hexagonal wavelets. Given any $\epsilon>0$ small enough, there does not exist $\left(M_k\right)_{0\le k\le K}\in L^2(\R^2/\Lambda^*)$ satisfying the following four conditions simultaneously
  \begin{enumerate}
  \item The filter bank $(M_k,M_k, \Lambda\to\Gamma_k)_{k=0}^K$ is PR.
  \item The filter bank $(M_k,M_k, \Lambda\to\Gamma_k)_{k=0}^K$ is critically sampled.
  \item Each $M_k$ is continuous.
  \item Each $M_k$ is supported on $\A_k^p+B(0,\epsilon) = A_k^p+B(0,\epsilon)+\Lambda^*$, where $B(0,\epsilon)$ is the Euclidean ball of radius $\epsilon$.
  \end{enumerate}
\end{prop}

The reason why such partitions are not permissible is the existence of ``singular'' boundaries, which we shall 
discuss in more detail in the next section, that are 
incontrovertible obstacles to the continuity of the $M_k$.
In order to build multidirectional wavelet systems, we thus have to sacrifice either the critical sampling condition (so that we obtain redundant frames rather than bases) or the continuity of $M_k$ in certain directions.

\section{Hexagonal wavelets with optimal spatial frequency localization}
\label{sec:basis_and_frame}

We discuss, in this section, the explicit construction of alias-free hexagonal wavelet  bases and low-redundancy frames with optimal spatial  localization. Compared to dyadic wavelets, one  benefit of  hexagonal wavelets is that their refinement filter is more isotropic (it is invariant by rotation of $\pi/3$ instead of $\pi/2$.) Moreover, hexagonal lattices require the least number of samples to represent images whose spectrum is supported on a disc \cite{durand2007,durand2019rotation}.  Proposition~\ref{prop:permissibility} states that there do not exist alias-free hexagonal wavelet bases with continuous Fourier transforms when the number of high frequency directions is six or higher. We henceforth study the optimally achievable hexagonal wavelet systems when at least one constraint in Proposition~\ref{prop:permissibility} is partially violated, i.e., orthonormal bases with only unavoidable discontinuity in the frequency domain, or low-redundancy frames with better spatial frequency localization. The analysis is based on hexagonal wavelets with six directions, although it can be easily generalized for any admissible frequency partition. Without explicit mentioning, we use, in this section, $\Gamma_k$ and $A_k$ to denote $\Gamma_k^{p=2}$ and $A_k^{p=2}$.

\subsection{Alias-free wavelet orthonormal bases with optimal spatial localization }
\label{sec:basis}

\begin{SCfigure}
  \centering
  \includegraphics[width=.5\textwidth]{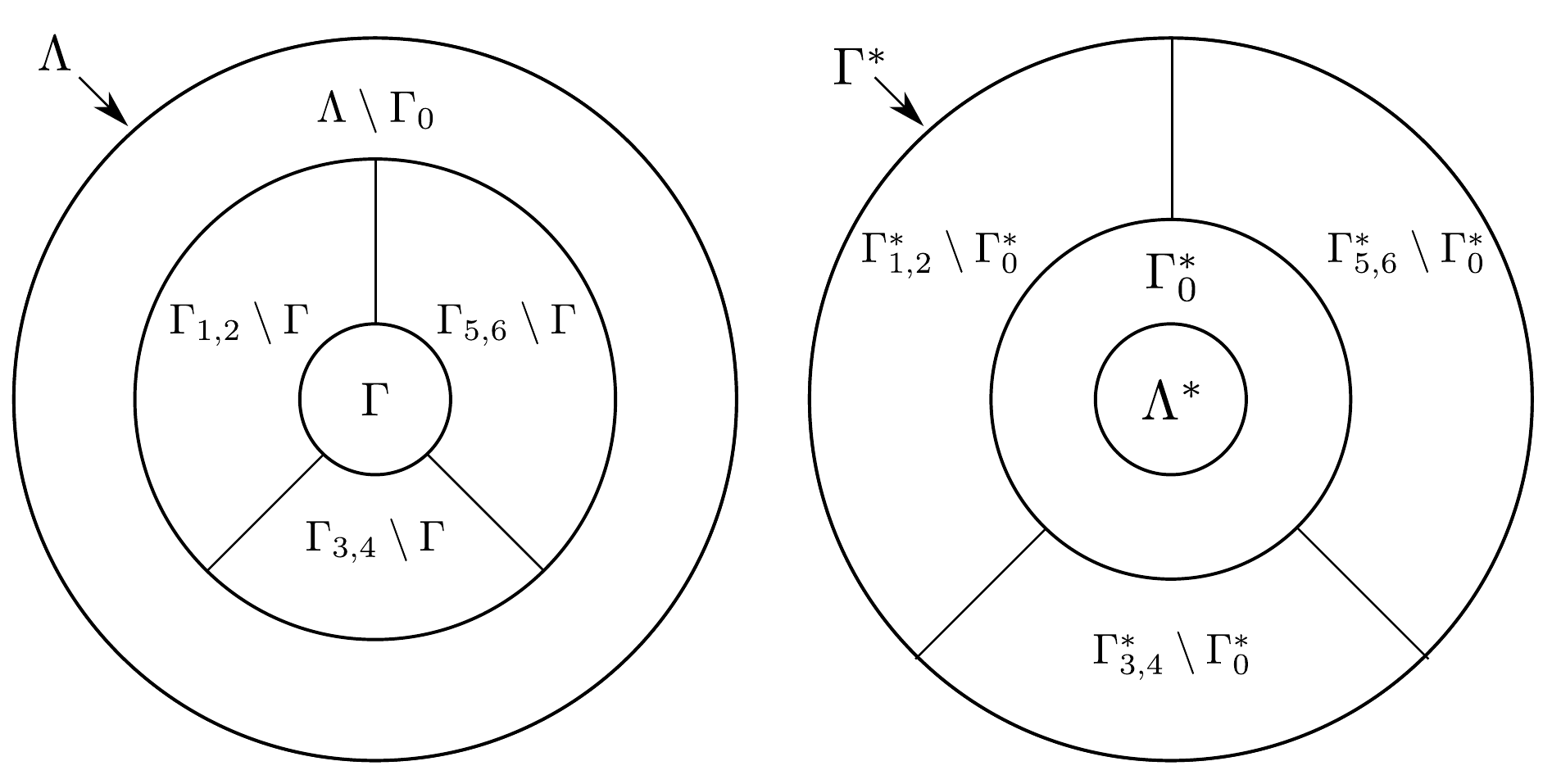}
  \caption{  \small Venn diagrams indicating the inclusion relation among the lattices $\left\{\Lambda,\Gamma_0, \Gamma_{1,2},\Gamma_{3,4},\Gamma_{5,6}, \Gamma\right\}$ and their corresponding reciprocal lattices $\left\{\Lambda^*,\Gamma_0^*, \Gamma_{1,2}^*,\Gamma_{3,4}^*,\Gamma_{5,6}^*, \Gamma^*\right\}$ of the hexagonal wavelets with six directions.
  }
  \label{fig:venn}
\end{SCfigure}

Let $\Lambda, \left\{\Gamma_k\right\}_{0\le k\le 6}$ be defined as in \eqref{eq:hex_lattice} \eqref{eq:hex_lattice_hp_1} \eqref{eq:hex_lattice_hp_2}, and define
\begin{align}
  \Gamma \coloneqq \bigcap_{k=0}^6\Gamma_k, \quad \Gamma_{1,2}\coloneqq \Gamma_1 = \Gamma_2,  \quad  \Gamma_{3,4}\coloneqq \Gamma_3 = \Gamma_4, \quad  \Gamma_{5,6}\coloneqq \Gamma_5 = \Gamma_6.
\end{align}
The inclusion relation among the lattices are illustrated via two Venn diagrams in Figure~\ref{fig:venn}.  Based on Theorem~\ref{thm:PR}, we can simplify the PR condition for this specific choice of lattices as follows
\begin{prop}
  \label{prop:PR_hex}
  The filter bank $(M_k, M_k, \Lambda \to \Gamma_k)_{k=0}^6$ is PR if and only if the following two conditions hold  for $a.e.~\xi\in\R^2$
  \begin{align}
    \label{eq:PR_hex_identity_sum}
    &\quad \quad \quad \quad \sum_{k=0}^6\left|m_k(\xi)\right|^2=1,\\
    \label{eq:PR_hex_shift_cancel}
    &\left\{
    \begin{aligned}
      &\sum_{k=0}^6\overline{m_k(\xi+\gamma)}m_k(\xi) = 0, \quad && \forall \gamma \in \Gamma_0^*\setminus \Lambda^*\\
      &\overline{m_1(\xi+\gamma)}m_1(\xi)+\overline{m_2(\xi+\gamma)}m_2(\xi) = 0, \quad && \forall \gamma \in \Gamma_{1,2}^*\setminus \Gamma_0^*\\
      &\overline{m_3(\xi+\gamma)}m_3(\xi)+\overline{m_4(\xi+\gamma)}m_4(\xi) = 0, \quad && \forall \gamma \in \Gamma_{3,4}^*\setminus \Gamma_0^*\\
      &\overline{m_5(\xi+\gamma)}m_5(\xi)+\overline{m_6(\xi+\gamma)}m_6(\xi) = 0, \quad && \forall \gamma \in \Gamma_{5,6}^*\setminus \Gamma_0^*\\
    \end{aligned}\right.
  \end{align}
where $m_k(\xi) = M_k(\xi)/\sqrt{|\Lambda/\Gamma_k|}$. Condition \eqref{eq:PR_hex_shift_cancel} can be written in the following compact form
  \begin{align}
    \label{eq:PR_hex_compact}
    \sum_{k\in I_{\gamma}}\overline{m_k(\xi+\gamma)}m_k(\xi) = 0, \quad \forall \gamma \in \Gamma^*\setminus \Lambda^*, \quad \text{where  } I_{\gamma}=\left\{ k: \gamma\in\Gamma_k^* \right\}.
   \end{align}
\end{prop}

The detailed proof of this proposition can be found in Appendix~\ref{apdx:prop_PR_hex}. Borrowing the terminology in \cite{yin-wavelet-1}, we call \eqref{eq:PR_hex_identity_sum} the \textit{identity summation} condition, and \eqref{eq:PR_hex_shift_cancel}\eqref{eq:PR_hex_compact} the \textit{shift cancellation} condition. In order to construct a critically-sampled PR filter bank with $m_k\in L^2(\R^2/\Lambda^*)$  having optimal continuity and supported on $\A_k+B(0,\epsilon)$, we thus need to continuously extend $m_k(\xi)$ across the boundaries of $\A_k$ such that \eqref{eq:PR_hex_identity_sum} and \eqref{eq:PR_hex_shift_cancel} hold valid for $a.e.~\xi\in\R^2$.

\subsubsection{Boundary classification}

Proposition~\ref{prop:permissibility} suggests the existence of certain boundaries of $\A_k$ beyond which $m_k$ cannot continuously extend while performing a perfect reconstruction. Building on the idea proposed in \cite{yin-wavelet-1}, we explain how to identify such boundaries while making the definition and analysis more rigorous.

Given $\epsilon\in [0,\sqrt{3}\pi/12)$, we assume the support $\supp(m_k)$ of $m_k\in L^2(\R^2/\Lambda^*)$ satisfies
\begin{align}
  \overline{\A_k} \subset \supp(m_k)\subset \overline{\A_k}+\overline{B(0,\epsilon)}.
\end{align}
For any $\gamma\in \Gamma_k^*\setminus \Lambda^*$, or equivalently $0\neq \gamma \in \Gamma_k^*/\Lambda^*$, define
\begin{align}
  \label{eq:def_C_epsilon}
  C_\epsilon(k,\gamma) \coloneqq \left(\overline{\A_k}+\overline{B(0,\epsilon)}\right) \cap \left(\overline{\A_k}+\overline{B(0,\epsilon)}-\gamma\right)
  \supset \supp\left(m_k(\cdot)\overline{m_k(\cdot+\gamma)}\right).
\end{align}
See Figure~\ref{fig:hex_shift_1} for the special case $C_{\epsilon}(1,(\pi,0))$. In particular, when $\epsilon = 0$,
\begin{align}
  C_0(k,\gamma) = \overline{\A_k} \cap \left(\overline{\A_k}-\gamma\right) \subset \partial \A_k.
\end{align}
The following result shows an unavoidable restriction on the continuity of $m_k$ across certain boundaries of $\A_k$.

\begin{figure}
    \centering
    \begin{subfigure}[t]{0.46\textwidth}
        \includegraphics[width=\textwidth]{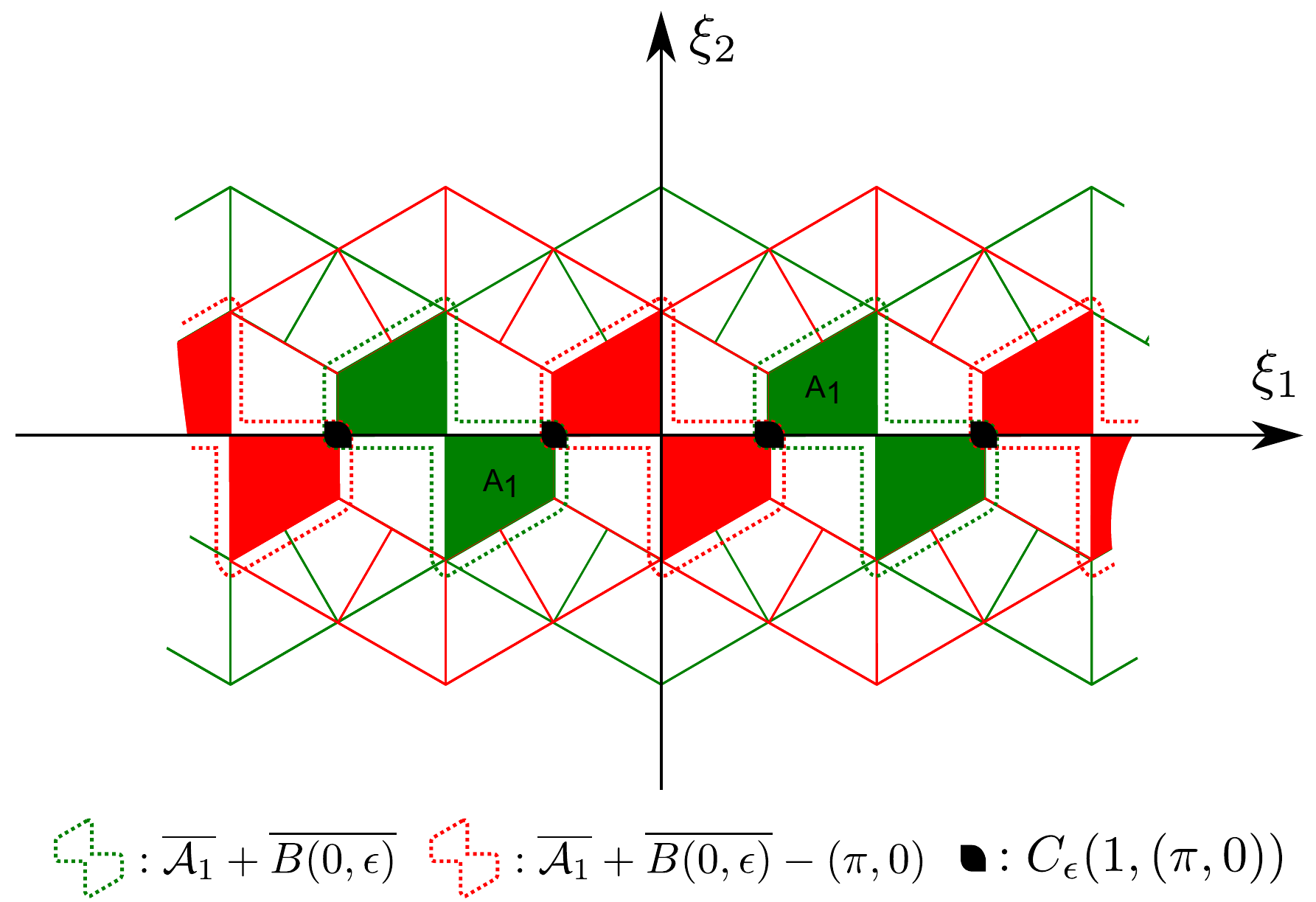}
        \caption{}
        \label{fig:hex_shift_1}
    \end{subfigure}
    ~~~
    \begin{subfigure}[t]{0.23\textwidth}
        \includegraphics[width=\textwidth]{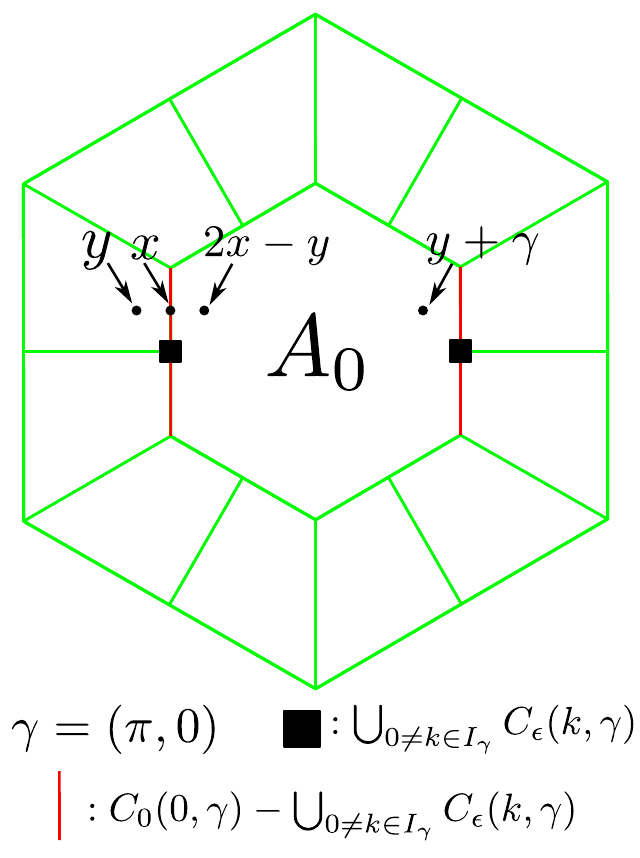}
        \caption{}
        \label{fig:nonextend}
    \end{subfigure}
    ~~~
    \begin{subfigure}[t]{0.23\textwidth}
        \includegraphics[width=\textwidth]{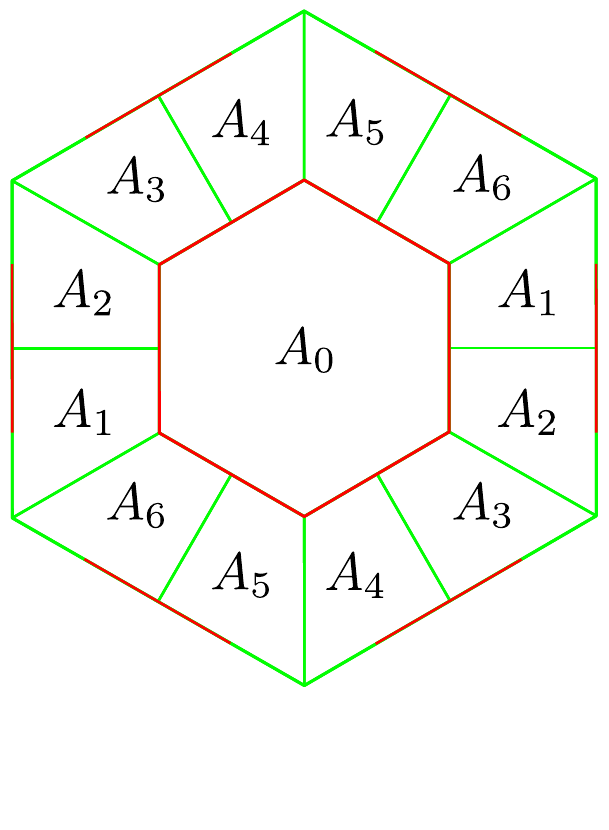}
        \caption{}
        \label{fig:bdry_classify}
    \end{subfigure}
    \caption{ \small \textbf{(a)}: The support of $m_1$ is larger than $\overline{\A_1}$ (solid green) and included in $\overline{\A_1}+\overline{B(0,\epsilon)}$ (the area enclosed in dashed green). The corresponding sets after shifting by $\gamma = (\pi,0)$ are colored in red. The set in color black corresponds to $C_\epsilon(1,(\pi,0))$, the intersection of $\overline{\A_1}+\overline{B(0,\epsilon)}$ (dashed green) and $\overline{\A_1}+\overline{B(0,\epsilon)}-\gamma$ (dashed red). Note that $C_\epsilon(k,(\pi,0)) \neq \emptyset$ only when $k=0,1,2$. \textbf{(b)}: A pictorial explanation for the proof of Proposition~\ref{prop:bdry} for the special case $k=0$ and $\gamma = (\pi,0)$. \textbf{(c)}: Regular boundaries (green) and singular boundaries (red) of the hexagonal wavelets with 6 directions.}
\end{figure}

\begin{prop}
  \label{prop:bdry}
  Suppose $\epsilon > 0$, $k\in \{0,\ldots, 6\}$, and $\gamma \in \Gamma_k^*\setminus \Lambda^*$, then $m_k$ cannot have nonzero continuous extension in the normal direction beyond the boundary
  \begin{align}
    C_0(k,\gamma) - \bigcup_{k'\in I_{\gamma} \atop k'\neq k}C_{\epsilon}(k',\gamma) \subset \partial \A_k
  \end{align}
without violating the PR condition \eqref{eq:PR_hex_compact}.
\end{prop}

\begin{proof}
We prove this by contradiction. If not, then there exist $y\not\in\overline{\A_k}$ and $x\in C_0(k,\gamma) - \bigcup_{k'\in I_{\gamma} \atop k'\neq k}C_{\epsilon}(k',\gamma)$ such that $(y-x)\perp\partial \A_k$, $m_k(y)\neq 0$, and $d = |y-x|>0$ is sufficiently small. Hence the reflection point $2x-y$ of $x$ with respect to $\partial \A_k$ is inside the interior of $\A_k$ (see Figure~\ref{fig:nonextend} for the special case $k=0$ and $\gamma= (\pi,0)$.) After shifting by $\gamma$, we have
\begin{align}
    2x-y+\gamma\not \in \A_k, \quad  x+\gamma \in \partial \A_k, \quad  (2x-y+\gamma) - (x+\gamma) \perp \partial \A_k,
\end{align}
because of the admissibility condition \eqref{eq:a_k_shift_admissible} of the partition. Thus the reflection $y+\gamma$ of $2x-y+\gamma$ with respect to $\partial \A_k$ is inside the interior of $\A_k \subset \supp(m_k)$. Therefore
\begin{align*}
  m_k(y)\overline{m_k(y+\gamma)} \neq 0.
\end{align*}
Because of the continuity of $m_k$, there exists a sufficiently small $\delta$-neighborhood $B(y,\delta)$ of $y$ such that
\begin{align*}
  B(y,\delta)\cap \left(\bigcup_{k'\in I_\gamma\atop k'\neq k}C_\epsilon(k,\gamma)\right) = \emptyset, \quad \text{and} ~~ m_k(\xi)\overline{m_k(\xi+\gamma)} \neq 0, ~\forall \xi\in B(y,\delta).
\end{align*}
Thus for any $\xi\in B(y,\delta)$,
\begin{align*}
  \sum_{k'\in I_{\gamma}}m_{k'}(\xi)\overline{m_{k'}(\xi+\gamma)} = m_k(\xi)\overline{m_k(\xi+\gamma)} \neq  0.
\end{align*}
This contradicts the PR condition \eqref{eq:PR_hex_compact}.
\end{proof}

Based on Proposition~\ref{prop:bdry}, we can classify the boundaries of $\A_k$ as singular or regular as follows
\begin{definition}
  \label{def:singular_regular}
  Given $k\in\{0,\ldots, 6\}$ and $\gamma \in \Gamma_k^*\setminus \Lambda^*$ (or equivalently $0\neq \gamma \in \Gamma_k^*/\Lambda^*$), define
  \begin{align}
    D(k,\gamma) \coloneqq \lim_{\epsilon\to 0}C_0(k,\gamma) - \bigcup_{k'\in I_{\gamma} \atop k'\neq k}C_{\epsilon}(k',\gamma) = C_0(k,\gamma) - \bigcup_{k'\in I_{\gamma} \atop k'\neq k}C_{0}(k',\gamma).
  \end{align}
The singular boundary $S(k)$ and regular boundary $R(k)$ of $\A_k$ are defined as
\begin{align}
  \label{eq:bdry_singular_regular}
  S(k) \coloneqq \bigcup_{\gamma \in \Gamma_k^*\setminus \Lambda^*}D(k,\gamma), \quad R(k) \coloneqq \partial \A_k\setminus S(k).
\end{align}
\end{definition}

Figure~\ref{fig:bdry_classify} shows the boundary classification of the hexagonal wavelets with six directions. Note that all singular boundaries are ``double'' singular boundaries in the sense that $m_k$ cannot extend beyond those boundaries in either direction (this will be an easy corollary of Lemma~\ref{lemma:common_regular} in Section~\ref{sec:bdry_smoothing}.)

\subsubsection{Filter smoothing across regular boundaries}
\label{sec:bdry_smoothing}

We discuss how to continuously extend $m_k(\xi)$ across regular boundaries without violating the PR condition \eqref{eq:PR_hex_identity_sum} \eqref{eq:PR_hex_shift_cancel}. The following lemma characterizes the common regular boundaries of two adjacent domains $\A_{k_1}$ and $\A_{k_2}$.

\begin{lemma}
  \label{lemma:common_regular}
  Let $\A_{k_1}$ and $\A_{k_2}$ be adjacent domains with a common boundary. We have

\noindent {\normalfont(a)}
  \begin{align}
    \label{eq:common_regular}
    R(k_1)\cap R(k_2) = \bigcup_{\gamma\in (\Gamma_{k_1}^*\cap \Gamma_{k_2}^*)\setminus\Lambda^*}C_0(k_1,\gamma)\cap C_0(k_2,\gamma)
  \end{align}
  holds up to a discrete set of ``corner'' points that live on the boundaries of  more than two $\A_k$.

\noindent {\normalfont(b)} If $\gamma, \gamma'\in\Gamma_{k}^*\setminus\Lambda^*$ and $\gamma-\gamma'\not \in \Lambda^*$, then
    \begin{align}
      \label{eq:C-distinct}
      C_0(k,\gamma) \cap C_0(k,\gamma') = \emptyset.
    \end{align}
Thus in particular the right hand side of \eqref{eq:common_regular} is the disjoint union of $C_0(k_1,\gamma)\cap C_0(k_2,\gamma)$ for $0\neq \gamma \in (\Gamma_{k_1}^*\cap\Gamma_{k_2}^*)/\Lambda^*$.
\end{lemma}

\begin{proof}
(a) Given $x\in R(k_1)\cap R(k_2)$, we have $x\in\partial \A_{k_1}\setminus S(k_1)$ by  Definition~\ref{def:singular_regular}. The admissibility of $(A_k)_{0\le k\le 6}$ implies
  \begin{align*}
    \exists  \gamma\in\Gamma_{k_1}^*\setminus\Lambda^*, \quad \text{s.t.} ~~x\in C_0(k_1,\gamma).
  \end{align*}
Since $x\not \in S(k_1)$, we have, in particular, 
\begin{align*}
  x\not\in D(k_1,\gamma) = C_0(k_1,\gamma)-\bigcup_{k'\in I_\gamma \atop k'\neq k_1}C_0(k',\gamma).
\end{align*}
This implies $x\in C_0(k',\gamma)\subset \partial\A_{k'}$ for some $k'\in I_\gamma, k'\neq k_1$. This is only possible when $k' = k_2$ if $x$ does not belong to the boundary of more than two $\A_k$. Therefore $x\in C_0(k_2,\gamma)$, and $k_2\in I_\gamma$, i.e., $\gamma \in\Gamma_{k_2}^*$. We thus have
\begin{align*}
  x \in \bigcup_{\gamma\in (\Gamma_{k_1}^*\cap \Gamma_{k_2}^*)\setminus\Lambda^*}C_0(k_1,\gamma)\cap C_0(k_2,\gamma) 
\end{align*}
The other direction is trivial from Definition~\ref{def:singular_regular}.

(b) This is an easy result of the admissibility condition \eqref{eq:a_k_shift_admissible}.
\end{proof}

Thus the common regular boundary of $\A_{k_1}$ and $\A_{k_2}$ consists of
\begin{align}
  \label{eq:def_E}
  E(k_1,k_2,\gamma) \coloneqq C_0(k_1,\gamma)\cap C_0(k_2,\gamma), \quad \gamma \in (\Gamma_{k_1}^*\cap \Gamma_{k_2}^*)\setminus\Lambda^*.
\end{align}
The following lemma shows that $E(k_1,k_2,\gamma)$ can be paired with $E(k_1,k_2,-\gamma)$ according to the shift $\pm \gamma$.

\begin{lemma}
  Two common regular boundaries, $E(k_1,k_2,\gamma)$ and $E(k_1,k_2,-\gamma)$, of $\A_{k_1}$ and $\A_{k_2}$ can be paired according to the shift $\pm \gamma$:
  \begin{align}
    E(k_1,k_2,-\gamma) = E(k_1,k_2,\gamma) + \gamma.
  \end{align}
Moreover, if $\gamma\neq -\gamma$, i.e., $2\gamma \not\in\Lambda^*$, then
\begin{align}
  \label{eq:E-distinct}
  E(k_1,k_2,-\gamma) \cap E(k_1,k_2,\gamma) = \emptyset
\end{align}
\end{lemma}
\begin{proof}
  This can be easily proved by the definition \eqref{eq:def_E}:
  \begin{align*}
    E(k_1,k_2,\gamma) &= C_0(k_1,\gamma)\cap C_0(k_2,\gamma) = \left(\overline{\A_{k_1}} \cap \left(\overline{\A_{k_1}} -\gamma\right)\right) \cap \left(\overline{\A_{k_2}} \cap \left(\overline{\A_{k_2}} -\gamma\right)\right).\\
    E(k_1,k_2,-\gamma) &= C_0(k_1,-\gamma)\cap C_0(k_2,-\gamma) = \left(\overline{\A_{k_1}} \cap \left(\overline{\A_{k_1}} +\gamma\right)\right) \cap \left(\overline{\A_{k_2}} \cap \left(\overline{\A_{k_2}} +\gamma\right)\right).
  \end{align*}
Hence we have $E(k_1,k_2,-\gamma) = E(k_1,k_2,\gamma) + \gamma$. The fact that $E(k_1,k_2,-\gamma) \cap E(k_1,k_2,\gamma) =\emptyset$ when $\gamma \neq -\gamma$ is an easy corollary of \eqref{eq:C-distinct}.
\end{proof}

\begin{SCfigure}[]
  \centering
  \includegraphics[width=.5\textwidth]{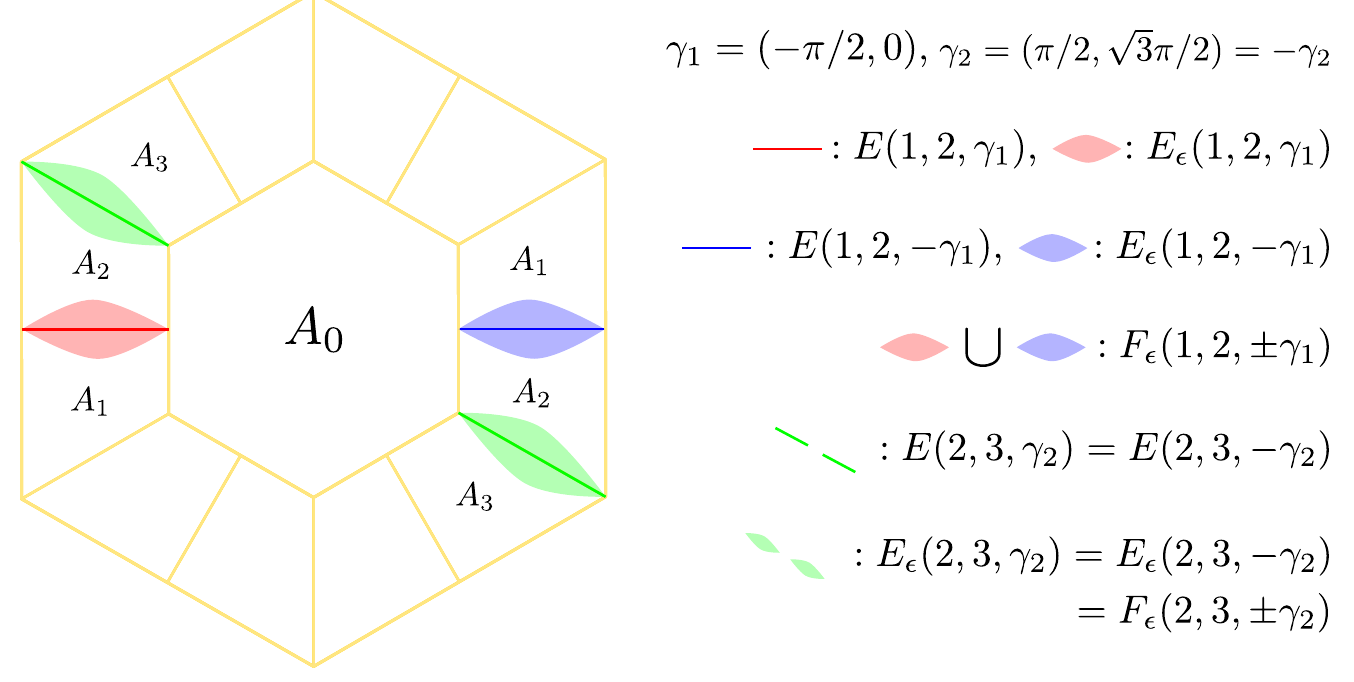}
  \caption{ \small Two particular cases for  $E(k_1,k_2,\gamma)$,  $E_\epsilon(k_1,k_2,\gamma)$, and  $F_\epsilon(k_1,k_2,\pm\gamma)$.\vspace{.2em}\\
Case (1): $(k_1,k_2,\gamma_1) = (1,2,(-\pi/2,0))$. In this case $\left|E_\epsilon(1,2,\gamma_1)\cap E_{\epsilon}(1,2,-\gamma_1)\right|=0$, for $\gamma_1 \neq -\gamma_1$. $F_\epsilon(1,2,\pm\gamma_1) = \Omega \cup (\Omega+\gamma)$ is the disjoint union of $\Omega$ and $\Omega+\gamma_1$, where $\Omega = E_\epsilon(1,2,\gamma_1)$.\vspace{.2em}\\
Case (2): $(k_1,k_2,\gamma_2) = (2,3,(\pi/2,\sqrt{3}\pi/2))$. In this case $E_\epsilon(2,3,\gamma_2) =  E_{\epsilon}(2,3,-\gamma_2)=F_\epsilon(2,,\pm\gamma_2)$, for $\gamma_2 = -\gamma_2$.  $F_\epsilon(2,3,\pm\gamma_2)$ is still the disjoint union of $\Omega$ and $\Omega+\gamma_2$, where $\Omega$ is the upper-left (or lower-right) green shaded region.
  }
  \label{fig:bdry_E_F}
\end{SCfigure}

Define $F_\epsilon(k_1,k_2,\pm \gamma)$ as the union of the pair $E_\epsilon(k_1,k_2,\pm \gamma)$
\begin{align}
  \label{eq:def_F}
  F_\epsilon(k_1,k_2,\pm \gamma)\coloneqq E_\epsilon(k_1,k_2,\gamma) \cup E_\epsilon(k_1,k_2,-\gamma),
\end{align}
where $E_\epsilon(k_1,k_2,\gamma)$ is an $\epsilon$-neighborhood of $E(k_1,k_2,\gamma)$ satisfying
\begin{enumerate}
\item $\text{dist}(\xi,E(k_1,k_2,\gamma))\le \epsilon,~ \forall \xi \in E_\epsilon(k_1,k_2,\gamma)$.
\item $E_\epsilon(k_1,k_2,\gamma)$ and $E_\epsilon(k_1,k_2,-\gamma)$ satisfy the same shifting relation
  \begin{align}
    \label{eq:E_epsilon_shift}
    E_\epsilon(k_1,k_2,-\gamma) = E_\epsilon(k_1,k_2,\gamma) + \gamma    
  \end{align}  
\item For every distinct triple $(k_1,k_2,\gamma)\neq (\widetilde{k_1},\widetilde{k_2},\widetilde{\gamma})$,
  \begin{align}
    \label{eq:E-epsilon-distinct}
    \left|E_\epsilon(k_1,k_2,\gamma) \cap E_\epsilon(\widetilde{k_1},\widetilde{k_2},\widetilde{\gamma})\right| = 0,
  \end{align}
where $\left|\cdot\right|$ is the Lebesgue measure.
\end{enumerate}
See Figure~\ref{fig:bdry_E_F} for a visual illustration of the sets $E(k_1,k_2,\gamma), E_\epsilon(k_1,k_2,\gamma)$ and $F_\epsilon(k_1,k_2,\pm\gamma)$. The following proposition simplifies the PR condition on $m_{k_1}$ and $m_{k_2}$ if only the values of these two filters are changed on $F_\epsilon(k_1,k_2,\pm \gamma)$.

\begin{prop}
  \label{prop:change_m1_m2}
  Suppose $\left(m_k^{(0)}\right)_{0\le k\le 6}$ satisfies the PR condition \eqref{eq:PR_hex_identity_sum} and \eqref{eq:PR_hex_compact}, and
  \begin{align}
    \label{eq:maximum_supp}
    \supp(m_k^{(0)}) \subset \overline{\A_k}\cup \left(\bigcup_{k' \text{s.t.}\atop R(k)\cap R(k')\neq \emptyset}\bigcup_{\gamma \in \left(\Gamma_k^* \cap \Gamma_{k'}^*\right)\setminus \Lambda^*}E_\epsilon(k',k,\gamma)  \right).
  \end{align}
In particular, \eqref{eq:maximum_supp} implies $m_k^{(0)}$ has (possibly) been extended across some regular boundaries of $\A_k$, and at most two distinct filters $m_k^{(0)}$, $m_{k'}^{(0)}$ can take nonzero values for $a.e. ~\xi\in\R^2$ (because of \eqref{eq:E-epsilon-distinct}).

Suppose the new filter bank $\left( m_k \right)_{0\le k\le 6}$ differs from $\left(m_k^{(0)}\right)_{0\le k\le 6}$ only on $F_\epsilon(k_1,k_2,\pm \gamma)$ for the filters $m_{k_1}, m_{k_2}$, i.e.,
\begin{align}
  \label{eq:change_m1_m2}
  m_k(\xi) \neq m_k^{(0)}(\xi) \implies \xi\in F_\epsilon(k_1,k_2,\pm \gamma), ~\text{and} ~~k = k_1 ~\text{or} ~k_2,
\end{align}
then $\left( m_k \right)_{0\le k\le 6}$ satisfies the PR condition if and only if
\begin{align}
  \label{eq:new_PR}
  \left\{
  \begin{aligned}
    &\left|m_{k_1}(\xi)\right|^2 + \left|m_{k_2}(\xi)\right|^2  = 1, \quad &&a.e. ~\xi \in F_\epsilon(k_1,k_2,\pm\gamma)\\
    & m_{k_1}(\xi)\overline{m_{k_1}(\xi+\gamma)} + m_{k_2}(\xi)\overline{m_{k_2}(\xi+\gamma)} = 0, \quad && a.e. ~\xi \in E_{\epsilon}(k_1,k_2,\gamma)\\
    & m_{k_1}(\xi)\overline{m_{k_1}(\xi-\gamma)} + m_{k_2}(\xi)\overline{m_{k_2}(\xi-\gamma)} = 0, \quad && a.e. ~\xi \in E_{\epsilon}(k_1,k_2,-\gamma)
  \end{aligned}\right.
\end{align}
\end{prop}

\begin{proof}
  We first verify the PR condition \eqref{eq:PR_hex_identity_sum}\eqref{eq:PR_hex_compact} when $\xi \not \in F_\epsilon(k_1,k_2,\pm\gamma)$.
  \begin{itemize}
  \item The identity summation condition \eqref{eq:PR_hex_identity_sum} always holds for $\xi \not \in F_\epsilon(k_1,k_2,\pm\gamma)$, for \eqref{eq:change_m1_m2} implies
    \begin{align*}
      \sum_{k=0}^6\left|m_k(\xi)\right|^2=\sum_{k=0}^6\left|m_k^{(0)}(\xi)\right|^2=1
    \end{align*}
  \item If the shift cancellation condition \eqref{eq:PR_hex_compact} fails to hold, i.e.,
    \begin{align*}
      \sum_{k\in I_{\nu}}\overline{m_k(\xi+\nu)}m_k(\xi) \neq 0 = \sum_{k\in I_{\nu}}\overline{m_k^{(0)}(\xi+\nu)}m_k^{(0)}(\xi)   
    \end{align*}
for some $0\neq \nu \in \Gamma^*/\Lambda^*$, then
\begin{align*}
  \left\{
  \begin{aligned}
    &\xi+\nu\in F_\epsilon(k_1,k_2,\pm\gamma) = E_\epsilon(k_1,k_2,\gamma)\cup E_\epsilon(k_1,k_2,-\gamma)\\
    &m_{k_j}(\xi)\neq 0\\
    &0\neq\nu\in \Gamma_{k_j}^*/\Lambda^*
  \end{aligned}\right.
\end{align*}
must hold simultaneously for at least one of the $j\in\{1,2\}$. Assume without loss of generality that 
    \begin{align*}
      \xi+\nu\in E_\epsilon(k_1,k_2,\gamma), \quad m_{k_1}(\xi)\neq 0, \quad \text{and} ~~0\neq\nu\in \Gamma_{k_1}^*/\Lambda^*.
    \end{align*}
The admissibility condition \eqref{eq:a_k_shift_admissible} combined with the maximum support of $m_{k_1}$ \eqref{eq:maximum_supp} implies that $\text{dist}(\xi,\partial\A_{k_1})\le \epsilon$. Let $\xi'\in\partial\A_{k_1}$ be the projection of $\xi$ in $\partial\A_{k_1}$, i.e.,
\begin{align*}
  \xi - \xi'\perp\partial \A_{k_1}, \quad \|\xi-\xi'\|\le \epsilon.
\end{align*}
Then $\xi'+\nu\in E(k_1,k_2,\gamma)\subset E_\epsilon(k_1,k_2,\gamma)$. The admissibility condition implies that this can hold only if $\nu = -\gamma$, and thus
\begin{align*}
 \xi\in E_{\epsilon}(k_1,k_2,\gamma)-\nu = E_{\epsilon}(k_1,k_2,\gamma)+\gamma = E_{\epsilon}(k_1,k_2,-\gamma) \subset F_{\epsilon}(k_1,k_2,\pm\gamma)
\end{align*}
This contradicts the  assumption  $\xi\not \in F_{\epsilon}(k_1,k_2,\pm\gamma)$.
  \end{itemize}
We next show that the PR condition holds on $F_{\epsilon}(k_1,k_2,\pm\gamma)$ when \eqref{eq:new_PR} is satisfied.
\begin{itemize}
\item Since $m_{k_1}$ and $m_{k_2}$ are the only two filters that can take non-zero values on $F_{\epsilon}(k_1,k_2,\pm\gamma)$, we have
  \begin{align*}
    \sum_{k=0}^6\left|m_k(\xi)\right|^2 = 1 \iff \left|m_{k_1}(\xi)\right|^2 + \left|m_{k_2}(\xi)\right|^2  = 1, \quad \forall~\xi \in F_\epsilon(k_1,k_2,\pm \gamma).
  \end{align*}
\item When $\xi\in E_\epsilon(k_1,k_2,\gamma)$, the admissibility condition \eqref{eq:a_k_shift_admissible} and \eqref{eq:maximum_supp} imply that $m_k(\xi)\overline{m_k(\xi+\nu)}\neq 0$  for some $\nu\not \in\Lambda^*$ only if $k=k_1$ or $k_2$, and $\nu= \gamma$. Hence
  \begin{align*}
    \sum_{k\in I_{\nu}}\overline{m_k(\xi+\nu)}m_k(\xi) = 0,~\forall \nu \in \Gamma^*\setminus \Lambda^* \iff m_{k_1}(\xi)\overline{m_{k_1}(\xi+\gamma)} + m_{k_2}(\xi)\overline{m_{k_2}(\xi+\gamma)} = 0
  \end{align*}
\item Similarly we can show that the shift-cancellation condition for $\xi\in E_\epsilon(k_1,k_2,-\gamma)$ is equivalent to the last equation of \eqref{eq:new_PR}.
\end{itemize}
\end{proof}

Therefore, in order to extend $m_{k_1}$ and $m_{k_2}$ across their common regular boundary $E(k_1,k_2,\gamma)\cup E(k_1,k_2,-\gamma)$, we only need to ensure \eqref{eq:new_PR} holds on $F_\epsilon(k_1,k_2,\pm\gamma)$ while leaving other filters $(m_k)_{k \neq k_1, k_2}$ unchanged. Note that whether $\gamma$ and $-\gamma$ are identical in $(\Gamma_{k_1}^*\cap \Gamma_{k_2}^*)/\Lambda^*$, $F_\epsilon(k_1,k_2,\pm\gamma)$ is always the (almost) disjoint union of $\Omega$ and $\Omega + \gamma$, since
\begin{itemize}
\item if $\gamma\neq \gamma$, then $ F_\epsilon(k_1,k_2,\pm \gamma) =  E_\epsilon(k_1,k_2,\gamma) \cup E_\epsilon(k_1,k_2,-\gamma) \eqqcolon \Omega \cup (\Omega +\gamma)$, and $\left|\Omega \cap (\Omega+\gamma)\right| = \left| E_\epsilon(k_1,k_2,\gamma)\cap E_\epsilon(k_1,k_2,-\gamma) \right| = 0$ because of \eqref{eq:E-epsilon-distinct}. 
\item if $\gamma = -\gamma$, i.e., $2\gamma \in \Lambda^*$, then \eqref{eq:E_epsilon_shift} implies $F_\epsilon(k_1,k_2,\pm\gamma) = E_{\epsilon}(k_1,k_2,\gamma)= E_{\epsilon}(k_1,k_2,-\gamma)$ is still the union of its two disjoint subsets $\Omega$ and $\Omega+\gamma$ (see the green shaded region in Figure~\ref{fig:bdry_E_F} representing the special case $F_\epsilon(2,3,\pm(\pi/2,\sqrt{3}\pi/2))$.)
\end{itemize}
The following proposition shows that after choosing a proper set of spatial translations $\left\{\eta_k\right\}_{0\le k\le 6}\subset\Lambda$, one needs only a mild restriction on the moduli $\left|m_{k_1}(\xi)\right|, \left|m_{k_2}(\xi)\right|$  for the filters to satisfy the PR condition \eqref{eq:new_PR} when extending $m_{k_1}$ and $m_{k_2}$ onto $F_\epsilon(k_1,k_2,\pm\gamma) =\Omega\cup (\Omega+\gamma)$.

\begin{prop}
\label{prop:omega_eta}
  Let $\M_{k_1}$ and $\M_{k_2}$ be real-valued $\Lambda^*$-periodic functions defined on $\Omega \cup (\Omega+\gamma)$ satisfying
  \begin{align}
    \left\{
    \begin{aligned}
      &\M_{k_1}(\xi), \M_{k_2}(\xi)\ge 0,~~ \M_{k_1}^2(\xi) + \M_{k_2}^2(\xi) = 1,\\
      & \M_{k_1}(\xi+\gamma) = \M_{k_2}(\xi), ~~ \M_{k_2}(\xi+\gamma) = \M_{k_1}(\xi),
    \end{aligned}\right. \quad\quad \forall \xi\in \Omega.
  \end{align}
Define $m_{k_i}(\xi)  \coloneqq \M_{k_i}(\xi)e^{i\left<\xi,\eta_{k_i}\right>}$, where $\eta_{k_i}\in\Lambda$ satisfies $e^{i\left<\gamma,\eta_{k_1}-\eta_{k_2}\right>} = -1$, then
\begin{align}
  \left\{
  \begin{aligned}
        &\left|m_{k_1}(\xi)\right|^2 + \left|m_{k_2}(\xi)\right|^2  = 1, \quad && \forall \xi \in \Omega\cup (\Omega+\gamma) \\
    & m_{k_1}(\xi)\overline{m_{k_1}(\xi+\gamma)} + m_{k_2}(\xi)\overline{m_{k_2}(\xi+\gamma)} = 0, \quad && \forall \xi\in \Omega\\
    & m_{k_1}(\xi)\overline{m_{k_1}(\xi-\gamma)} + m_{k_2}(\xi)\overline{m_{k_2}(\xi-\gamma)} = 0, \quad && \forall \xi \in \Omega+\gamma
  \end{aligned}\right.
\end{align}
\end{prop}
\begin{proof}
  This can be  verified directly. For instance, given $\xi \in \Omega$,
  \begin{align*}
    & m_{k_1}(\xi)\overline{m_{k_1}(\xi+\gamma)} + m_{k_2}(\xi)\overline{m_{k_2}(\xi+\gamma)} \\
    = &\M_{k_1}(\xi)e^{i\left<\xi,\eta_{k_1}\right>}\M_{k_1}(\xi+\gamma)e^{-i\left<\xi+\gamma,\eta_{k_1}\right>} + \M_{k_2}(\xi)e^{i\left<\xi,\eta_{k_2}\right>}\M_{k_2}(\xi+\gamma)e^{-i\left<\xi+\gamma,\eta_{k_2}\right>}\\
    = &\M_{k_1}(\xi) \M_{k_2}(\xi) e^{-i\left<\gamma,\eta_{k_1}\right>} + \M_{k_2}(\xi) \M_{k_1}(\xi) e^{-i\left<\gamma,\eta_{k_2}\right>} = 0
  \end{align*}
\end{proof}

We thus need to find $\left\{\eta_k\right\}_{0\le k\le 6}$ such that $e^{i\left<\gamma,\eta_{k_1}-\eta_{k_2}\right>} = -1$ for every triple $(k_1,k_2,\gamma)$ corresponding to the regular boundary $E(k_1,k_2,\gamma)\cup E(k_1,k_2,-\gamma)$. In the case of hexagonal wavelets with six directions, such collection of triples is
\begin{align}\nonumber
  T=\left\{ (1,2,\left(\pi/2,0\right)), (2,3,\left(\pi/2,\sqrt{3}\pi/2\right)), (3,4, \left( \pi/4,-\sqrt{3}\pi/4\right)),\right.\\ \label{eq:basis_triples}
  \left.(4,5,\left(\pi,0\right)), (5,6, \left(\pi/4,\sqrt{3}\pi/4 \right)), (6,1,\left( \pi/2,-\sqrt{3}\pi/2\right)) \right\}.
\end{align}
Note that we have omitted the regular boundaries on the boundary of $S$ (see Figure~\ref{fig:bdry_classify}) since extending $m_k$ beyond such boundaries also causes aliasing after $\Lambda^*$ periodic folding. One can easily verify that the following choice of $\left\{\eta_k\right\}_{0\le k\le 6}$ satisfies the condition in Proposition~\ref{prop:omega_eta} for all triples in \eqref{eq:basis_triples}
\begin{align}
  \label{eq:eta_basis}
  \eta_0 = (0,0), \eta_1 = (-1,-\sqrt{3}), \eta_2 = (1,\sqrt{3}), \eta_3 = (2,0), \eta_4 = (-2,0), \eta_5 = (-1,\sqrt{3}), \eta_6 = (1,-\sqrt{3}).
\end{align}

With the specific choice of $\left\{ \eta_k\right\}_{0\le k\le 6}$ in \eqref{eq:eta_basis}, Algorithm~\ref{alg:smooth_basis} summarizes the iterative procedure of continuously extending $(m_k)_{0\le k\le 6}$ across their regular boundaries based on Proposition~\ref{prop:change_m1_m2} and~\ref{prop:omega_eta}. When designing $\left(\M_k\right)_{0\le k\le 6}$, we let the  values $\M_{k_1}(\xi), \M_{k_2}(\xi)$ change continuously on each $F_\epsilon(k_1,k_2,\pm\gamma)$ while obeying \eqref{eq:alg_sum_to_1}. Each $\M_k$ is set to be symmetric with respect to the origin so that $\psi^k$ and $\phi$ are real-valued. Furthermore, we require $\left( \M_k\right)_{0\le k\le 6}$ to be $2\pi/3$-rotation invariant in the sense that both $\M_0$ and the set $\left\{(\M_1,\M_2),(\M_3,\M_4),(\M_5,\M_6) \right\}$ are invariant by a $2\pi/3$ rotation.

\begin{algorithm}
\floatname{algorithm}{Algorithm}
\caption{Iterative filter smoothing across regular boundaries}
\label{alg:smooth_basis}
\begin{algorithmic}[1]
\REQUIRE $\M_k = \chi_{\A_k}, k = 0, \ldots, 6$
\ENSURE  PR filter bank $(m_k)_{0\le k\le 6}$ continuously extended across regular boundaries.
\FOR{$(k_1,k_2,\gamma)\in T$ defined in \eqref{eq:basis_triples}}
\STATE Identify disjoint $\Omega$ and $\Omega+\gamma$ such that $F_\epsilon(k_1,k_2,\pm\gamma)=\Omega \cup (\Omega+\gamma)$.
\STATE Design real-valued functions $(\M_{k_1}, \M_{k_2})$ on $\Omega$ such that
\begin{align}
  \label{eq:alg_sum_to_1}
  \M_{k_1}^2(\xi) + \M_{k_1}^2(\xi) = 1, \quad \forall \xi \in \Omega.
\end{align}
\STATE Define $(\M_{k_1}, \M_{k_2})$ on $\Omega+\gamma$:
\begin{align*}
  \M_{k_1}(\xi) = \M_{k_2}(\xi-\gamma), ~\M_{k_2}(\xi) = \M_{k_1}(\xi-\gamma), \quad \forall \xi \in \Omega+\gamma.
\end{align*}
\ENDFOR
\STATE Let $m_k(\xi) = \M_{k}(\xi)e^{i\left<\xi,\eta_{k}\right>}, k = 0\ldots, 6$, where $\left(\eta_k\right)_{0\le k\le 6}$ are chosen as in \eqref{eq:eta_basis}.
\end{algorithmic}
\end{algorithm}

\begin{figure}
    \centering
    \begin{subfigure}[t]{0.45\textwidth}
        \includegraphics[width=\textwidth]{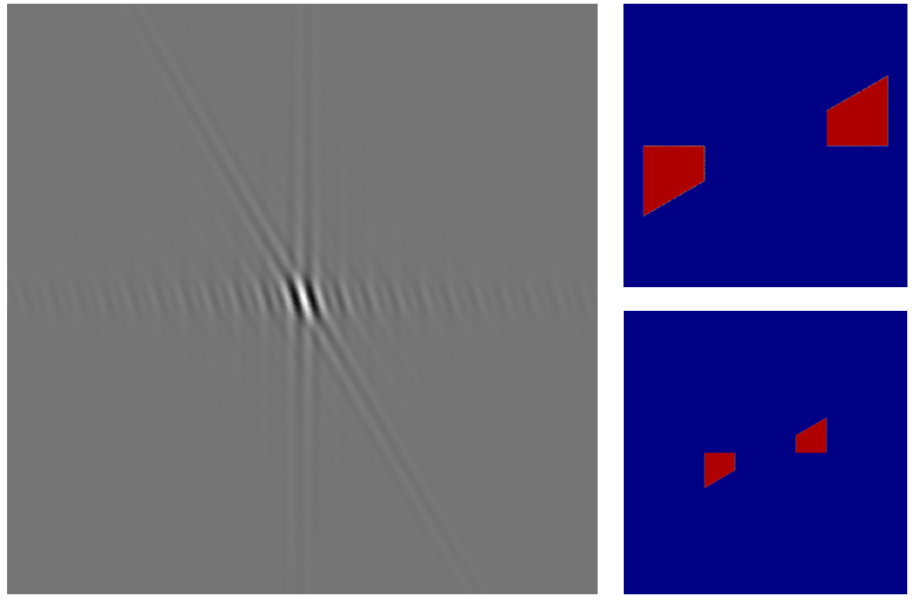}
        \caption{\small Shannon wavelet $\psi^1_{\text{sh}}$.}
        \label{fig:lvl0_shannon_idx_1}
    \end{subfigure}
     ~
    \begin{subfigure}[t]{0.45\textwidth}
        \includegraphics[width=\textwidth]{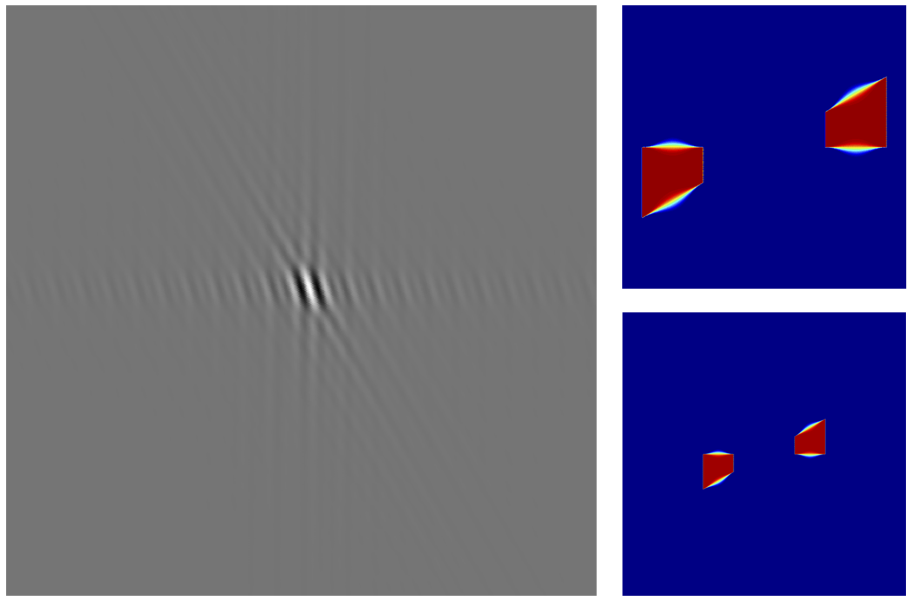}
        \caption{\small $\psi^1_{\text{ob1}}$ after regular boundary smoothing.}
        \label{fig:lvl0_ON1_idx_1}
    \end{subfigure}
     ~
    \begin{subfigure}[t]{0.45\textwidth}
        \includegraphics[width=\textwidth]{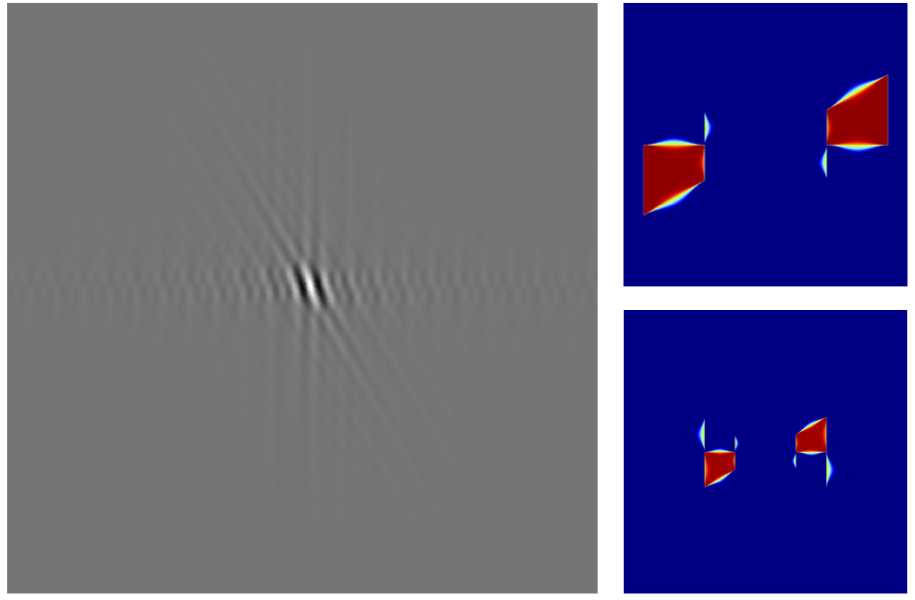}
        \caption{\small $\psi^1_{\text{ob2}}$ after regular boundary and refinement filter smoothing.}
        \label{fig:lvl0_ON2_idx_1}
    \end{subfigure}
     ~
    \begin{subfigure}[t]{0.45\textwidth}
        \includegraphics[width=\textwidth]{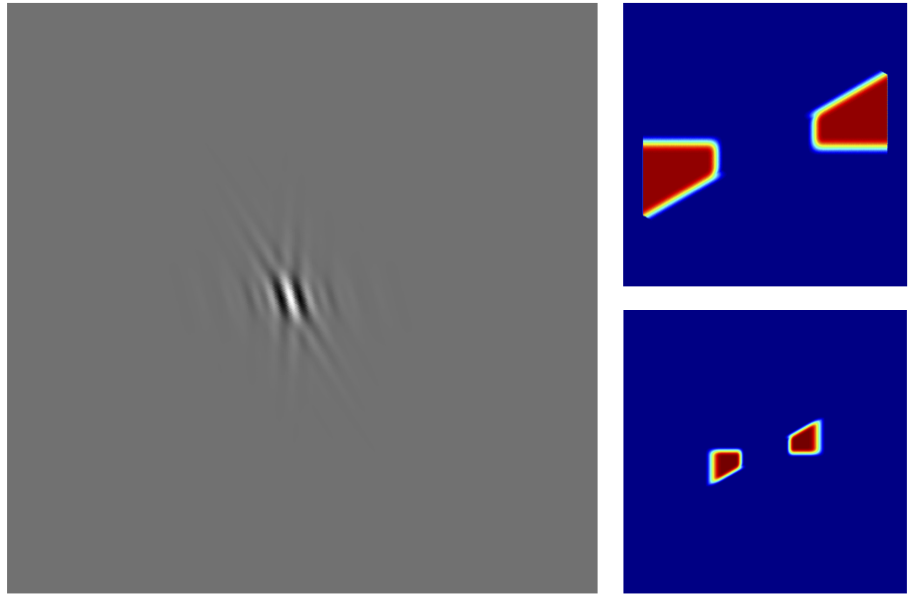}
        \caption{\small Wavelet frame $\psi^1_{\text{fr}}$. }
        \label{fig:lvl0_frame_idx_1}
    \end{subfigure}
    \caption{\small The wavelet function $\psi^1$ after boundary smoothing. \textbf{(a)} Shannon wavelet $\psi^1_{\text{sh}}$ (no smoothing.) \textbf{(b)} $\psi^1_{\text{ob1}}$ obtained from regular boundary smoothing in Section~\ref{sec:bdry_smoothing}. \textbf{(c)} $\psi^1_{\text{ob2}}$ obtained from regular boundary and refinement filter smoothing in Section~\ref{sec:scaling_smoothing}. \textbf{(d)} Wavelet frame $\psi^1_{\text{fr}}$ in Section~\ref{sec:frame}. For every block of three figures, the large figure on the left shows $\psi^1$ in the spatial domain, the upper right figure displays the modulus $\M_1=|m_1|$ of the filter $m_1\in L^2(\R^2/\Lambda^*)$, and the lower right figure shows the modulus $|\widehat{\psi^1}|$ in the frequency domain, which is supported mainly on $A_1/2$.}
\label{fig:lvl0_idx_1}
\end{figure}

Figure~\ref{fig:lvl0_ON1_idx_1} shows in particular the design of $\M_1\in L^2(\R^2/\Lambda^*)$ (the upper right figure) after regular boundary smoothing,  the corresponding wavelet function $\psi^1$ (the large figure on the left) in the spatial domain, and the modulus $|\widehat{\psi^1}|$ of its Fourier transform  (the lower right figure.) Compared to its Shannon counterpart (Figure~\ref{fig:lvl0_shannon_idx_1}), it is clear that $\psi^1$ after boundary smoothing has much faster decay in the directions corresponding to the regular boundaries, while the horizontal oscillation is unavoidable due to the vertical singular boundary.

Putting everthing together, we thus have the following theorem
\begin{thm}
  Let $\left(m_k\right)_{0\le k\le 6}$ be the (normalized) critically sampled PR filter bank continuously extended across the regular boundaries of $(\A_k)_{0\le k\le 6}$ according to Algorithm~\ref{alg:smooth_basis}. The correponding scaling function $\phi$ and  wavelets $\left\{\psi^k\right\}_{k=1}^6$ defined in \eqref{eq:def_phi} and \eqref{eq:def_psi} have optimal continuity in the frequency domain satisfying
  \begin{enumerate}
  \item $\left\{\psi_{j,n}^k \right\}_{1\le k\le 6,j\in\Z,n\in 2^{j-1}\Gamma_k}$ is an orthonormal basis of $L^2(\R^2)$.
  \item The functions $\phi$ and $\psi^k$ are well localized in the frequency domain according to the admissible frequency partition $S = \bigcup_{k=0}^6A_k$:
    \begin{align*}
      \sup_{\xi'\in\supp\left(\widehat{\phi_{1}}\right)}\inf_{\xi\in A_0}\|\xi'-\xi\| \le \epsilon, ~\text{and}~
      \sup_{\xi'\in\supp\left(\widehat{\psi^k_{1}}\right)}\inf_{\xi\in A_k}\|\xi'-\xi\| \le \epsilon, ~1\le k\le 6.
    \end{align*}
  \end{enumerate}
Moreover, the wavelet basis of $L^2(\R^2)$ is $2\pi/3$-rotation invariant in the sense that the scaling function $\phi$ and the set $\left\{\left(\psi^1,\psi^2 \right), \left(\psi^3,\psi^4 \right), \left(\psi^5,\psi^6 \right)  \right\}$ are invariant by a $2\pi/3$ rotation.
\end{thm}
\begin{proof}
  The only thing left to show is the rotation invariance of the wavelet system, which, by the definition of $\phi$ and $\left\{\psi^k\right\}_{k=1}^6$, is equivalent to the rotation invariance of $m_0$ and $\left\{(m_1,m_2), (m_3,m_4), (m_5,m_6) \right\}$, where $m_k(\xi) = \M_k(\xi)e^{i\left<\eta_k,\xi\right>}$. This is true because both the moduli $\M_0$, $\left\{(\M_1,\M_2), (\M_3,\M_4), (\M_5,\M_6) \right\}$ and the phases  $e^{\left<\eta_0,\cdot\right>}$, $\{(e^{\left<\eta_1,\cdot\right>},e^{\left<\eta_2,\cdot\right>}), (e^{\left<\eta_3,\cdot\right>},e^{\left<\eta_4,\cdot\right>}), (e^{\left<\eta_5,\cdot\right>},e^{\left<\eta_6,\cdot\right>})\}$ are constructed to be invariant by a $2\pi/3$ rotation. 
\end{proof}

\subsubsection{Smoothing of the refinement filter}
\label{sec:scaling_smoothing}

\begin{figure}
    \centering
    \begin{subfigure}[t]{0.45\textwidth}
        \includegraphics[width=\textwidth]{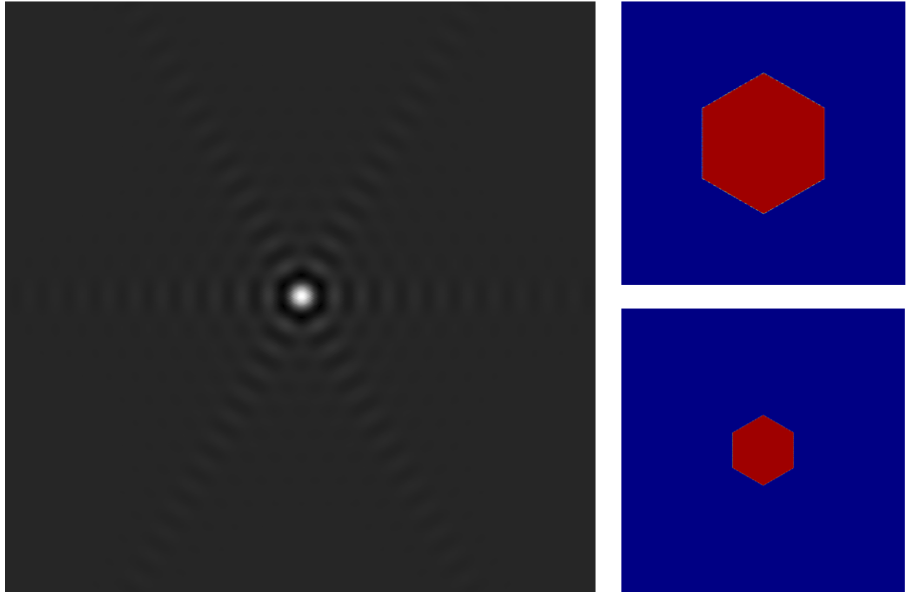}
        \caption{\small Shannon scaling function $\phi_{\text{sh}}$.}
        \label{fig:lvl0_shannon_idx_0}
    \end{subfigure}
     ~
    \begin{subfigure}[t]{0.45\textwidth}
        \includegraphics[width=\textwidth]{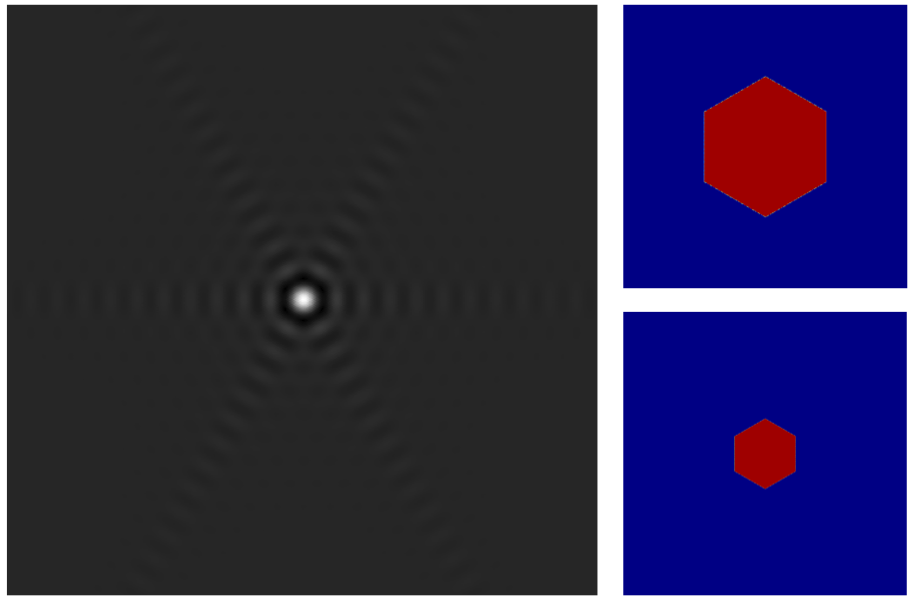}
        \caption{\small $\phi_{\text{ob1}}$ after regular boundary smoothing (unchanged because $\A_0$ has no regular boundary.)}
        \label{fig:lvl0_ON1_idx_0}
    \end{subfigure}
     ~
    \begin{subfigure}[t]{0.45\textwidth}
        \includegraphics[width=\textwidth]{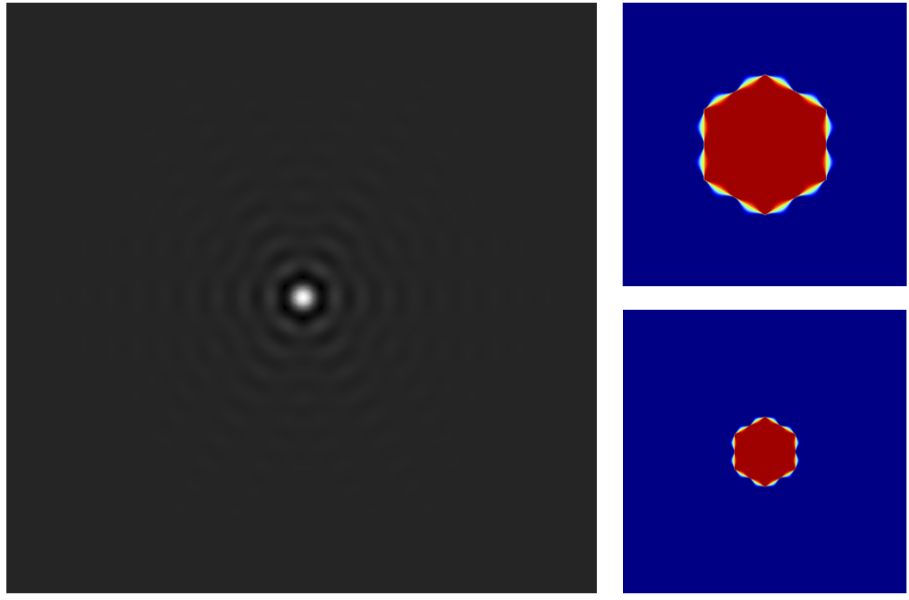}
        \caption{\small $\phi_{\text{ob2}}$ after regular boundary and refinement filter smoothing.}
        \label{fig:lvl0_ON2_idx_0}
    \end{subfigure}
     ~
    \begin{subfigure}[t]{0.45\textwidth}
        \includegraphics[width=\textwidth]{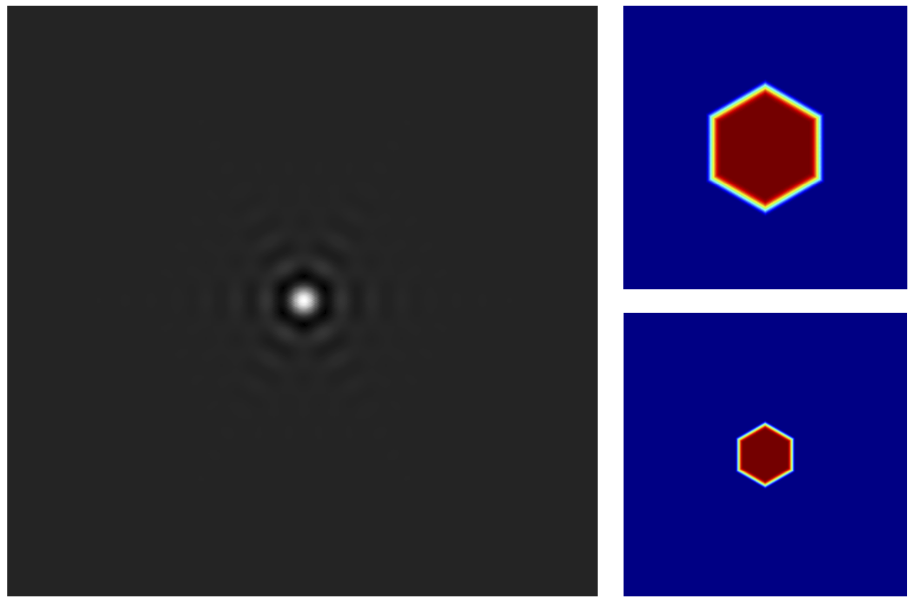}
        \caption{\small Scaling function $\phi_{\text{fr}}$ of the wavelet frame. }
        \label{fig:lvl0_frame_idx_0}
    \end{subfigure}
    \caption{\small The scaling function $\phi$  after boundary smoothing. \textbf{(a)} Shannon scaling function $\phi_{\text{sh}}$ (no smoothing.) \textbf{(b)} $\phi_{\text{ob1}}$ obtained from regular boundary smoothing in Section~\ref{sec:bdry_smoothing} (this is the same as $\phi_{\text{sh}}$ because $\A_0$ has no regular boundary.) \textbf{(c)} $\phi_{\text{ob2}}$ obtained from regular boundary and refinement filter smoothing in Section~\ref{sec:scaling_smoothing}. \textbf{(d)} Scaling function $\phi_{\text{fr}}$ of the wavelet frame in Section~\ref{sec:frame}. For every block of three figures, the large figure on the left shows $\phi$ in the spatial domain, the upper right figure displays the modulus $\M_0=|m_0|$ of the filter $m_0\in L^2(\R^2/\Lambda^*)$, and the lower right figure shows the modulus $|\widehat{\phi}|$ in the frequency domain, which is supported mainly on $A_0/2$.}
\label{fig:lvl0_idx_0}
\end{figure}

Since the boundary of $\A_0$ consists only of singular boundaries (see Figure~\ref{fig:bdry_classify}), the refinement filter $m_0$ (as well as the adjacent $(m_k)_{1\le k\le 6}$) cannot be continuously extended beyond $\partial\A_0$. Therefore $m_0(\xi) = \chi_{\A_0}(\xi)$, and the scaling function $\phi$ has slow spatial decay (see Figure~\ref{fig:lvl0_ON1_idx_0}.) One can partially remedy this by introducing some aliasing in $(m_k)_{1\le k\le 6}$. 

\begin{figure}
    \centering
    \begin{subfigure}[t]{0.48\textwidth}
        \includegraphics[height = 4cm]{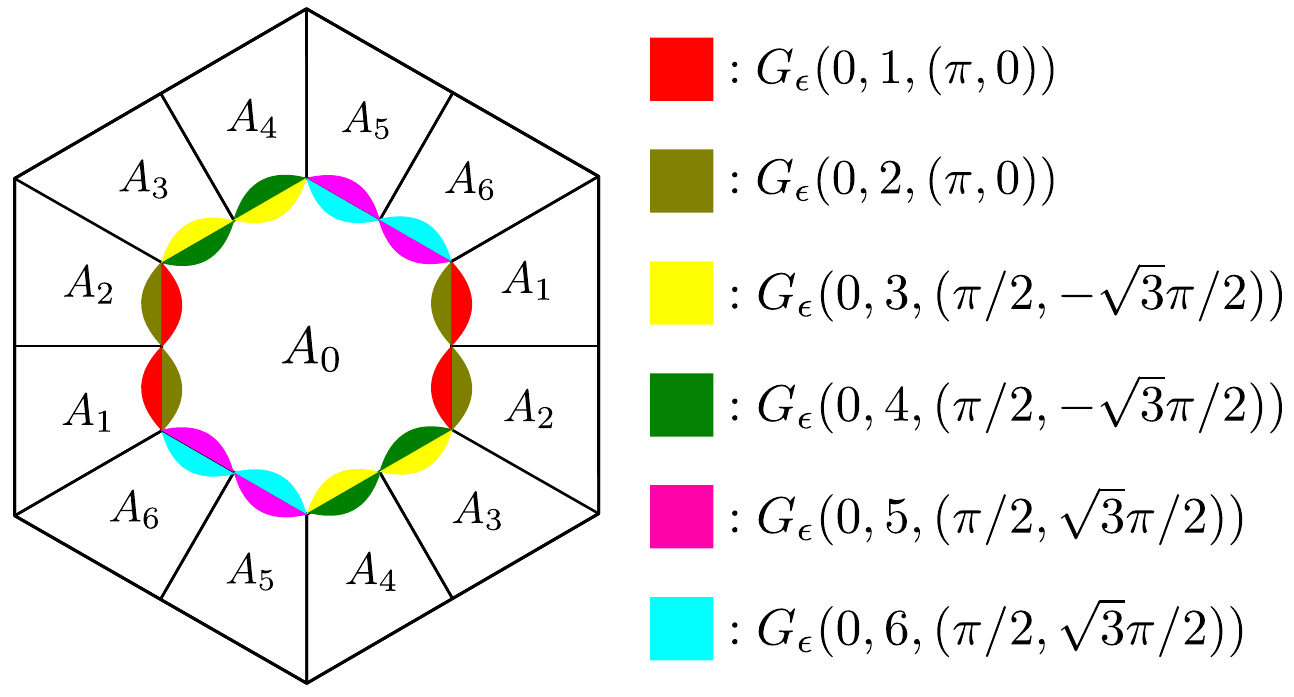}
        \caption{}
        \label{fig:G_epsilon}
    \end{subfigure}
     ~
    \begin{subfigure}[t]{0.24\textwidth}
        \includegraphics[height = 4cm]{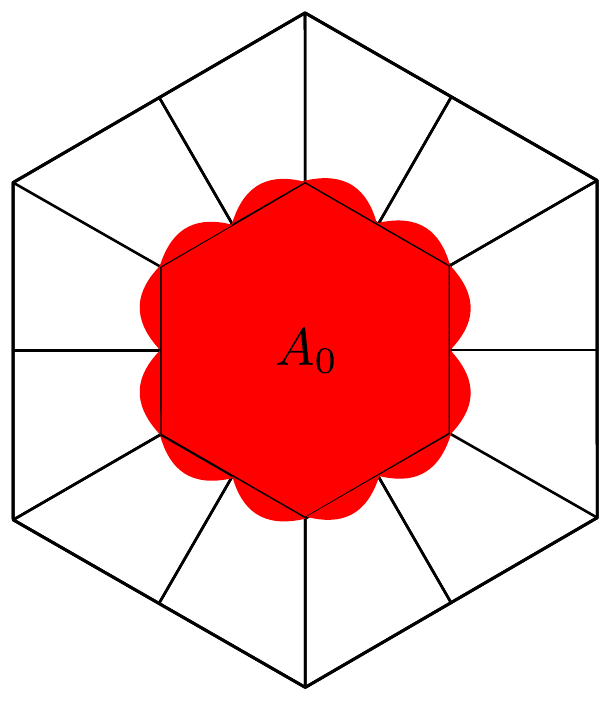}
        \caption{}
        \label{fig:spt_m0_alias}
    \end{subfigure}
     ~
    \begin{subfigure}[t]{0.24\textwidth}
        \includegraphics[height = 4cm]{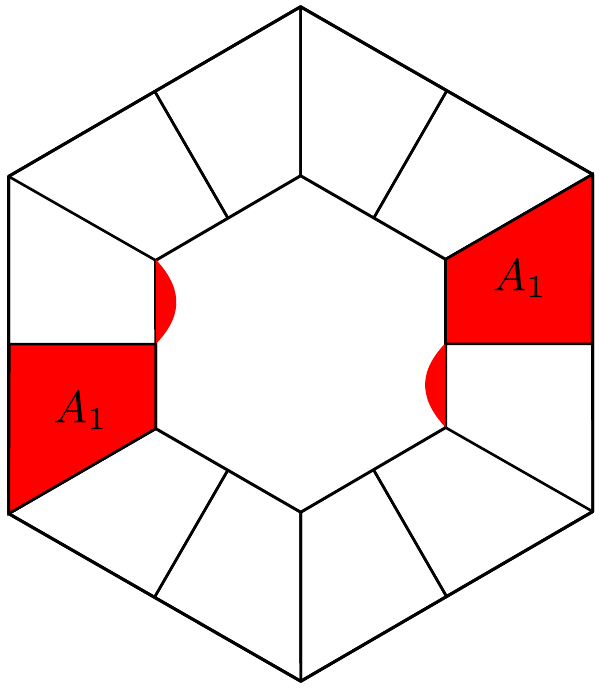}
        \caption{}
        \label{fig:spt_m1_alias}
    \end{subfigure}
    \caption{\small \textbf{(a)} Definition of $G_\epsilon(0,k,\gamma_k), 1\le k\le 6.$ \textbf{(b)} Support of $m_0$ after continuous extension into $\bigcup_{k=1}^6G_\epsilon(0,k,\gamma_k)$. \textbf{(c)} Support of $m_1$ after the refinement filter smoothing.}
\end{figure}

Define the sets $G_\epsilon(0,k,\gamma_k)$, $k = 1, \ldots, 6$, as shown in Figure~\ref{fig:G_epsilon}, where $\gamma_1= \gamma_2 = (\pi,0)$, $\gamma_3 = \gamma_4 = (\pi/2,-\sqrt{3}\pi/2)$, $\gamma_5 = \gamma_6 = (\pi/2,\sqrt{3}\pi/2)$, and each $G_\epsilon(0,k,\gamma_k)=\Omega_k\cup (\Omega_k+\gamma_k)$ is the disjoint union of $\Omega_k$ and $\Omega_k+\gamma_k$. We seek to continuously extend $m_0$ into the region $\bigcup_{k=1}^6G_\epsilon(0,k,\gamma_k)$ (see Figure~\ref{fig:spt_m0_alias}), which inevitably causes $m_k, 1\le k\le 6$, to be supported at least on $\A_k\cup G_\epsilon(0,k,\gamma_k)$ (see Figure~\ref{fig:spt_m1_alias}) because of the common singular boundary between $\A_0$ and $\A_k$. The following proposition explains how to achieve this.

\begin{prop}
  \label{prop:smooth_scaling_fct}
  Suppose $\left(m_k(\xi) = \M_k(\xi)e^{i\left<\xi,\eta_k\right>}\right)_{0\le k\le 6}$ satisfies the PR condition \eqref{eq:PR_hex_identity_sum}\eqref{eq:PR_hex_compact}, and
  \begin{align*}
    \left\{
    \begin{aligned}
      &\supp(m_k) \subset \overline{\A_k}\cup \left(\bigcup_{k' \text{s.t.}\atop R(k)\cap R(k')\neq \emptyset}\bigcup_{\gamma \in \left(\Gamma_k^* \cap \Gamma_{k'}^*\right)\setminus \Lambda^*}E_\epsilon(k',k,\gamma)  \right)\cup G_\epsilon(0,k,\gamma_k),\\
      &\supp(m_0) \subset \overline{\A_0}\cup \bigcup_{k=1}^6G_\epsilon(0,k,\gamma_k),\\
      & \left|G_\epsilon(0,k,\gamma_k)\cap E_\epsilon(k_1,k_2,\gamma)\right| = 0, ~\forall k_1\neq k_2, \gamma\in \left(\Gamma_{k_1}^*\cap \Gamma_{k_2}^*\right)\setminus\Lambda^*, 1\le k\le 6.
    \end{aligned}\right.
  \end{align*}
Suppose we change only the values of $\M_0(\xi)$ and $\M_k(\xi)$ on $G_\epsilon(0,k,\gamma_k) = \Omega_k \cup (\Omega_k+\gamma_k)$ such that
\begin{align*}
  \left\{
  \begin{aligned}
    &\M_{0}(\xi), \M_{k}(\xi)\ge 0,~~ \M_{0}^2(\xi) + \M_{k}^2(\xi) = 1,\\
    & \M_{0}(\xi+\gamma) = \M_{k}(\xi), ~~ \M_{k}(\xi+\gamma) = \M_{0}(\xi),
  \end{aligned}\right. \quad \forall \xi\in \Omega_k,
\end{align*}
and $\eta_0,\eta_k\in\Lambda^*$ satisfies
\begin{align}
  \label{eq:eta_smooth_scaling}
  e^{i\left<\gamma_k,\eta_{0}-\eta_{k}\right>} = -1,
\end{align}
then the new filter bank $\left(m_k(\xi) = \M_k(\xi)e^{i\left<\xi,\eta_k\right>}\right)_{0\le k\le 6}$ still satisfies the PR condition.
\end{prop}

Proposition~\ref{prop:smooth_scaling_fct} can be proved similarly as Proposition~\ref{prop:change_m1_m2} and~\ref{prop:omega_eta}. One can easily verify that the choice of $\left\{\eta_k\right\}_{0\le k\le 6}$ in \eqref{eq:eta_basis} satisfies the extra condition \eqref{eq:eta_smooth_scaling} for all $1\le k\le 6$. When iteratively extending $\M_0$ and $\M_k$ into $G_\epsilon(0,k,\gamma_k)$, we change the values $\M_0(\xi)$ and $\M_k(\xi)$ continuously from $1$ to $1/\sqrt{2}$ (or from $1/\sqrt{2}$ to $0$) while satisfying $\M_0(\xi)^2+\M_k(\xi)^2=1$. The resulting filters are still symmetric with respect to the origin and invariant by a $2\pi/3$ rotation. As can be seen in Figure~\ref{fig:lvl0_ON2_idx_0}, the scaling function $\phi$ after refinement filter smoothing has  better spatial localization as compared to its Shannon counterpart (Figure~\ref{fig:lvl0_shannon_idx_0} and~\ref{fig:lvl0_ON1_idx_0}.) The consequence of such smoothing is the unavoidable aliasing in the wavelet function $\psi^k$ (see Figure~\ref{fig:lvl0_ON2_idx_1}.)

\subsection{Alias-free low-redundancy wavelet frame}
\label{sec:frame}

As we have shown in Section~\ref{sec:basis}, the existence of singular boundaries in the six-direction admissible partition (Figure~\ref{fig:bdry_classify}) restricts us from obtaining alias-free hexagonal wavelet bases with continuous transfer functions. Such singular boundaries are caused by certain $\gamma\in \Gamma_k^*\setminus \Gamma_0^*$, $k=1, \ldots, 6$. For instance, the common singular boundary of $\A_0$ and $\A_1$ is caused by $\gamma = (\pi/2, 0)\in\Gamma_1^*\setminus \Gamma_0^*$. This motivates us to downsample the signal filtered by $m_k$, $k=1\ldots, 6$, on a denser lattice $\Gamma_k^{\text{fr}}\supset \Gamma_k$, thus having a coarser dual lattice $\Gamma_k^{\text{fr}*}\subset \Gamma_k^*$ not containing such $\gamma$. In particular, we set
\begin{align}
  \label{eq:gamma_frame}
  \Gamma_k^{\text{fr}} = 2\Lambda =
  \begin{bmatrix}
    2 & 0\\
    \frac{2}{\sqrt{3}} & \frac{4}{\sqrt{3}}
  \end{bmatrix}\Z^2 \eqqcolon \Gamma^{\text{fr}}.
\end{align}
If $\left( M_k, M_k, \Lambda\to\Gamma_k^{\text{fr}} \right)_{k=0}^6$ is a PR filter bank, then $\left\{\psi_{j,n}^k\right\}_{k\in [6],j\in\Z, n\in 2^{j-1}\Gamma_k^{\text{fr}}}$ defined by \eqref{eq:def_psi} will be a Parseval frame of $L^2(\R^2)$ by Theorem~\ref{thm:MRA-filter}. The frame redundancy at each level of decomposition is $\sum_{k=0}^61/|\Lambda/\Gamma_k^{\text{fr}}| = 7/4$, and the redundancy of the entire decomposition is
\begin{align*}
  \left(\sum_{k=1}^6\frac{1}{|\Lambda/\Gamma_k^{\text{fr}}|}\right)\cdot \frac{1}{1-1/|\Lambda/\Gamma_0^{\text{fr}}|} = 2.
\end{align*}
Similar to Proposition~\ref{prop:PR_hex}, we can simplify the PR condition of the filter bank $\left( M_k, M_k, \Lambda\to\Gamma_k^{\text{fr}} \right)_{k=0}^6$ as follows
\begin{prop}
  \label{prop:PR_hex_frame}
  The filter bank $(M_k, M_k, \Lambda \to \Gamma_k^{\text{\normalfont fr}} = \Gamma^{\text{\normalfont fr}})_{k=0}^6$ is PR if and only if the following two conditions hold  for $a.e.~\xi\in\R^2$
  \begin{align}
    \label{eq:PR_hex_identity_sum_frame}
    & \sum_{k=0}^6\left|m_k(\xi)\right|^2=1,\\
    \label{eq:PR_hex_shift_cancel_frame}
    &\sum_{k=0}^6\overline{m_k(\xi+\gamma)}m_k(\xi) = 0, \quad  \forall \gamma \in \Lambda^*\setminus \Gamma^{\text{\normalfont fr}*}
  \end{align}
where $m_k(\xi) = M_k(\xi)/\sqrt{|\Lambda/\Gamma_k^{\text{\normalfont fr}}|}$.
\end{prop}

Using the definition of singular/regular boundaries in Definition~\ref{def:singular_regular}, one can easily verify that all boundaries of $\A_0$ are singular boundaries, while $\A_k$, $k = 1,\ldots, 6$, has only regular boundaries. This implies that we may be able to extend $m_k$, $k = 1,\ldots, 6$, into the interior of $\A_0$, while $m_0$ has to be supported on $\A_0$. Since the critical-sampling condition, and in particular \eqref{eq:a_k_shift_admissible}, fails to hold for such choice of $\Gamma_k^{\text{fr}}$, we take a route different from the pairwise filter smoothing in Section~\ref{sec:bdry_smoothing} to construct wavelet functions whose Fourier transforms are supported on an $\epsilon$-neighborhood of $A_k$.

Given $\epsilon\in (0,\pi/(4+\sqrt{3}))$, define $\M_0\in L^2(\R^2/\Lambda^*)$ on the reciprocal cell $S$ to be a continuous function  satisfying
\begin{align}
  \label{eq:m0_frame_def}
  \left\{
  \begin{aligned}
    &0\le \M_0(\xi) \le 1, \quad \quad &&\forall \xi \in S\\
    &\M_0(\xi)= 0, \quad \quad && \forall \xi \in S\setminus A_0\\
    &\M_0(\xi) = 1, \quad \quad&&\forall \xi \in S_{\epsilon}\coloneqq \frac{\pi-2\epsilon}{2\pi}S\subset A_0
  \end{aligned}\right.
\end{align}
We further require  $\M_0$ to be invariant by a $\pi/3$-rotation (see Figure~\ref{fig:S_epsilon} for the definition of $S_\epsilon$ and $\M_0$ on the reciprocal cell $S$.)

\begin{figure}
    \centering
    \begin{subfigure}[t]{0.25\textwidth}
        \includegraphics[width = \textwidth]{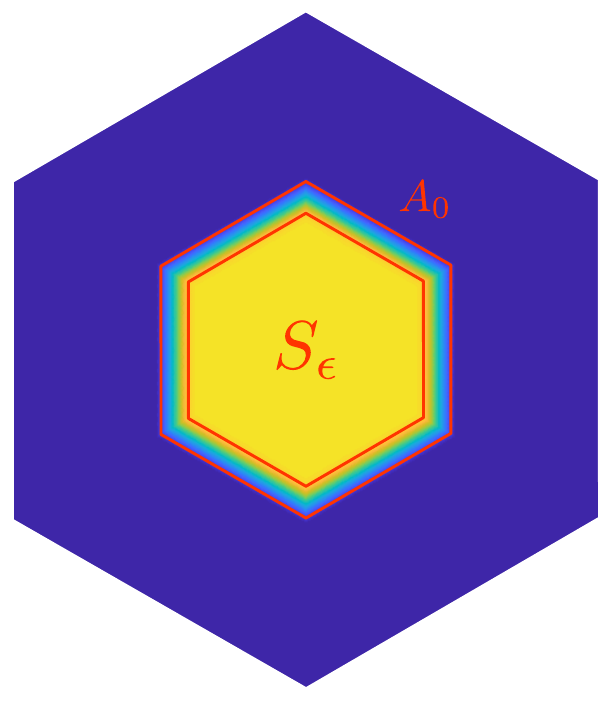}
        \caption{}
        \label{fig:S_epsilon}
    \end{subfigure}
     ~
    \begin{subfigure}[t]{0.275\textwidth}
        \includegraphics[width = \textwidth]{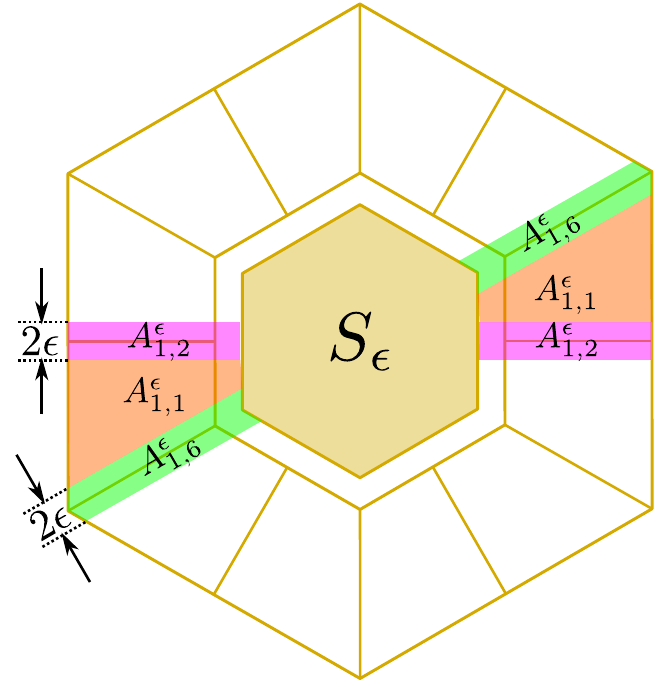}
        \caption{}
        \label{fig:A_1_epsilon}
    \end{subfigure}
    ~~~
    \begin{subfigure}[t]{0.25\textwidth}
        \includegraphics[width = \textwidth]{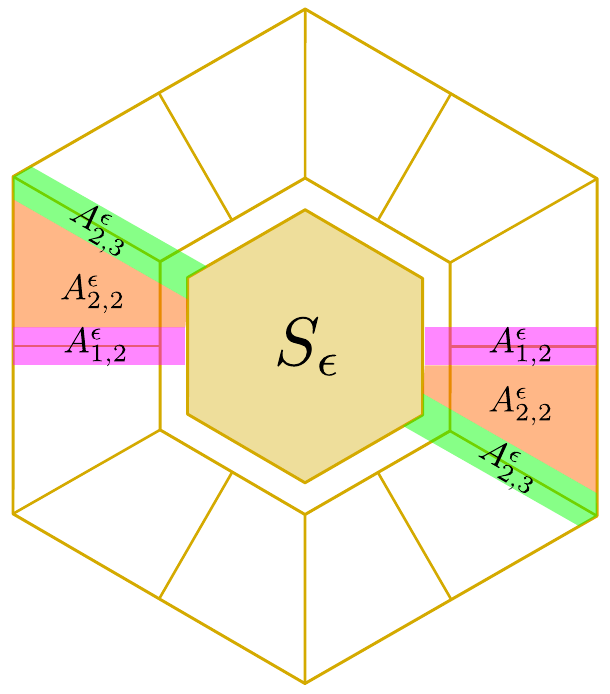}
        \caption{}
        \label{fig:A_2_epsilon}
    \end{subfigure}
    \caption{\small \textbf{(a)} $\M_0\in L^2(\R^2/\Lambda^*)$ is a continuous function supported on $A_0$ in the reciprocal cell $S$, and $\M_0(\xi) = 1, \forall \xi \in S_\epsilon \coloneqq \frac{\pi-2\epsilon}{2\pi}S\subset A_0$. \textbf{(b)} The $\epsilon$-neighborhood $A_1^\epsilon \coloneqq A_{1,2}^\epsilon \cup A_{1,6}^\epsilon \cup A_{1,1}^\epsilon$ of $A_1$. \textbf{(c)} The $\epsilon$-neighborhood $A_2^\epsilon$ of $A_2$ is symmetric to $A_1^\epsilon$ with respect to the $\xi_1$-axis.}
\end{figure}

We next define an $\epsilon$-neighborhood $\A_k^{\epsilon}$ of $\A_k$ on which the filter $m_k$ will be supported. In particular, for $k=1$, let
\begin{align}
  \label{eq:A_1_epsilon}
  \left\{
  \begin{aligned}
  &A_{1,2}^{\epsilon} \coloneqq \left\{\xi\in S\setminus S_\epsilon: \left| \xi_2\right|\le \epsilon  \right\},\\
  &A_{1,6}^{\epsilon} \coloneqq \left\{\xi\in S\setminus S_\epsilon: \frac{\sqrt{3}}{2}\left|\xi_2-\frac{\xi_1}{\sqrt{3}} \right|\le \epsilon  \right\},\\
  &A_{1,1}^{\epsilon} \coloneqq \left\{\xi\in S\setminus S_\epsilon: \xi_1\le 0, \frac{\xi_1+2\epsilon}{\sqrt{3}} \le \xi_2\le -\epsilon  \right\} \cup \left\{\xi\in S\setminus S_\epsilon: \xi_1\ge 0, \frac{\xi_1-2\epsilon}{\sqrt{3}} \ge \xi_2\ge \epsilon  \right\},\\
  &A_{1}^\epsilon \coloneqq A_{1,2}^\epsilon \cup A_{1,6}^\epsilon \cup A_{1,1}^\epsilon,\\
  &\A_1^\epsilon \coloneqq A_1^\epsilon + \Lambda^*.
  \end{aligned}\right.
\end{align}
See Figure~\ref{fig:A_1_epsilon} for the definition of the corresponding sets in \eqref{eq:A_1_epsilon}. The $\epsilon$-neighborhood $\A_2^\epsilon$ of $\A_2$ is defined to be symmetric to $\A_1^\epsilon$ with respect to the $\xi_1$-axis (see Figure~\ref{fig:A_2_epsilon}), and the remaining $\A_k^\epsilon$, $k=3, \ldots, 6$, are obtained from rotating $\A_1^\epsilon$ or $\A_2^\epsilon$ by $\pm \pi/3$.

Define $\N_1\in L^2(\R^2/\Lambda^*)$ on the reciprocal cell $S$ as follows
\begin{align}
  \label{eq:N_1_def}
  \N_1(\xi) = \left\{
  \begin{aligned}
    & 0, \quad && \xi \not\in A_1^\epsilon\\
    & 1, \quad && \xi \in A_{1,1}^\epsilon\\
    & \cos\left(\nu\left(\frac{\xi_2}{2\epsilon}+\frac{1}{2}\right)\cdot \frac{\pi}{2}\right),\quad && \xi \in A_{1,2}^\epsilon, ~\xi_1 < 0,\\
    & \cos\left(\nu\left(-\frac{\xi_2}{2\epsilon}+\frac{1}{2}\right)\cdot \frac{\pi}{2}\right),\quad && \xi \in A_{1,2}^\epsilon, ~\xi_1 > 0,\\
    & \cos\left(\nu\left[-\left(\xi_2-\frac{\xi_1}{\sqrt{3}}\right)\frac{\sqrt{3}}{4\epsilon}+\frac{1}{2} \right]\cdot \frac{\pi}{2} \right), \quad && \xi\in A_{1,6}^\epsilon, ~\xi_1<0,\\
    & \cos\left(\nu\left[\left(\xi_2-\frac{\xi_1}{\sqrt{3}}\right)\frac{\sqrt{3}}{4\epsilon}+\frac{1}{2} \right]\cdot \frac{\pi}{2} \right), \quad && \xi\in A_{1,6}^\epsilon, ~\xi_1>0,
  \end{aligned}\right.
\end{align}
where $\nu : \R \to \R$ is a continuous function defined as
\begin{align}
  \label{eq:nu_def}
  \nu(x) = \left\{
  \begin{aligned}
    &0,  \quad && x<0,\\
    &x, \quad && x\in [0,1],\\
    &1, \quad && x > 1.\\
  \end{aligned}\right.
\end{align}
In particular, $\nu(x) + \nu(1-x) = 1$, $\forall x\in [0,1]$. Similar to the definition of $\A_k^{\epsilon}$, we define $\N_2$ to be symmetric to $\N_1$ with respect to the $\xi_1$-axis, and $\N_k$, $k=3, \ldots, 6$, are rotated versions of $\N_1$ or $\N_2$ by $\pm\pi/3$. Define the filters $m_k\in L^2(\R^2/\Lambda^*)$ as
\begin{align}
  \label{eq:mk_frame_def}
  m_k(\xi) = \M_k(\xi)e^{i\left<\xi,\eta_k^{\text{fr}}\right>}, \quad \forall k\in \left\{ 0,\ldots,6 \right\},
\end{align}
where 
\begin{align}
  \label{eq:Mk_frame_eta}
  \left\{
  \begin{aligned}
    &\M_k(\xi) \coloneqq \N_k(\xi)\RR(\xi), \quad \forall k\in \left\{1, \ldots, 6 \right\}\\
    & \RR(\xi)\coloneqq \left(1-\M_0^2(\xi)\right)^{1/2},\\
    &\eta_0^{\text{fr}} = (0,0), ~\eta_1^{\text{fr}} = (2,0), ~\eta_2^{\text{fr}} = (1,\sqrt{3}), ~\eta_3^{\text{fr}} = (-1,\sqrt{3}),\\
    &\eta_4^{\text{fr}} = (-2,0), ~\eta_5^{\text{fr}} = (-1,-\sqrt{3}), ~\eta_6^{\text{fr}} = (1,-\sqrt{3}).
  \end{aligned}\right. 
\end{align}

\begin{figure}
    \centering
    \begin{subfigure}[t]{0.25\textwidth}
        \includegraphics[width = \textwidth]{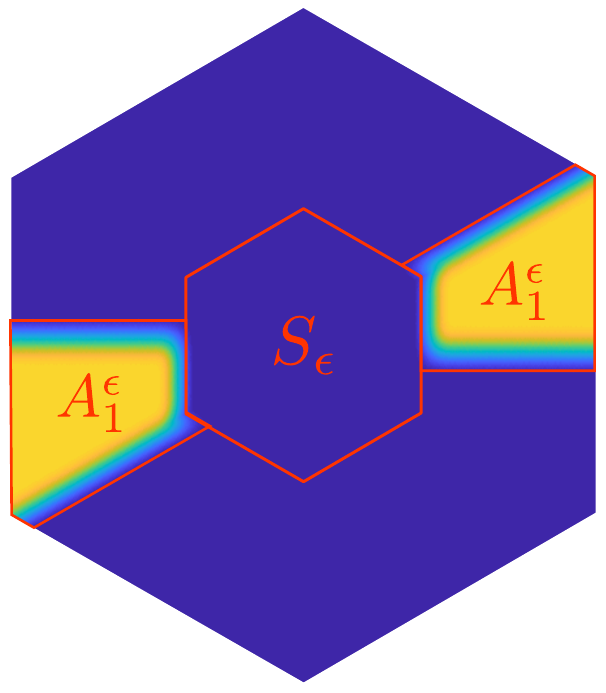}
        \caption{}
        \label{fig:m1_frame_parula}
    \end{subfigure}
     ~~
    \begin{subfigure}[t]{0.25\textwidth}
        \includegraphics[width = \textwidth]{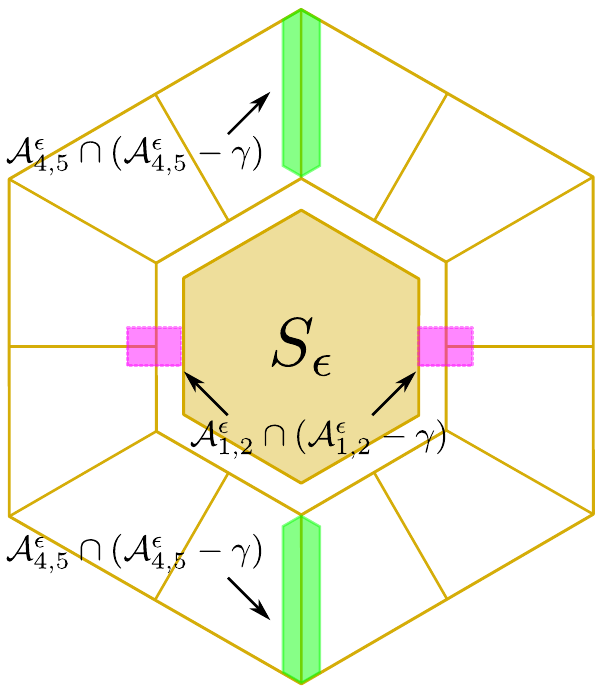}
        \caption{}
        \label{fig:proof_frame_PR}
    \end{subfigure}
    ~~
    \begin{subfigure}[t]{0.25\textwidth}
        \includegraphics[width = \textwidth]{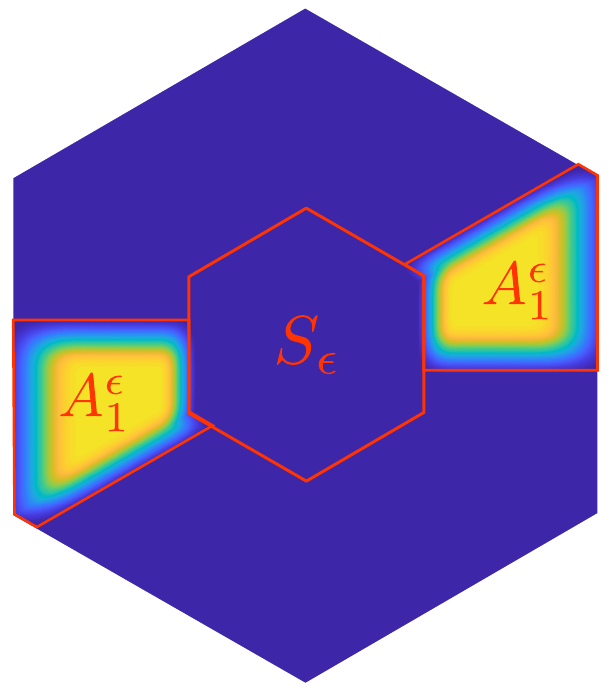}
        \caption{}
        \label{fig:psi_k_frame}
    \end{subfigure}
    \caption{\small \textbf{(a)} $\M_1\in L^2(\R^2/\Lambda^*)$ is supported on $A_1^\epsilon$ in the reciprocal cell $S$, and $\M_0(\xi) = 1, \forall \xi \in A_{1,1}^\epsilon$. \textbf{(b)} Visual illustration for the proof of Proposition~\ref{prop:PR_mk_frame}. \textbf{(c)} $\widehat{\psi_{1}^1}$ is a continuous function, and $\supp\left(\widehat{\psi_{1}^1}\right) = A_1^\epsilon$.}
\end{figure}

Figure~\ref{fig:m1_frame_parula} displays the definition of $\M_1(\xi)$ on the reciprocal cell $S$.  The following proposition shows that $\left(m_k\right)_{0\le k \le 6}$ satisfies the PR condition \eqref{eq:PR_hex_identity_sum_frame} and \eqref{eq:PR_hex_shift_cancel_frame}.
\begin{prop}
  \label{prop:PR_mk_frame}
  The normalized filter bank $\left(m_k\right)_{0\le k\le 6}$ defined in \eqref{eq:mk_frame_def} satisfies the PR condition \eqref{eq:PR_hex_identity_sum_frame} and \eqref{eq:PR_hex_shift_cancel_frame}. Moreover,
  \begin{align}
    \supp(m_0) = S_\epsilon + \Lambda^*\eqqcolon \CS^\epsilon, \quad \text{and}~~ \supp(m_k) = \A_k^\epsilon, ~~\forall k\in \left\{ 1, \ldots, 6 \right\}.
  \end{align}
\end{prop}
\begin{proof}
  The support of $m_k$, $k=0, \ldots, 6$, is clear from the definition \eqref{eq:m0_frame_def}, \eqref{eq:N_1_def} and \eqref{eq:Mk_frame_eta}. We next prove the identity summation condition \eqref{eq:PR_hex_identity_sum_frame} and the shift cancellation condition \eqref{eq:PR_hex_shift_cancel_frame}.
  \begin{itemize}
  \item \textit{Identity summation.} Given $\xi$ in the reciprocal cell $S$:
    \begin{enumerate}
    \item If $\xi \in S_\epsilon$, then $\sum_{k=0}^6\left|m_k(\xi)\right|^2 = \M_0^2(\xi) = 1$.
    \item If $\xi \in A_{j,j}^\epsilon$, then $\sum_{k=0}^6\left|m_k(\xi)\right|^2 = \M_j^2(\xi)+\M_0^2(\xi) = \N_j^2(\xi)\RR^2(\xi)+ \M_0^2(\xi)=1$.
    \item If $\xi \in A_{j,l}^\epsilon$, where $A_j$ and $A_l$ are adjacent, we consider without loss of generality the following two particular cases:
      \begin{enumerate}[(a)]
      \item $\xi \in A_{1,2}^\epsilon$ and $\xi_1<0$. In this case,
        \begin{align*}
          &\N_1^2(\xi) + \N_2^2(\xi) = \cos^2\left(\nu\left(\frac{\xi_2}{2\epsilon}+\frac{1}{2}\right)\cdot \frac{\pi}{2}\right) + \cos^2\left(\nu\left(-\frac{\xi_2}{2\epsilon}+\frac{1}{2}\right)\cdot \frac{\pi}{2}\right)\\
          = & \cos^2\left(\nu\left(\frac{\xi_2}{2\epsilon}+\frac{1}{2}\right)\cdot \frac{\pi}{2}\right) + \sin^2\left(\nu\left(\frac{\xi_2}{2\epsilon}+\frac{1}{2}\right)\cdot \frac{\pi}{2}\right) = 1,
        \end{align*}
where the second equation holds because $\nu(x)+\nu(1-x)=1$. We thus have
\begin{align*}
  \sum_{k=0}^6\left|m_k(\xi)\right|^2 = \M_0^2(\xi) + \RR^2(\xi) \left(\N_1^2(\xi) + \N_2^2(\xi)\right) = 1.
\end{align*}
        \item $\xi \in A_{1,6}^\epsilon$ and $\xi_1<0$. We can similarly prove that
          \begin{align*}
            \sum_{k=0}^6\left|m_k(\xi)\right|^2 = \M_0^2(\xi) + \RR^2(\xi) \left(\N_1^2(\xi) + \N_6^2(\xi)\right) = 1.          
          \end{align*}
      \end{enumerate}
    \end{enumerate}
  \item \textit{Shift cancellation.} Consider without loss of generality the particular case $\gamma = (\pi,0)\in\Lambda^*\setminus \Gamma^{\text{fr}*}$. The sum $\sum_{k=0}^6\overline{m_k(\xi+\gamma)}m_k(\xi)$ can only be supported on $\A_{1,2}^\epsilon\cap (\A_{1,2}^\epsilon-\gamma)$ and $\A_{4,5}^\epsilon\cap (\A_{4,5}^\epsilon-\gamma)$ (see Figure~\ref{fig:proof_frame_PR}.) We show only the case when $\xi\in S\cap \A_{1,2}^\epsilon\cap (\A_{1,2}^\epsilon-\gamma) \cap \left\{ \xi_1\le 0\right\}$, i.e., the pink region on the left in Figure~\ref{fig:proof_frame_PR}, while the remaining can be proved similarly.
    \begin{align*}
      &\sum_{k=0}^6m_k(\xi)\overline{m_k(\xi+\gamma)} = m_1(\xi)\overline{m_1(\xi+\gamma)} + m_2(\xi)\overline{m_2(\xi+\gamma)}\\
      = & \M_1(\xi)e^{i\left<\eta_1^\text{fr},\xi\right>}\M_1(\xi+\gamma)e^{-i\left<\eta_1^\text{fr},\xi+\gamma\right>} + \M_2(\xi)e^{i\left<\eta_2^\text{fr},\xi\right>}\M_2(\xi+\gamma)e^{-i\left<\eta_2^\text{fr},\xi+\gamma\right>}\\
      = & \RR(\xi)\RR(\xi+\gamma)\N_1(\xi)\N_1(\xi+\gamma)e^{-i\left<\eta^\text{fr}_1,\gamma \right>} + \RR(\xi)\RR(\xi+\gamma)\N_2(\xi)\N_2(\xi+\gamma)e^{-i\left<\eta^\text{fr}_2,\gamma \right>}\\
      = & \RR(\xi)\RR(\xi+\gamma)\left\{\cos\left[\nu\left(\frac{\xi_2}{2\epsilon}+\frac{1}{2}\right)\cdot \frac{\pi}{2}\right] \cos\left[\nu\left(-\frac{\xi_2}{2\epsilon}+\frac{1}{2}\right)\cdot \frac{\pi}{2}\right]e^{-i\left<\eta^\text{fr}_1,\gamma \right>}  \right.\\
      &\hspace{7em}\left. + \cos\left[\nu\left(-\frac{\xi_2}{2\epsilon}+\frac{1}{2}\right)\cdot \frac{\pi}{2}\right] \cos\left[\nu\left(\frac{\xi_2}{2\epsilon}+\frac{1}{2}\right)\cdot \frac{\pi}{2}\right]e^{-i\left<\eta^\text{fr}_2,\gamma \right>} \right\}\\
      = & \RR(\xi)\RR(\xi+\gamma)\cos\left[\nu\left(\frac{\xi_2}{2\epsilon}+\frac{1}{2}\right)\cdot \frac{\pi}{2}\right] \cos\left[\nu\left(-\frac{\xi_2}{2\epsilon}+\frac{1}{2}\right)\cdot \frac{\pi}{2}\right]e^{-i\left<\eta^\text{fr}_1,\gamma \right>}\left(1+ e^{-i\left<\eta^\text{fr}_2-\eta^\text{fr}_1,\gamma \right>}\right)\\
      = & 0,
    \end{align*}
    where the last equation holds because
    \begin{align*}
      1+ e^{-i\left<\eta^\text{fr}_2-\eta^\text{fr}_1,\gamma \right>} = 1 + e^{-i\left<(1,\sqrt{3})-(2,0),(\pi,0) \right>} = 0
    \end{align*}
  \end{itemize}
This concludes the proof of the proposition.
\end{proof}

The following theorem shows that the filter bank $(M_k, M_k, \Lambda \to \Gamma_k^{\text{\normalfont fr}})_{k=0}^6$ generates alias-free hexagonal wavelet frames with continuous Fourier transforms.
\begin{thm}
  Given $\epsilon\in(0,\pi/(4+\sqrt{3}))$, the PR filter bank $(M_k, M_k, \Lambda \to \Gamma_k^{\text{\normalfont fr}})_{k=0}^6$ defined in \eqref{eq:gamma_frame} and \eqref{eq:mk_frame_def} generates a Parseval frame $\left\{\psi_{j,n}^k \right\}_{1\le k\le 6,j\in\Z,n\in 2^{j-1}\Gamma_k^{\text{fr}}}$. Moreover, the scaling function $\phi$ and wavelets $\left\{\psi^k\right\}_{1\le k\le 6}$ satisfy the following properties
  \begin{enumerate}
  \item Both $\phi$ and $\psi^k$ are well-localized in the frequency domain. More specifically, $\supp\left(\widehat{\phi_{1}}\right) = A_0$, and $\supp\left(\widehat{\psi^k_{1}}\right) = A_k^\epsilon, ~\forall k = 1, \ldots, 6$, where $A_k^\epsilon$ is the $\epsilon$-neighborhood of $A_k$ defined in \eqref{eq:A_1_epsilon}.
  \item The Fourier transforms $\widehat{\phi}$,  $\widehat{\psi^k}$ are continuous.
  \item The scaling function $\phi$ and the set $\left\{\left(\psi^1,\psi^2 \right), \left(\psi^3,\psi^4 \right), \left(\psi^5,\psi^6 \right)  \right\}$ are invariant by a $2\pi/3$ rotation.
  \end{enumerate}
\end{thm}
\begin{proof}
  \begin{enumerate}
  \item Since $\epsilon<\pi/(4+\sqrt{3})$, $\M_0(\xi) = 1$ on $S_\epsilon$, and $\M_0(\xi) = 0$ on $S\setminus A_0$ \eqref{eq:m0_frame_def}, we have
    \begin{align}
      \label{eq:spt_phi_frame}
      \widehat{\phi_{1}}(\xi)&= 2\widehat{\phi}(2\xi) = 2\prod_{p=0}^\infty \frac{1}{2}M_0(2^{-p}\xi) = 2\prod_{p=0}^\infty m_0(2^{-p}\xi) = \left\{
        \begin{aligned}
          &2m_0(\xi), &&\quad \xi \in A_0,\\
          &0, &&\quad \xi \not \in A_0.
        \end{aligned}\right.
    \end{align}
    \begin{align}
      \label{eq:spt_psi_k_frame}
      \widehat{\psi_{1}^k}(\xi) = 2\widehat{\psi^k}(2\xi) = \widehat{\phi_{1}}\left(\xi/2\right)m_k(\xi) = \left\{
      \begin{aligned}
        &2m_0(\xi/2)m_k(\xi), &&\quad \xi\in A^\epsilon_k,\\
        & 0, &&\quad \xi \not \in A^\epsilon_k.
      \end{aligned}\right.
    \end{align}
Thus $\supp\left(\widehat{\phi_{1}}\right) = A_0$, and $\supp\left(\widehat{\psi^k_{1}}\right) = A_k^\epsilon, ~\forall k = 1, \ldots, 6$.
  \item It is clear from \eqref{eq:spt_phi_frame} and \eqref{eq:spt_psi_k_frame} that $\widehat{\phi_{1}}$ and $\widehat{\psi_{1}^k}$ are continuous on their support $A_0$ and $A_k^\epsilon$. Moreover, it is easy to verify from the definition of $\M_0$ \eqref{eq:m0_frame_def} and $\M_k$ \eqref{eq:N_1_def}\eqref{eq:Mk_frame_eta}
that  $\widehat{\phi_{1}}$ and $\widehat{\psi_{1}^k}$ both vanish on the boundary of their corresponding supports (see Figure~\ref{fig:psi_k_frame}).
\item The rotation invariance of $\phi$ and  $\left\{\left(\psi^1,\psi^2 \right), \left(\psi^3,\psi^4 \right), \left(\psi^5,\psi^6 \right)  \right\}$ is guaranteed by the rotation invariance of the moduli $\M_0$, $\left\{\left(\\M_1,\M_2 \right), \left(\M_3,\M_4 \right), \left(\M_5,\M_6 \right)  \right\}$ and the phases  $e^{\left<\eta_0^{\text{fr}},\cdot\right>}$, $\{(e^{\left<\eta^{\text{fr}}_1,\cdot\right>},e^{\left<\eta^{\text{fr}}_2,\cdot\right>}), (e^{\left<\eta^{\text{fr}}_3,\cdot\right>},e^{\left<\eta^{\text{fr}}_4,\cdot\right>}), (e^{\left<\eta^{\text{fr}}_5,\cdot\right>},e^{\left<\eta^{\text{fr}}_6,\cdot\right>})\}$.
  \end{enumerate}
\end{proof}

\begin{remark}
  \normalfont Given any positive integer $p\in\Z$, one can easily modify the construction of $\M_k$ to obtain $\widehat{\phi}, \widehat{\psi^k}\in C^p(\R^2,\mathbb{C})$. More specifically, we can replace $\M_0$ by a $C^p$ function satisfying \eqref{eq:m0_frame_def}, and redefine $\nu$ in \eqref{eq:nu_def} on the interval $[0,1]$ such that it is $C^p$ and satisfies $\nu(x) + \nu(1-x) = 1$, $\forall x\in [0,1]$. Therefore $\phi$ and $\psi^k$ can be constructed to have arbitrarily fast spatial decay.
\end{remark}

Figure~\ref{fig:lvl0_frame_idx_1} and~\ref{fig:lvl0_frame_idx_0} display the wavelet function $\psi^1_{\text{fr}}$ and the scaling function $\phi_{\text{fr}}$ constructed in this section. Compared to their wavelet bases counterparts (Figure~\ref{fig:lvl0_ON1_idx_1},~\ref{fig:lvl0_ON2_idx_1} and Figure~\ref{fig:lvl0_ON1_idx_0},~\ref{fig:lvl0_ON2_idx_0}) explained in Section~\ref{sec:bdry_smoothing} and~\ref{sec:scaling_smoothing}, $\phi_{\text{fr}}$ and $\psi_{\text{fr}}^1$ have much better spatial and frequency localization, although the wavelet frame comes with a sacrifice of being slightly redundant (of factor 2.)

\section{Parabolic scaling law and cutting lemma}
\label{sec:cutting}

\begin{figure}
    \centering
    \begin{subfigure}[t]{0.24\textwidth}
        \includegraphics[height = 4cm]{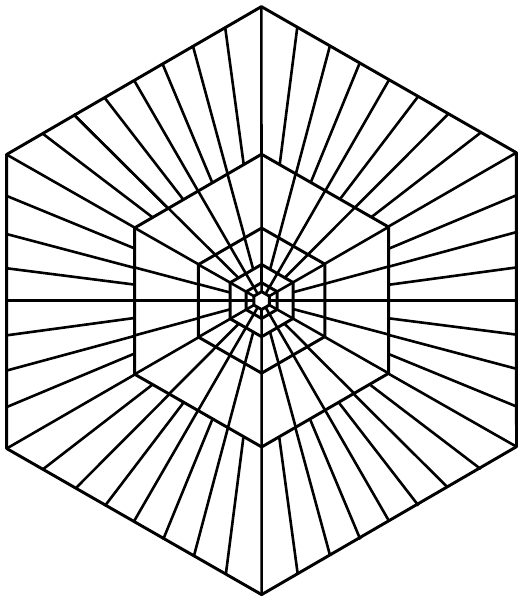}
        \caption{}
        \label{fig:parabolic}
    \end{subfigure}
     ~
    \begin{subfigure}[t]{0.72\textwidth}
        \includegraphics[height = 4cm]{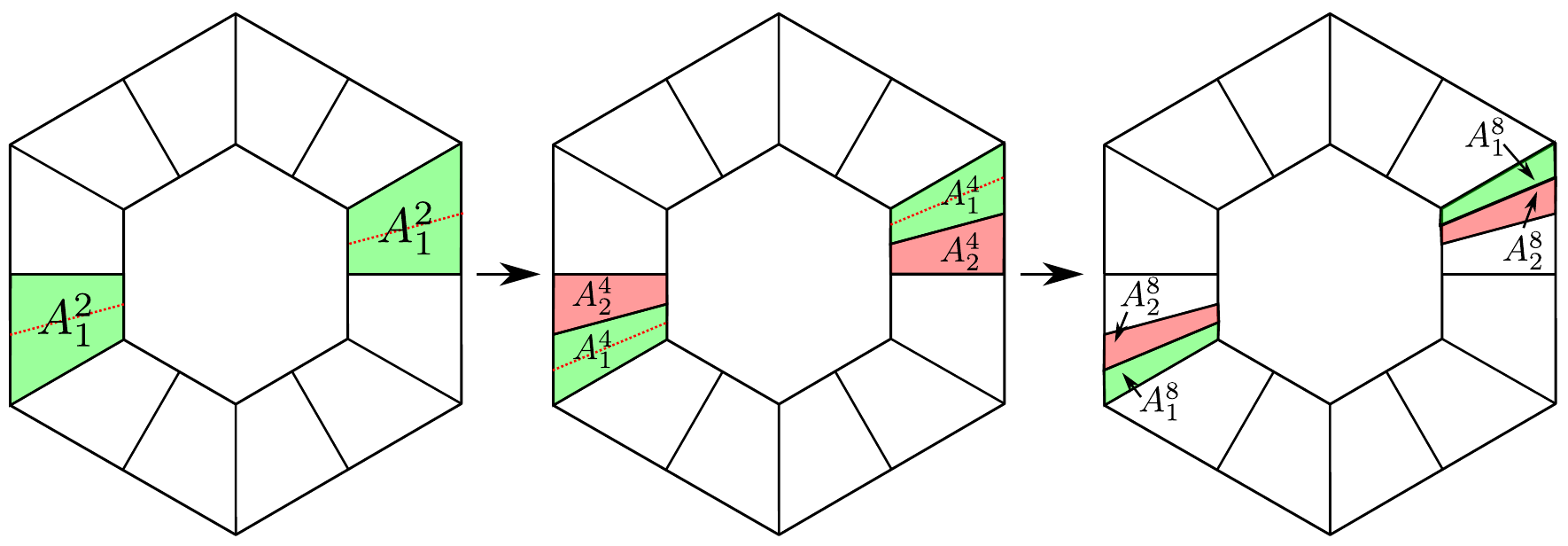}
        \caption{}
        \label{fig:cutting-demo}
    \end{subfigure}
    \caption{\small \textbf{(a)} Parabolic scaling law of the hexagonal wavelets. The number of orientations is doubled after every other dilation step. \textbf{(b)} Cutting the wavelet function $\widehat{\psi^1_{1}}$ supported mainly on $A_1^2$ into $\widehat{\psi^{1,k}_{1}}$ supported essentially on $A^4_k$, $k=1,2$. The wavelet $\widehat{\psi^{1,1}_{1}}$ can be further cut into $\widehat{\psi^{1,1,k}_{1}}$ supported mainly on $A^8_k$, $k=1,2$. }
\end{figure}

We discuss next how to obtain finer orientation selectivity in higher frequency rings to achieve the ``parabolic scaling'' law that is crucial in the construction of curvelets, shearlets, and contourlets. In essence, one needs to double the number of orientations every other dilation step by subdividing each $A_k^p$, $1\le k\le 3p$ (see Figure~\ref{fig:parabolic}.) The  techniques detailed in Section~\ref{sec:basis} and~\ref{sec:frame} can be easily generalized to such frequency partitions, but there is one caveat for the frame construction: we have to keep subsampling the filtered signals on the dense sublattice $\Gamma^{\text{fr}}$ in \eqref{eq:gamma_frame} to avoid introducing singular boundaries (see Section~\ref{sec:frame}.) This unfortunately increases the frame redundancy in higher decomposition levels (similar to the shearlet construction in \cite{easley2008sparse}.) In order to keep the redundancy of the wavelet frame (and basis) unchanged, we consider instead using the following 2D generalization of the cutting lemma from \cite{daubechies1992ten}:

\begin{lemma}
\label{lemma:cutting_basis}
  Suppose $\left\{\psi(\cdot - n)\right\}_{n\in\Lambda}$ is an orthonormal basis of $V\subset L^2(\R^d)$, and $\left(M_k,M_k,\Lambda\to\Gamma \right)_{k=1,2}$ is a 2-band PR filter bank, where $|\Lambda/\Gamma|=2$. Define $\psi^1$ and $\psi^2$ by
  \begin{align}
    \widehat{\psi^1}(\xi) = M_1(\xi)\widehat{\psi}(\xi), \quad \text{and}~~ \widehat{\psi^2}(\xi) = M_2(\xi)\widehat{\psi}(\xi),
  \end{align}
and let $V_1$ and $V_2$, respectively, be the subspaces of $V$ spanned by the $\Gamma$-translates of $\psi^1$ and $\psi^2$. Then the following hold:
\begin{itemize}
\item $\left\{\psi^1(\cdot-n)\right\}_{n\in \Gamma}$ and $\left\{\psi^2(\cdot-n)\right\}_{n\in \Gamma}$, respectively, are orthonormal bases of $V^1$ and $V^2$.
\item $V = V^1 \obot V^2$.
\end{itemize}
\end{lemma}
The 1D analogue of this lemma is one of the basic tools in building wavelet bases from multiresolution ladders of $L^2(\R)$. Similarly we have the following cutting lemma for wavelet frames:

\begin{lemma}
\label{lemma:cutting_frame}
  Suppose $\left\{\psi(\cdot - n)\right\}_{n\in\Lambda}\cup F$ constitutes a Parseval frame of $L^2(\R^2)$, and $\left(M_k,M_k,\Lambda\to\Gamma \right)_{k=1,2}$ is a 2-band PR filter bank, where $|\Lambda/\Gamma|=2$. Define $\psi^1$ and $\psi^2$ by
  \begin{align}
    \widehat{\psi^1}(\xi) = M_1(\xi)\widehat{\psi}(\xi), \quad \text{and}~~ \widehat{\psi^2}(\xi) = M_2(\xi)\widehat{\psi}(\xi).
  \end{align}
Then $\left\{\psi^1(\cdot-n)\right\}_{n\in \Gamma}\cup\left\{\psi^2(\cdot-n)\right\}_{n\in \Gamma} \cup F$ is also a Parseval frame of $L^2(\R^2)$.
\end{lemma}

If the wavelet function $\psi$ satisfies $\supp\left(\widehat{\psi}\right)\approx A\subset \R^2$, and $M_1,M_2\in L^2(\R^2/\Lambda^*)$ are supported mainly on $\B_1$ and $\B_2$ respectively, then Lemma~\ref{lemma:cutting_basis} and~\ref{lemma:cutting_frame} suggest that we can ``cut'' the wavelet function $\psi$ in the frequency domain by $\widehat{\psi^k}(\xi)=M_k(\xi)\widehat{\psi}(\xi)$ such that $\supp\left(\widehat{\psi^k}\right)\approx A\cap \B_k$, and the $\Gamma$-translates of $\psi^1$ and $\psi^2$ (together with other unchanged functions) still constitute an orthonormal basis or a Parseval frame of $L^2(\R^2)$. In what follows, we discuss, without loss of generality, how to cut the wavelet functions constructed in Section~\ref{sec:basis_and_frame} whose Fourier transforms are supported mainly on $A_1^2$ (see Figure~\ref{fig:cutting-demo}.)

\begin{figure}
    \centering
    \begin{subfigure}[t]{0.23\textwidth}
        \includegraphics[width=\textwidth]{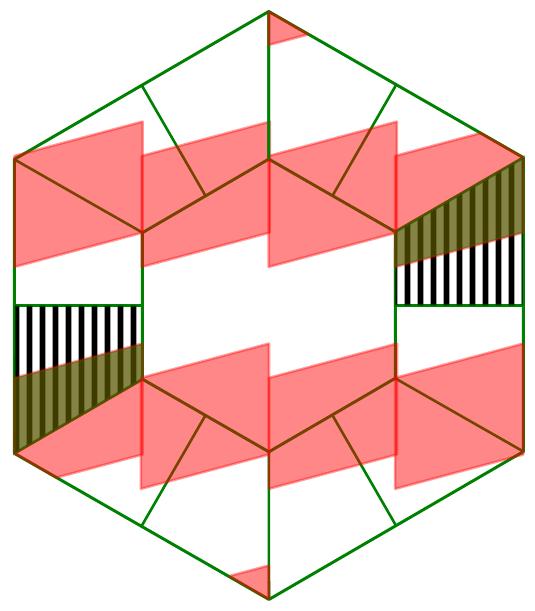}
        \caption{\small $\chi_{\B_{1,1}}\in L^2(\R^2/\Gamma_1^{2*})$}
        \label{fig:cutting-basis-1}
    \end{subfigure}
     ~
    \begin{subfigure}[t]{0.23\textwidth}
        \includegraphics[width=\textwidth]{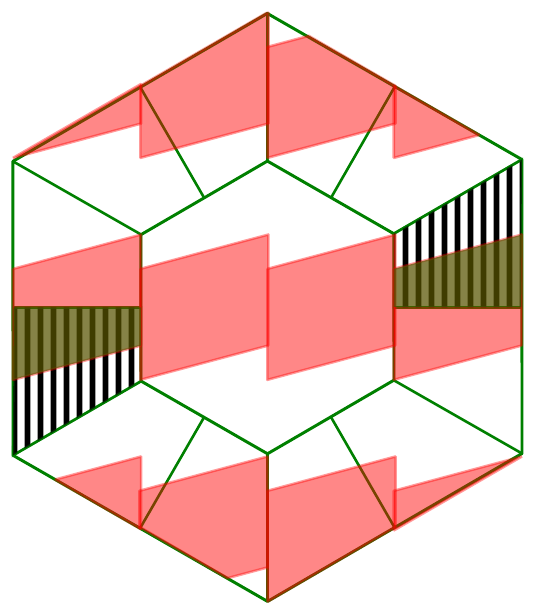}
        \caption{\small $\chi_{\B_{1,2}}\in L^2(\R^2/\Gamma_1^{2*})$}
        \label{fig:cutting-basis-2}
    \end{subfigure}
     ~
    \begin{subfigure}[t]{0.23\textwidth}
        \includegraphics[width=\textwidth]{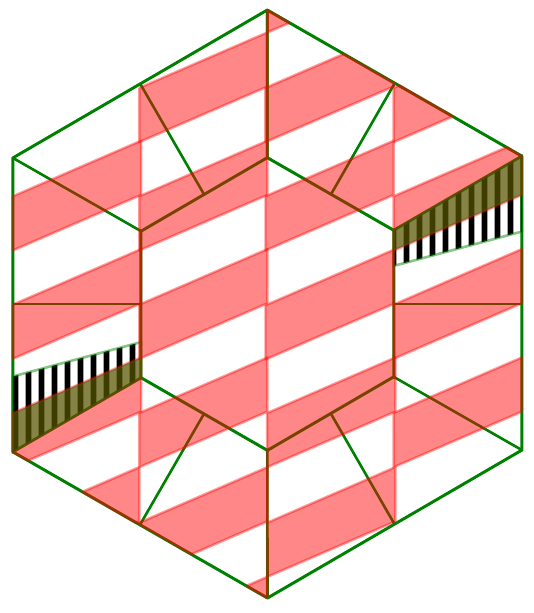}
        \caption{\small $\chi_{\B_{2,1}}\in L^2(\R^2/\Gamma_1^{4*})$}
        \label{fig:cutting-basis-3}
    \end{subfigure}
     ~
    \begin{subfigure}[t]{0.23\textwidth}
        \includegraphics[width=\textwidth]{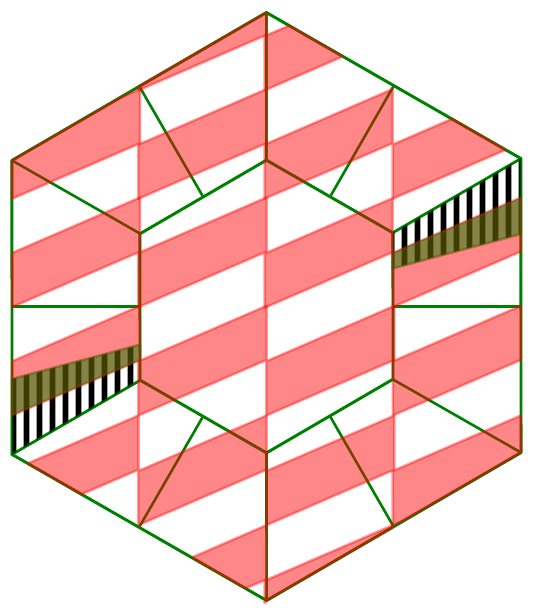}
        \caption{\small $\chi_{\B_{2,2}}\in L^2(\R^2/\Gamma_1^{4*})$}
        \label{fig:cutting-basis-4}
    \end{subfigure}
    \begin{subfigure}[t]{0.23\textwidth}
        \includegraphics[width=\textwidth]{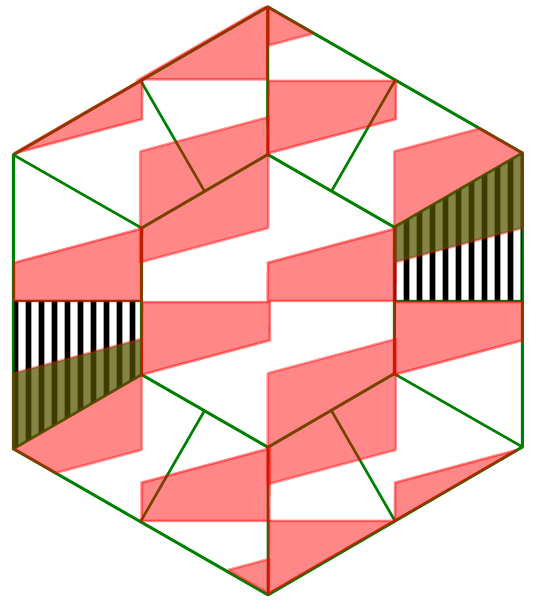}
        \caption{\small $\chi_{\B_{1,1}^{\text{fr}}}\in L^2(\R^2/\Gamma_1^{\text{fr},2*})$}
        \label{fig:cutting-frame-1}
    \end{subfigure}
     ~
    \begin{subfigure}[t]{0.23\textwidth}
        \includegraphics[width=\textwidth]{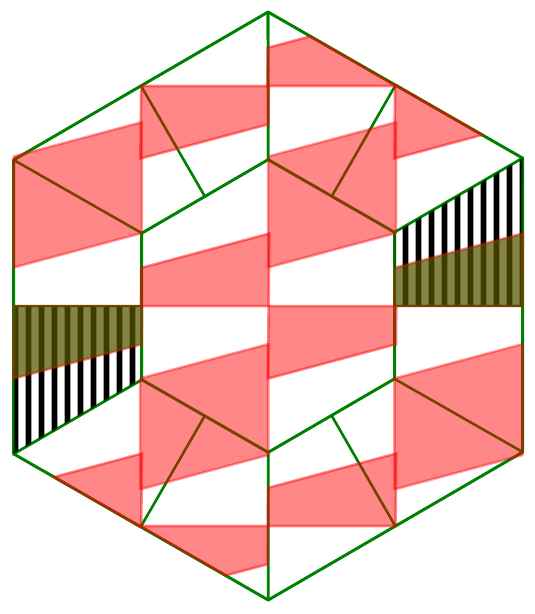}
        \caption{\small $\chi_{\B_{1,2}^{\text{fr}}}\in L^2(\R^2/\Gamma_1^{\text{fr},2*})$}
        \label{fig:cutting-frame-2}
    \end{subfigure}
     ~
    \begin{subfigure}[t]{0.23\textwidth}
        \includegraphics[width=\textwidth]{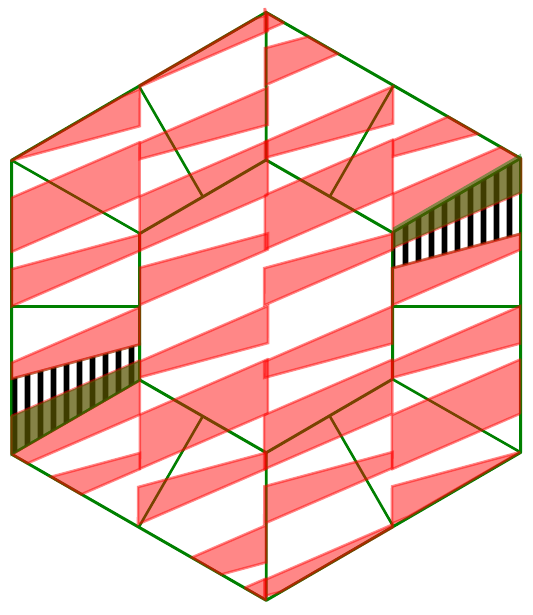}
        \caption{\small $\chi_{\B_{2,1}^{\text{fr}}}\in L^2(\R^2/\Gamma_1^{\text{fr},4*})$}
        \label{fig:cutting-frame-3}
    \end{subfigure}
     ~
    \begin{subfigure}[t]{0.23\textwidth}
        \includegraphics[width=\textwidth]{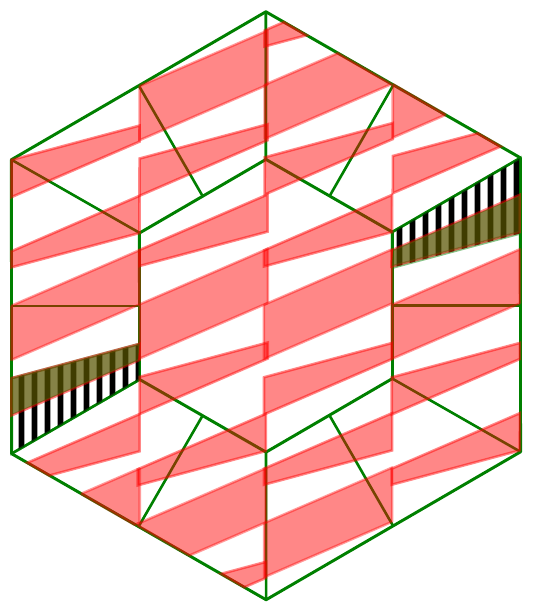}
        \caption{\small $\chi_{\B_{2,2}^{\text{fr}}}\in L^2(\R^2/\Gamma_1^{\text{fr},4*})$}
        \label{fig:cutting-frame-4}
    \end{subfigure}
    \caption{ \small The characteristic functions $\chi_{\B_k}$ which after smoothing are used to cut the wavelet basis/frame functions. Red: the essential support of the filters $\B_k$. Striped: the essential support of the original wavelet functions. Green: the essential support of the wavelet functions after the cutting.}
\end{figure}

Let $\psi^1$ be the  wavelet basis function obtained in Section~\ref{sec:basis} after regular boundary smoothing. Thus $\supp\left(\widehat{\psi^1_{1}}\right)\approx A_1^2$, and $\left\{\psi_{1}^1(\cdot - n)=\psi^1_{1,n} \right\}_{n\in \Gamma_1^{2}}$ constitutes an orthonormal basis for  $V_1 = \widebar{\Span}\left\{\psi_{1,n}^1 \right\}_{n\in \Gamma_1^{2}}$. Let $\B_{1,1}$ and $\B_{1,2}$ be the sets colored in red in Figure~\ref{fig:cutting-basis-1} and~\ref{fig:cutting-basis-2}, then $\chi_{\B_{1,1}}, \chi_{\B_{1,2}} \in L^2(\R^2/\Gamma_1^{2*})$, and $\chi_{\B_{1,2}}(\xi) = \chi_{\B_{1,1}}(\xi+\gamma_1)$, where $\gamma_1 = (0,\sqrt{3}\pi/4)\in \Gamma_1^{4*}/\Gamma_1^{2*}=\Gamma_2^{4*}/\Gamma_1^{2*}$. Define
\begin{align}
  \label{eq:molify}
  M_1^1(\xi) = \frac{\sqrt{2}g_\epsilon*\chi_{\B_{1,1}}(\xi)}{\left(\left|g_\epsilon*\chi_{\B_{1,1}}(\xi)\right|^2+ \left|g_\epsilon*\chi_{\B_{1,2}}(\xi)\right|^2\right)^{1/2}}, \quad M_2^1(\xi) = M_1^1(\xi+\gamma_1)e^{i\left<\xi , \eta_1\right>},
\end{align}
where $\eta_1 = (0,4/\sqrt{3})$, and $g_\epsilon\in C_c^\infty(\R^2)$ is a smooth function supported on a Euclidean ball with radius $\epsilon$. One can easily check that $(M^1_k,M^1_k,\Gamma_1^2\to\Gamma_k^4)_{k=1,2}$ is a PR filter bank, and $M^1_k\in C^\infty(\R^2)\cap L^2(\R^2/\Gamma_1^{2*})$ is supported mainly on $\B_{1,k}$. Thus $\widehat{\psi_{1}^{1,k}}(\xi)\coloneqq M_k^1(\xi)\widehat{\psi_{1}^1}(\xi)$ is supported primarily on $\B_{1,k}\cap A_1^2 = A^4_k$, i.e., the regions in color green in Figure~\ref{fig:cutting-basis-1} and~\ref{fig:cutting-basis-2}. Moreover, by Lemma~\ref{lemma:cutting_basis}, the $\Gamma^4_k$-translates of $\psi_{1}^{1,k}$,
\begin{align}
  \left\{\psi_{1}^{1,1}(\cdot-n)=\psi_{1,n}^{1,1} \right\}_{n\in \Gamma^4_1} \cup \left\{\psi_{1}^{1,2}(\cdot-n)=\psi_{1,n}^{1,2} \right\}_{n\in \Gamma^4_2},
\end{align}
constitute an orthonormal basis of $V_1 = \widebar{\Span}\left\{\psi_{1,n}^1 \right\}_{n\in \Gamma_1^{2}}$. Figure~\ref{fig:lvl1_ON1_idx1} and~\ref{fig:lvl1_ON1_idx2} show $\psi_{1}^{1,1}$ and  $\psi_{1}^{1,2}$ in the spatial and frequency domain obtained from cutting the basis function $\psi_{1}^{1}$. In order to cut $\psi_{1}^{1,1}$ further to obtain basis functions with Fourier transforms  supported mainly on $A_1^8$ and $A_2^8$ respectively (see Figure~\ref{fig:cutting-demo}), we can use the 2-band filters $M^2_1, M^2_2\in C^\infty(\R^2)\cap L^2(\R^2/\Gamma_1^{4*})$ constructed from $\chi_{\B_{2,1}}, \chi_{\B_{2,2}}$ (Figure~\ref{fig:cutting-basis-3} and~\ref{fig:cutting-basis-4}) after similar smoothing as in \eqref{eq:molify}.

\begin{figure}
    \centering
    \begin{subfigure}[t]{0.15\textwidth}
        \includegraphics[width=\textwidth]{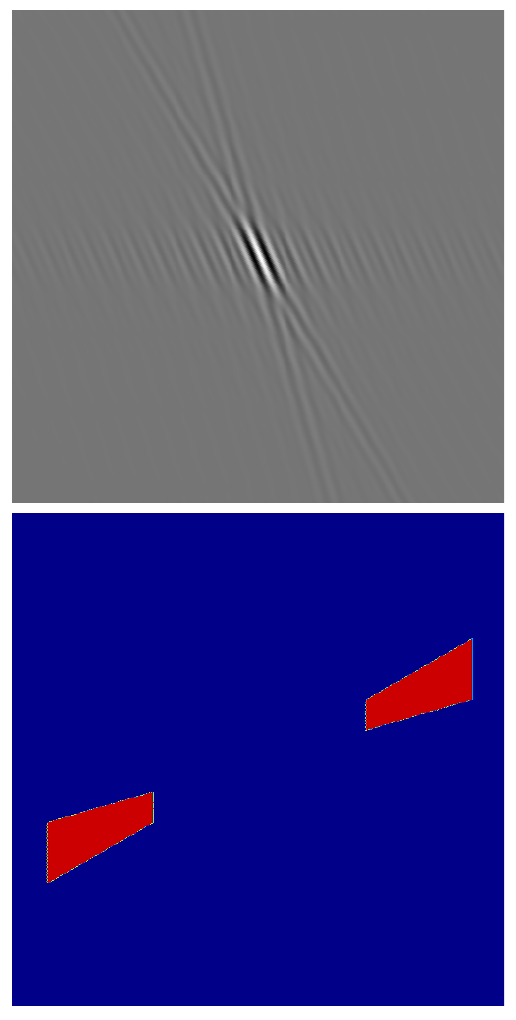}
        \caption{}
        \label{fig:lvl1_shannon_idx1}
    \end{subfigure}
     ~
    \begin{subfigure}[t]{0.15\textwidth}
        \includegraphics[width=\textwidth]{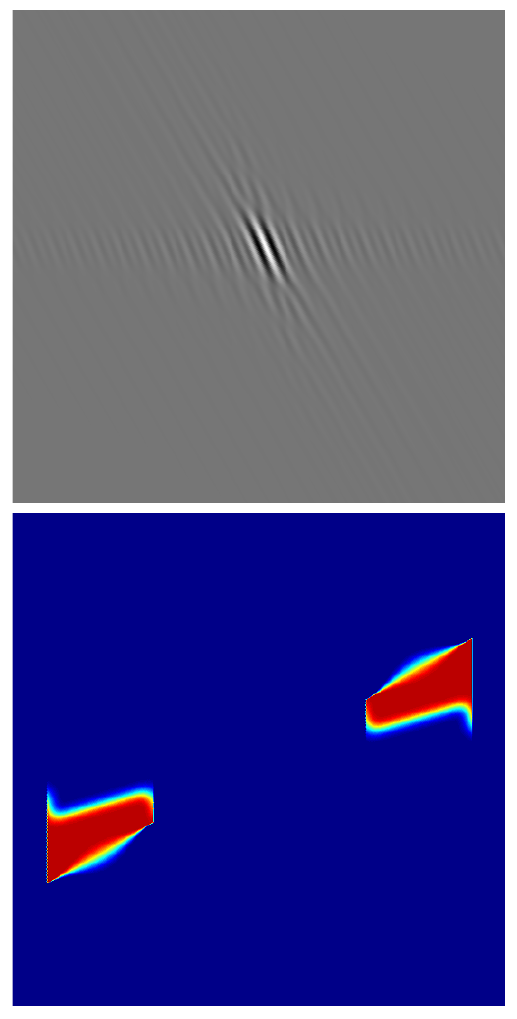}
        \caption{}
        \label{fig:lvl1_ON1_idx1}
    \end{subfigure}
     ~
    \begin{subfigure}[t]{0.15\textwidth}
        \includegraphics[width=\textwidth]{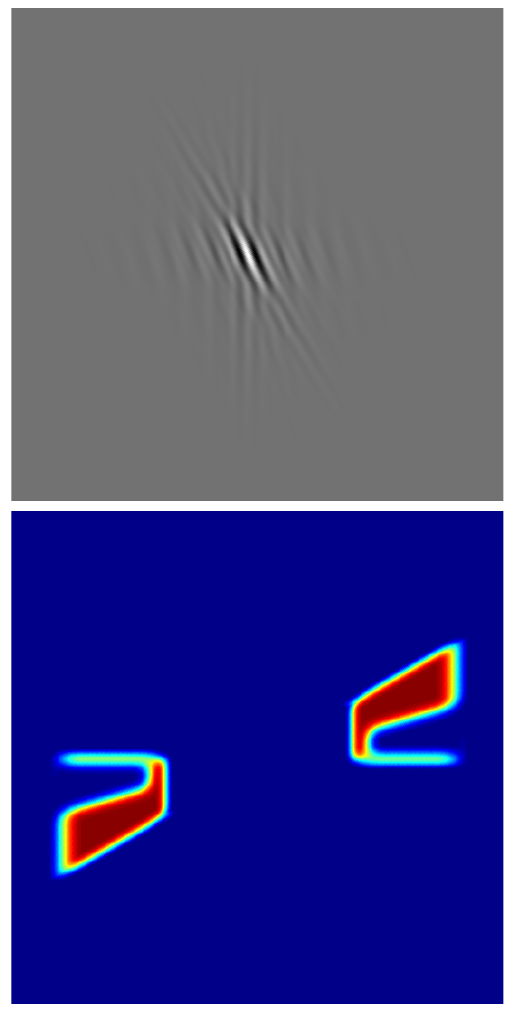}
        \caption{}
        \label{fig:lvl1_frame_idx1}
    \end{subfigure}
     ~
    \begin{subfigure}[t]{0.15\textwidth}
        \includegraphics[width=\textwidth]{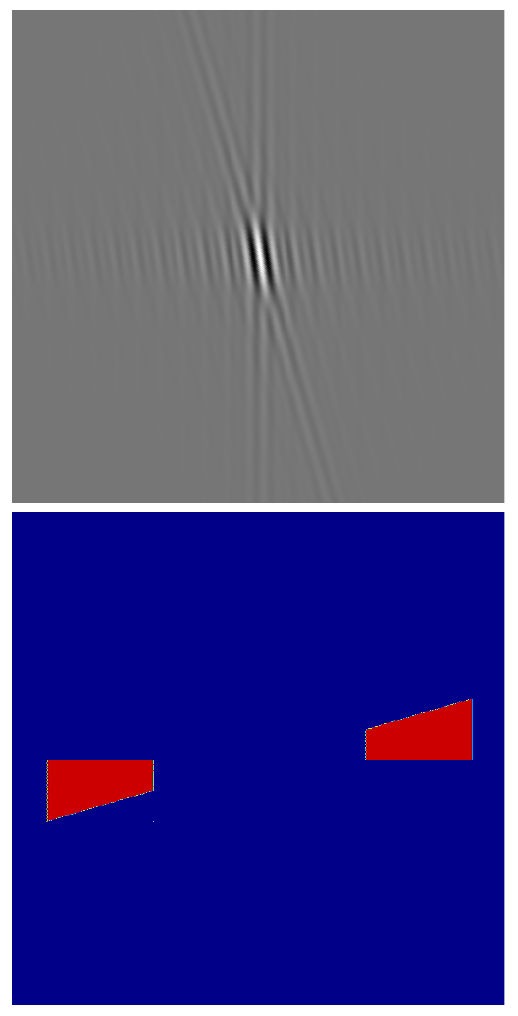}
        \caption{}
        \label{fig:lvl1_shannon_idx2}
    \end{subfigure}
     ~
    \begin{subfigure}[t]{0.15\textwidth}
        \includegraphics[width=\textwidth]{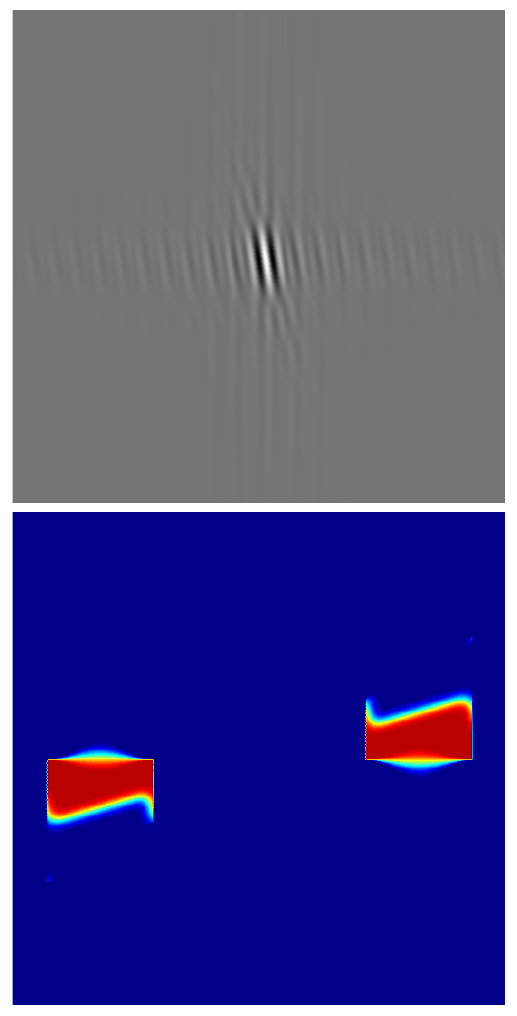}
        \caption{}
        \label{fig:lvl1_ON1_idx2}
    \end{subfigure}
     ~
    \begin{subfigure}[t]{0.15\textwidth}
        \includegraphics[width=\textwidth]{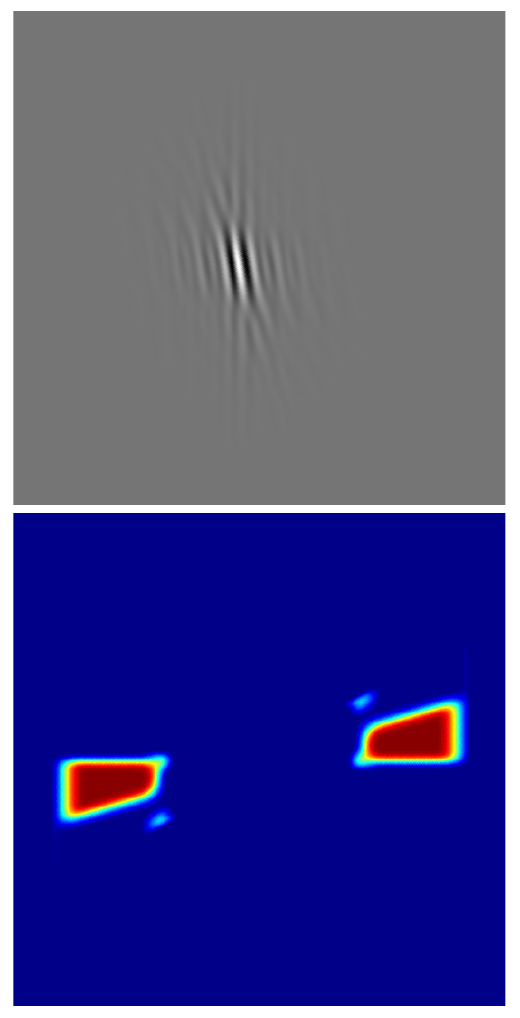}
        \caption{}
        \label{fig:lvl1_frame_idx2}
    \end{subfigure}
    \caption{ \small Wavelet functions $\psi_{1}^{1,1}$ and $\psi_{1}^{1,2}$ obtained from the cutting lemma. First row: $\psi_{1}^{1,k}$ in the time domain. Second row: $\widehat{\psi_{1}^{1,k}}$ in the frequency domain supported mainly on $A^4_{k}$. \textbf{(a)(d)}: Shannon wavelets. \textbf{(b)(e)}: $\psi_{1}^{1,1}$ and $\psi_{1}^{1,2}$ obtained from cutting the wavelet basis function $\psi^{1}_{1}$ constructed in Section~\ref{sec:basis}. \textbf{(c)(f)}: $\psi_{\text{fr},1}^{1,1}$ and $\psi_{\text{fr},1}^{1,2}$ obtained from cutting the wavelet frame function $\psi_{\text{fr},1}^{1}$ constructed in Section~\ref{sec:frame}.}
\end{figure}

In order to subdivide the wavelet frame functions $\widehat{\psi^k_{\text{fr}}}$ constructed in Section~\ref{sec:frame}, we need to modify the downsampling lattices and the corresponding 2-band filters . Given $p\ge 2$ and $1\le k\le 3p$, define
\begin{align}
\label{eq:hex_lattice_frame}
  \Gamma_k^{\text{fr},p} = \left\{
  \begin{aligned}
    &
    \begin{bmatrix}
      2(p-1) & 4 \\
      -\frac{2}{\sqrt{3}}(p-1) & 0
    \end{bmatrix}\Z^2, \quad &&1\le k \le p.\\
    &
    \begin{bmatrix}
      0 & -2\\
      \frac{4}{\sqrt{3}}(p-1) & 2\sqrt{3}
    \end{bmatrix}\Z^2, \quad &&p+1\le k \le 2p.\\
    &
    \begin{bmatrix}
      2(p-1) & 2\\
      \frac{2}{\sqrt{3}}(p-1) & 2\sqrt{3}
    \end{bmatrix}\Z^2, \quad &&2p+1 \le k\le 3p.\\
  \end{aligned}\right.
\end{align}
In particular, $\Gamma_k^{\text{fr},2}$ is the same as $\Gamma_k^{\text{fr}}$ defined in \eqref{eq:gamma_frame}, and $\left\{\psi_{\text{fr},j,n}^k \right\}_{1\le k\le 6, j\in \Z, n\in 2^{j-1}\Gamma_k^{\text{fr},2}}$ constitutes a Parseval frame of $L^2(\R^2)$. The wavelet function $\psi^1_{\text{fr},1}$, whose Fourier transform $\widehat{\psi^1_{\text{fr},1}}$ is supported on an $\epsilon$-neighborhood of $A_1^2$, can be cut by the 2-band filter bank $\left(M_{\text{fr},k}^1,M_{\text{fr},k}^1, \Gamma_1^{\text{fr},2}\to \Gamma_k^{\text{fr},4}\right)_{k=1,2}$ supported mainly on $\B_{1,1}^{\text{fr}}$ and $\B_{1,2}^{\text{fr}}$ (see Figure~\ref{fig:cutting-frame-1} and~\ref{fig:cutting-frame-2}). More specifically,
\begin{align}
  \label{eq:molify-frame}
  M_{\text{fr},1}^1(\xi) = \frac{\sqrt{2}g_\epsilon*\chi_{\B_{1,1}^{\text{fr}}}(\xi)}{\left(\left|g_\epsilon*\chi_{\B_{1,1}^{\text{fr}}}(\xi)\right|^2+ \left|g_\epsilon*\chi_{\B_{1,2}^{\text{fr}}}(\xi)\right|^2\right)^{1/2}}, \quad M_{\text{fr},2}^1(\xi) = M_{\text{fr},1}^1(\xi+\gamma_1^{\text{fr}})e^{i\left<\xi , \eta^{\text{fr}}_1\right>},
\end{align}
where $\gamma_1^{\text{fr}} = (0,\sqrt{3}\pi/2)$, and $\eta_1^{\text{fr}} = (2,2/\sqrt{3})$. By Lemma~\ref{lemma:cutting_frame},
\begin{align*}
  \left\{ \psi^{1,1}_{\text{fr},1,n}\right\}_{n\in \Gamma_1^{\text{fr},4}} \cup \left\{ \psi^{1,2}_{\text{fr},1,n}\right\}_{n\in \Gamma_2^{\text{fr},4}} \cup \left\{\psi_{\text{fr},j,n}^k \right\}_{(k,j)\neq (1,1), n\in 2^{j-1}\Gamma_k^{\text{fr},2}}
\end{align*}
also constitutes a Parseval frame of $L^2(\R^2)$, where $\widehat{\psi^{1,k}_{\text{fr},1}} = M_{\text{fr},k}^1\widehat{\psi^1_{\text{fr},1}}$, $k=1, 2$, are shown in the spatial and frequency domain in Figure~\ref{fig:lvl1_frame_idx1} and~\ref{fig:lvl1_frame_idx2}. Again,  $\psi^{1,1}_{\text{fr},1}$ can be further subdivided  using the smooth 2-band filters supported mainly on $\B_{2,1}^{\text{fr}}$ and $\B_{2,2}^{\text{fr}}$ in Figure~\ref{fig:cutting-frame-3} and~\ref{fig:cutting-frame-4}.

We have thus, in this section, provided an alternative way of constructing wavelet systems with finer orientation selectivity satisfying the parabolic scaling law. Compared to the techniques in Section~\ref{sec:basis_and_frame}, the cutting lemma does not increase the frame redundancy to achieve more orientations, but it inevitably introduces aliasing due to the filter-smoothing \eqref{eq:molify}\eqref{eq:molify-frame}.

\section{Numerical experiments}
\label{sec:ex}

To illustrate the practical utility of the multidirectional hexagonal wavelet bases and low-redundancy frames constructed in Section~\ref{sec:basis_and_frame} and~\ref{sec:cutting}, we compare their performance to tensor wavelets  and  curvelets \cite{candes2006fast} in the task of high bit rate image compression.

In the experiment, the nonlinear approximation of the original image is obtained by keeping $N/20$ largest coefficients, where $N$ is the number of pixels in the original image, thus achieving a $20:1$ compression ratio. The tensor product of two compactly supported wavelets with 6 vanishing moments \cite{daubechies1988orthonormal} is used for the tensor wavelets, and the wrapping discrete curvelet transform\footnote{The online package is available at CurveLab \url{http://www.curvelet.org/}} is used for the experiments with curvelets. Three levels of decomposition (two levels with 6 directions, and another one with 12 directions) are conducted for the experiments with the proposed hexagonal wavelet bases and frames. Even though the input digital images are sampled on a square lattice, we treat them as sampled on a hexagonal lattice for the test with hexagonal wavelets, and thus the ``real'' input is the original image after a linear transformation. The peak signal-to-noise ratio (PSNR) defined by
\begin{align*}
  \text{PSNR} = 10\log_{10}\left( \frac{255^2N}{\|f-f_{c}\|_{2}^2}\right)
\end{align*}
is used to evaluate the quality of the compression, where $f,f_c$ are the original and the compressed image respectively.

\begin{figure}
  \centering
    \begin{subfigure}[t]{0.3\textwidth}
        \includegraphics[width=\textwidth]{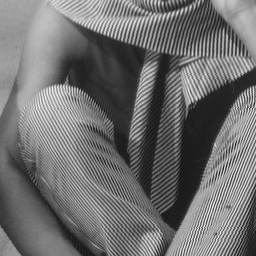}
        \caption{\small Original image.}
        \label{fig:barbara_original}
    \end{subfigure}
    ~
    \begin{subfigure}[t]{0.3\textwidth}
        \includegraphics[width=\textwidth]{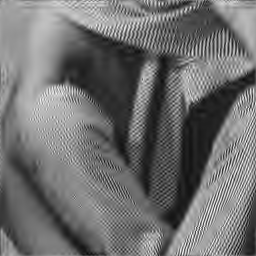}
        \caption{\small Tensor wavelets.}
        \label{fig:barbara_separable}
    \end{subfigure}
    ~
    \begin{subfigure}[t]{0.3\textwidth}
        \includegraphics[width=\textwidth]{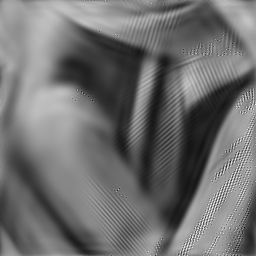}
        \caption{\small Curvelets.}
        \label{fig:barbara_curvelet}
    \end{subfigure}
    ~
    \begin{subfigure}[t]{0.3\textwidth}
        \includegraphics[width=\textwidth]{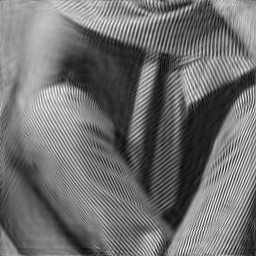}
        \caption{\small Hexagonal wavelet basis.}
        \label{fig:barbara_hex_basis}
    \end{subfigure}
    ~
    \begin{subfigure}[t]{0.3\textwidth}
        \includegraphics[width=\textwidth]{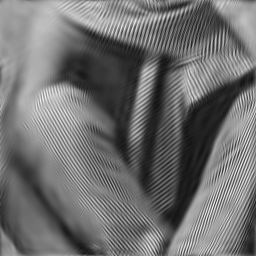}
        \caption{\small Hexagonal wavelet frame.}
        \label{fig:barbara_hex_frame}
    \end{subfigure}
  \caption{\small Compression results of the Barbara image with 20:1 compression ratio. \textbf{(a)} The original image. \textbf{(b)} Tensor wavelets (PSNR = 23.23 dB.) \textbf{(c)} Curvelets (PSNR = 21.14 dB.) \textbf{(d)} Proposed hexagonal wavelet basis (PSNR = 26.49 dB.) \textbf{(e)} Proposed hexagonal wavelet frame (PSNR = 24.59 dB.)}
  \label{fig:ex_barbara}
\end{figure}

\begin{figure}
  \centering
    \begin{subfigure}[t]{0.3\textwidth}
        \includegraphics[width=\textwidth]{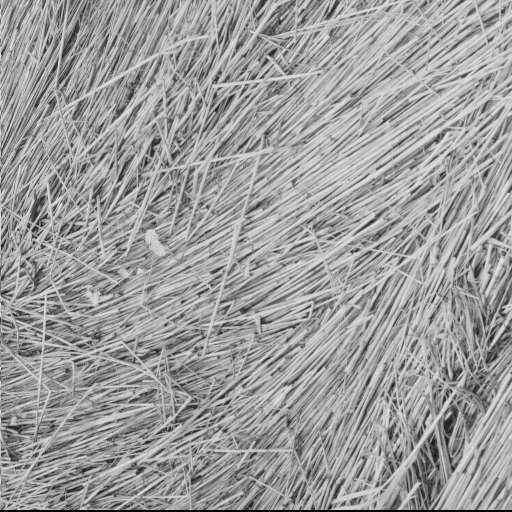}
        \caption{\small Original image.}
        \label{fig:straw_original}
    \end{subfigure}
    ~
    \begin{subfigure}[t]{0.3\textwidth}
        \includegraphics[width=\textwidth]{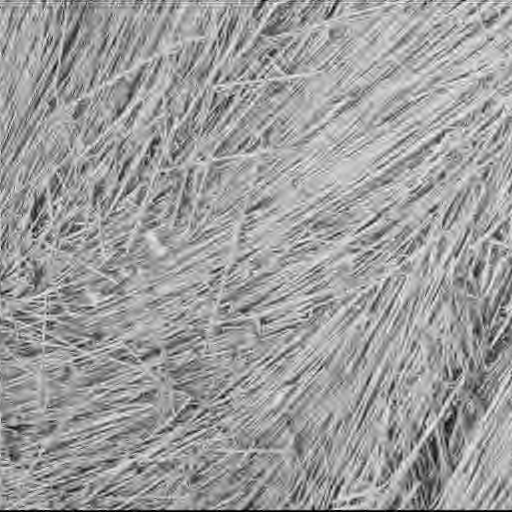}
        \caption{\small Tensor wavelets.}
        \label{fig:straw_separable}
    \end{subfigure}
    ~
    \begin{subfigure}[t]{0.3\textwidth}
        \includegraphics[width=\textwidth]{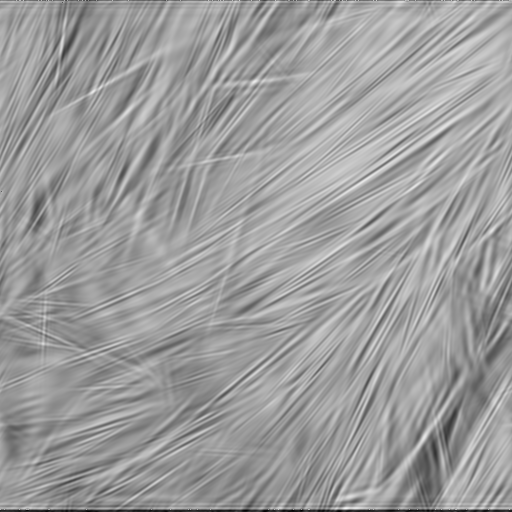}
        \caption{\small Curvelets.}
        \label{fig:straw_curvelet}
    \end{subfigure}
    ~
    \begin{subfigure}[t]{0.3\textwidth}
        \includegraphics[width=\textwidth]{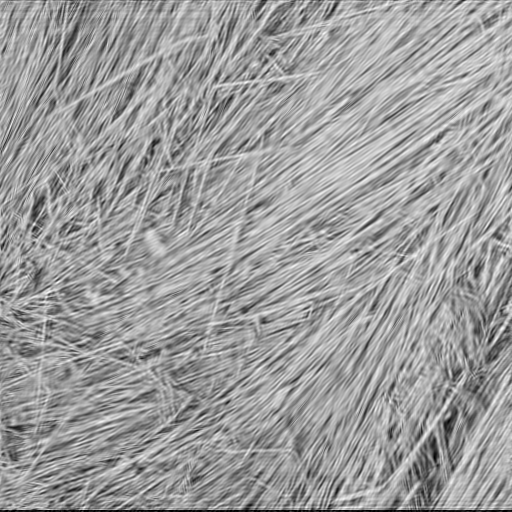}
        \caption{\small Hexagonal wavelet basis.}
        \label{fig:straw_hex_basis}
    \end{subfigure}
    ~
    \begin{subfigure}[t]{0.3\textwidth}
        \includegraphics[width=\textwidth]{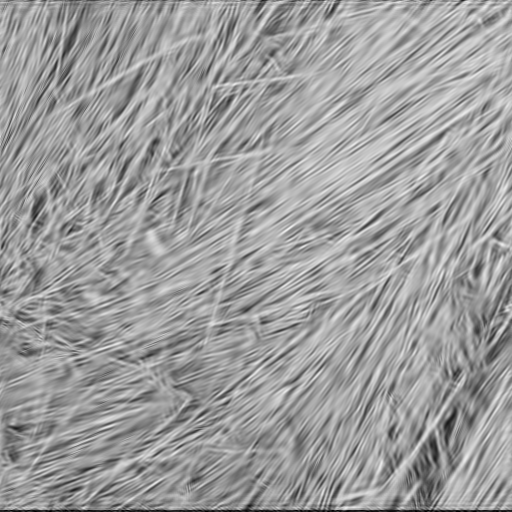}
        \caption{\small Hexagonal wavelet frame.}
        \label{fig:straw_hex_frame}
    \end{subfigure}
  \caption{\small Compression results of the straw image with 20:1 compression ratio. \textbf{(a)} The original image. \textbf{(b)} Tensor wavelets (PSNR = 21.79 dB.) \textbf{(c)} Curvelets (PSNR = 19.56 dB.) \textbf{(d)} Proposed hexagonal wavelet basis (PSNR = 23.31 dB.) \textbf{(e)} Proposed hexagonal wavelet frame (PSNR = 21.86 dB.)}
  \label{fig:ex_straw}
\end{figure}

Figure~\ref{fig:ex_barbara} and~\ref{fig:ex_straw} display the compression results for the Barbara and straw image. It is evident that the proposed multidirectional hexagonal wavelet systems are more efficient than tensor wavelets near edges and textures (notice that the proposed hexagonal frames achieve better compression results compared to tensor wavelets even though it is more redundant.) The over-complete curvelet transform, albeit having provably optimal approximation rate for piecewise smooth functions with jumps along curves, fails to deliver its potential in high bit rate image compression for reasonable-size images. We have not included the most recent implementation of the shearlet transform\footnote{The online package is available at ShearLab \url{http://www.shearlab.org/}} \cite{kutyniok2014shearlab} for it fails to achieve reasonable results with its extremely high redundancy (49 with 3 levels of decomposition.) 

\section{Conclusion}
\label{sec:conclusion}
We gave a detailed study of the restriction in building multidirectional wavelet bases corresponding to an admissible frequency partition in the framework of non-uniform directional filter bank. In particular, we constructed orientation-invariant and alias-free hexagonal wavelet orthonormal bases with optimal continuity in the frequency domain as well as low-redundancy frames with arbitrarily fast spatial decay. A 2D cutting lemma is used to subdivide the obtained wavelet systems in higher frequency rings to achieve the parabolic scaling law without increasing the redundancy.  Further applications of such wavelet systems, e.g., building local-orientation-invariant convolutional neural network in machine learning, will be studied in future work.

\appendix
\section{Proof of Proposition~\ref{prop:crit_PR}}
\label{apdx:proof_crit_PR}
  Since $(M_k, \widetilde{M}_k, \Lambda \to \Gamma_k)_{k=1}^K$ is PR and critically sampled, the matricies $M(\xi)$ and $\widetilde{M}(\xi)$ are square matrices, and \eqref{eq:PR} implies
  \begin{align*}
    \widetilde{M}(\xi)M(\xi)^*  = |\Lambda/\Gamma|Id_{|\Lambda/\Gamma|}, \quad a.e.~\xi \in \R^d.
  \end{align*}
Therefore, for $a.e. ~\xi \in \R^d$ and $\forall k\in [K]$, we have
\begin{align*}
  \sum_{\gamma\in\Gamma^*/\Lambda^*}e^{\left<\eta,\xi+\gamma\right>}\widetilde{M}_k(\xi+\gamma)\overline{M_k}(\xi+\gamma) = |\Lambda/\Gamma|\delta_{\eta,0}, ~\forall \eta \in \Gamma_k/\Gamma.
\end{align*}
This implies
\begin{align*}
  |\Lambda/\Gamma| \delta_{\eta,0} & = \sum_{s\in \Gamma^*/\Gamma_k^*}\sum_{t\in\Gamma_k^*/\Lambda^*}e^{i\left<\eta,s+t\right>}\widetilde{M}_k(\xi+s+t)\overline{M_k}(\xi+s+t)\\
  & = \sum_{s\in\Gamma^*/\Gamma_k^*}e^{i\left<\eta,s\right>}\sum_{t\in\Gamma_k^*/\Lambda^*}\widetilde{M}_k(\xi+s+t)\overline{M_k}(\xi+s+t)\\
  & = \sum_{s\in\Gamma^*/\Gamma_k^*}e^{i\left<\eta,s\right>}g(\xi,s), \quad \forall \eta \in \Gamma_k/\Gamma,
\end{align*}
where $g(\xi,s)\coloneqq \sum_{t\in\Gamma_k^*/\Lambda^*}\widetilde{M}_k(\xi+s+t)\overline{M_k}(\xi+s+t)$. Summing over all $\eta \in \Gamma_k/\Gamma$, we have
\begin{align*}
  |\Lambda/\Gamma| & = \sum_{s\in\Gamma^*/\Gamma_k^*}g(\xi,s)\sum_{\eta\in\Gamma_k/\Gamma}e^{i\left<\eta,s\right>} = \sum_{s\in\Gamma^*/\Gamma_k^*}g(\xi,s)\delta_{s,0}|\Gamma_k/\Gamma|\\
  & = |\Gamma_k/\Gamma| g(\xi,0) = |\Gamma_k/\Gamma|\sum_{t\in\Gamma_k^*/\Lambda^*}\widetilde{M}_k(\xi+t)\overline{M_k}(\xi+t).
\end{align*}
Hence we have
\begin{align*}
\sum_{t \in \Gamma_k^*/\Lambda^*}\widetilde{M}_k(\xi+t)\overline{M_k}(\xi+t) = |\Lambda/\Gamma|/ |\Gamma_k/\Gamma| = |\Lambda/\Gamma_k|.  
\end{align*}

\section{Proof of Proposition~\ref{prop:PR_hex}}
\label{apdx:prop_PR_hex}
According to Theorem~\ref{thm:PR}, it suffices to show \eqref{eq:PR_hex_identity_sum} and \eqref{eq:PR_hex_shift_cancel} hold if and only if
  \begin{align}
    \label{eq:PR_equation}
    \sum_{k=0}^6S_{k,\gamma} \coloneqq  \sum_{k=0}^6\sum_{\eta_k\in\Gamma_k/\Gamma}e^{-i\left<\eta_k,\gamma\right>}\overline{M_k(\xi+\gamma)}M_k(\xi)  = |\Lambda/\Gamma|\delta_{\gamma,0}, \quad \forall \gamma\in \Gamma^*/\Lambda^*.
  \end{align}
We  verify this for $\gamma$ in each segment of the Venn diagram in Fig~\ref{fig:venn}.
\begin{enumerate}
\item If $\gamma\in\Lambda^*$, i.e., $\gamma = 0\in\Gamma^*/\Lambda^*$, we have
  \begin{align*}
    \sum_{k=0}^6\sum_{\eta_k\in\Gamma_k/\Gamma}\overline{M_k(\xi)}M_k(\xi) = |\Lambda/\Gamma| \iff \sum_{k=0}^6|m_k(\xi)|^2 = \sum_{k=0}^6\frac{1}{|\Lambda/\Gamma_k|}|M_k(\xi)|^2 = 1.
  \end{align*}
\item If $\gamma\in\Gamma_0^*\setminus\Lambda^*$, then $\left<\gamma,\eta_k\right>\in 2\pi\Z, ~\forall \eta_k\in\Gamma_k\subset \Gamma_0$. Therefore,
  \begin{align*}
    &0 = \sum_{k=0}^6\sum_{\eta_k\in\Gamma_k/\Gamma}e^{-i\left<\eta_k,\gamma\right>}\overline{M_k(\xi+\gamma)}M_k(\xi)= \sum_{k=0}^6|\Gamma_k/\Gamma|\overline{M_k(\xi+\gamma)}M_k(\xi)\\
    \iff &\sum_{k=0}^6\overline{m_k(\xi+\gamma)}m_k(\xi) = 0.
  \end{align*}
\item If $\gamma\in\Gamma_{1,2}^*\setminus \Gamma_0^*$, then each individual $S_{k,\gamma}$ can be computed as follows
  \begin{itemize}
  \item For $k=0$: $\gamma \in \Gamma^*\setminus \Gamma_0^* \implies \sum_{\eta_0\in\Gamma_0/\Gamma}e^{-i\left<\eta_0,\gamma\right>} = 0 \implies S_{0,\gamma} = 0.$
  \item For $k \in \{3,4,5,6\}$: $\gamma \in\Gamma^*\setminus\Gamma_k \implies \sum_{\eta_k\in\Gamma_k/\Gamma}e^{-i\left<\eta_k,\gamma\right>} = 0 \implies S_{k,\gamma} = 0$.
  \item For $k \in \{1,2\}$: $\gamma \in \Gamma_k^*  \implies \sum_{\eta_k\in\Gamma_k/\Gamma}e^{-i\left<\eta_k,\gamma\right>} = |\Gamma_k/\Gamma|\implies S_{k,\gamma} = |\Gamma_k/\Gamma|\overline{M_k(\xi+\gamma)}M_k(\xi).$
  \end{itemize}
Hence we have
\begin{align*}
  0=\sum_{k=0}^6S_{k,r} = \sum_{k=1}^2|\Gamma_k/\Gamma|\overline{M_k(\xi+\gamma)}M_k(\xi) \iff \sum_{k=1}^2\overline{m_k(\xi+\gamma)}m_k(\xi) = 0.
\end{align*}
\item The last two equations of \eqref{eq:PR_hex_shift_cancel} can be proved similarly as the previous case.
\end{enumerate}

\bibliographystyle{abbrv}
\bibliography{ref}

\end{document}